\documentclass{amsart}
\usepackage{tikz-cd, tikz}
\usetikzlibrary{matrix, arrows}
\usepackage{mathrsfs} 
\usepackage{amssymb}
\usepackage{hyperref}
\usepackage{booktabs} 
\usepackage[normalem]{ulem} 
\usepackage{bbm}
\usepackage{enumitem}

\usepackage[mathscr]{eucal} 

\usetikzlibrary{decorations.markings}
\usetikzlibrary{positioning,cd,arrows, patterns}
\usetikzlibrary{arrows.meta}

\theoremstyle{plain}
\newtheorem{proposition}[equation]{Proposition}

\newtheorem{theorem}[equation]{Theorem}

\newtheorem{lemma}[equation]{Lemma}
\newtheorem{conjecture}[equation]{Conjecture}

\newtheorem{corollary}[equation]{Corollary}

\theoremstyle{definition}
\newtheorem{definition}[equation]{Definition}

\theoremstyle{remark}

\newtheorem{remark}[equation]{Remark}

\newtheorem{example}[equation]{Example}

\numberwithin{equation}{section}

\DeclareMathOperator{\coh}{\mathrm{H}}
\DeclareMathOperator{\Aut}{\mathrm{Aut}}

\DeclareMathOperator{\Ad}{\mathrm{Ad}}
\DeclareMathOperator{\ad}{\mathrm{ad}}
\DeclareMathOperator{\rank}{\mathrm{rk}}
\DeclareMathOperator{\Fl}{\mathbf{Fl}}
\DeclareMathOperator{\PFl}{\mathbf{PFl}}
\DeclareMathOperator{\Aff}{\mathbf{Aff}}
\DeclareMathOperator{\Hom}{\mathrm{Hom}}
\DeclareMathOperator{\Ext}{\mathrm{Ext}}
\DeclareMathOperator{\End}{\mathrm{End}}
\DeclareMathOperator{\Isom}{\mathrm{Isom}}
\DeclareMathOperator{\Sym}{\mathrm{Sym}}

\DeclareMathOperator{\Proj}{\mathrm{Proj}}
\DeclareMathOperator{\Sch}{\mathbf{Sch}}
\DeclareMathOperator{\Grpd}{\mathbf{Grpd}}

\DeclareMathOperator{\cHom}{\mathcal{H}\mathit{om}}

\DeclareMathOperator{\Gr}{\mathbf{Gr}}
\DeclareMathOperator{\height}{\mathrm{ht}}

\newcommand\tang{T}
\newcommand\PGL{\mathbf{PGL}}
\newcommand\GL{\mathbf{GL}}
\newcommand\SL{\mathbf{SL}}

\newcommand\pgl{\mathfrak{pgl}}
\newcommand\rs{{\mathrm{rs}}}
\newcommand\grs{\mathfrak{g}^{\mathrm{rs}}}

\newcommand\lra{\longrightarrow}
\newcommand\PP{\mathbb{P}}

\newcommand{\beq}{\begin{equation}}
\newcommand{\eeq}{\end{equation}}

\newcommand{\id}{\mathrm{id}}

\newcommand{\cB}{\mathcal{B}}

\newcommand{\cE}{\mathcal{E}}
\newcommand{\cF}{\mathcal{F}}
\newcommand{\cG}{\mathcal{G}}

\newcommand{\cL}{\mathcal{L}}

\newcommand{\cO}{\mathcal{O}}
\newcommand{\cP}{\mathcal{P}}
\newcommand{\cQ}{\mathcal{Q}}

\newcommand{\cV}{\mathcal{V}}
\newcommand{\cW}{\mathcal{W}}
\newcommand{\cX}{\mathcal{X}}

\newcommand{\fg}{\mathfrak{g}}
\newcommand{\ft}{\mathfrak{t}}
\newcommand{\tcG}{\widetilde{\cG}}

\newcommand\tgrs{\PP(\fg\oplus\fg)^\rs}

\newcommand{\fM}{\mathfrak{M}}

\newcommand{\fMFano}{\fM^{\mathrm{Fano}}}

\newcommand{\bk}{\Bbbk} 

\usepackage[utf8]{inputenc}

\begin{document}

\title[Deformations of Hessenberg varieties]{Automorphisms and deformations of regular semisimple Hessenberg varieties}
\author[Brosnan]{Patrick Brosnan}
\address{Department of Mathematics\\
  University of Maryland\\
  College Park, MD USA}
\email{pbrosnan@math.umd.edu}
\author[Escobar]{Laura Escobar}
\address{Department of Mathematics\\
University of California\\
Santa Cruz, CA 95064\\
USA}
\email{lauraescobar@ucsc.edu}
\author[Hong]{Jaehyun Hong}
\address{Center for Complex Geometry\\Institute for Basic Science\\Daejeon 34126\\Republic of Korea}
\email{jhhong00@ibs.re.kr}
\author[Lee]{Donggun Lee}
\address{Center for Complex Geometry\\Institute for Basic Science\\Daejeon 34126\\Republic of Korea}
\email{dglee@ibs.re.kr}
\author[Lee]{Eunjeong Lee}
\address{Department of Mathematics\\Chungbuk National University\\Cheongju 28644\\Republic of
Korea}
\email{eunjeong.lee@chungbuk.ac.kr}
\author[Mellit]{Anton Mellit}
\address{Faculty of Mathematics\\
University of Vienna\\
Oskar-Morgenstern-Platz 1, 1090 Vienna, Austria}
\email{anton.mellit@univie.ac.at}
\author[Sommers]{Eric Sommers}
\address{Department of Mathematics \& Statistics\\
University of Massachusetts\\
Lederle Graduate Research Tower\\
Amherst, MA 01003
}
\email{esommers@math.umass.edu}
\begin{abstract}
   We show that regular semisimple Hessenberg varieties can have moduli.
   To be precise, suppose $X$ is a regular semisimple Hessenberg variety of
   codimension $1$ in the flag variety $G/B$, where $G$ is a simple algebraic
   group of rank $r$ over $\mathbb{C}$ and $B$ is a Borel subgroup. 
   We show that the space~$\coh^1(X,TX)$ of first order deformations of
   $X$ has dimension $r-1$ except in type $A_2$. (In type $A_2$, the Hessenberg
   varieties in question are all isomorphic to the
   permutohedral toric surface, 
   and $\dim\coh^1(X,TX) = 0$.)
   Moreover, we show that the Kodaira--Spencer map $\mathfrak{g}\to \coh^1(X,TX)$ is onto,  
   that the identity component of the automorphism group of $X$ is a
   maximal torus of $G$, and that $\coh^i(X,TX) = 0$ for $i \geq 2$.
   Along the way, we prove several theorems of independent
   interest about the cohomology of homogeneous vector bundles on~$G/B$.

   In type $A$, we can give an even more precise statement determining
   when two codimension $1$ regular semisimple Hessenberg varieties in
   $G/B$ are isomorphic.
   We also compute the automorphism groups explicitly in type~$A_{n-1}$ in the terms
   of stabilizer subgroups of the action of the symmetric group $S_{n}$ on the
   moduli space $M_{0,n+1}$ of smooth genus $0$ curves with $n + 1$ marked
   points.
   Using this, we describe the moduli stack of the regular semisimple Hessenberg
   varieties $X$ explicitly as a quotient stack of $M_{0,n+1}$.

   We prove several analogous results for Hessenberg varieties in generalized flag varieties $G/P$, 
   where $P$ is a parabolic subgroup of $G$.
   In type $A$, these results are used in the proofs of the results for $G/B$, but they are
   also of independent interest because the associated moduli stacks are
   related directly to the action of $S_n$ on $M_{0,n}$.
   In arbitrary types, our results for $G/P$ require vanishing theorems for homogeneous vector 
   bundles on $G/P$, which, like the vanishing theorems we prove for $G/B$, may be of independent interest.
\end{abstract}

\maketitle
\tableofcontents

\section{Introduction}\label{s.intro}
Fix an algebraically closed field $\bk$ of characteristic $0$ and a reductive 
group $G$ over $\bk$ whose adjoint group is simple.
Hessenberg varieties are certain subvarieties $\mathcal{B}(M,s)$ of the variety
$\mathcal{B}$ of Borel subgroups of $G$, associated to a subset~$M$ of the
root system of~$G$ called a \emph{Hessenberg subset} 
and an element $s$ in the Lie algebra $\mathfrak{g}$ of $G$.
When $s$ lies in the Zariski open subset $\grs$ of $\mathfrak{g}$ consisting of
regular semisimple elements (those whose centralizer $Z_G(s)$ is a
maximal torus), the variety $\mathcal{B}(M,s)$ is smooth and preserved by the natural action
of the maximal torus $T(s)=Z_G(s)$.
By abuse of terminology, $\mathcal{B}(M,s)$ is called a \emph{regular semisimple}
Hessenberg variety when $s$ is regular semisimple.

In several respects, regular semisimple Hessenberg varieties behave analogously
to the generalized flag varieties that they live in. 
For example, we have $\mathcal{B}(M,s)^{T(s)}=\mathcal{B}^{T(s)}$.
It therefore follows from the theorem of Bia{\l}ynicki-Birula that, similarly
to $\mathcal{B}$
itself, 
for $s\in\grs$, $\mathcal{B}(M,s)$ is cellular,
with cells in bijection with the elements of the Weyl group $W=N_{G}(T(s))/T(s)$ of $G$~\cite{Bia73}. 
We refer to the foundational work of de Mari, Procesi and Shayman~\cite{demari-procesi-shayman} 
for this and several other important facts about the Hessenberg varieties
$\mathcal{B}(M,s)$.

When the Hessenberg subset $M$ is adapted to a conjugacy class $\mathcal{P}$
of parabolic subgroups of $G$,
the story above can be generalized to the case of \emph{generalized Hessenberg
varieties} $\mathcal{P}(M,s)$, which are subvarieties of $\mathcal{P}$. 
The canonical morphism $\mathcal{B}\to\mathcal{P}$ sending a Borel subgroup $B$ to 
the unique subgroup $P\in\mathcal{P}$ containing it restricts
to a morphism $\mathcal{B}(M,s)\to\mathcal{P}(M,s)$.
In fact, we have 
\[\mathcal{B}(M,s)=\mathcal{P}(M,s)\times_{\mathcal{P}} \mathcal{B}.\]
Thus, $\mathcal{B}(M,s)$ is fibered over $\mathcal{P}(M,s)$, with fiber over
$B\in\mathcal{B}$ isomorphic to the homogeneous variety $P/B$.

In~\cite[Theorem 2]{demazure}, Demazure computed the automorphism group of 
$\mathcal{P}$ and proved that $\coh^i(\mathcal{P}, T\mathcal{P}) = 0$ for $i>0$, 
where $T\mathcal{P}$ denotes the tangent bundle of $\mathcal{P}$.
In particular, 
the case $i=1$
shows that
the deformation space of $\cP$
is trivial.

In this paper, we study the automorphism groups and deformations of  
regular semisimple Hessenberg varieties $\mathcal{B}(M,s)$ when they have 
codimension $1$ in $\mathcal{B}$.
In this case, they are connected smooth divisors associated with a semi-ample line bundle. 
For certain conjugacy classes~$\mathcal{P}$ of 
parabolic subgroups of $G$, 
the
generalized Hessenberg varieties $\mathcal{P}(M,s)$
are also well-defined inside
the corresponding generalized flag variety $\mathcal{P}$, and we consider some of
these spaces as well.

Roughly speaking, we find that  the 
following properties typically hold for the Hessenberg varieties $\mathcal{B}(M,s)$.
\begin{enumerate}
  \item Typically (outside of a few exceptional cases), they move in families of
    dimension $\rank G -1$, where $\rank G$ denotes the semisimple rank of $G$; 
  \item Typically the identity component $\Aut^\circ\mathcal{B}(M,s)$ 
    of $\Aut\mathcal{B}(M,s)$ is the maximal torus $ Z_{G^{\ad}}(s)$  in the
    adjoint group $G^{\ad}$ of $G$.
\end{enumerate}

In particular, we see that, unlike the flag varieties they live in, Hessenberg varieties
generally have nontrivial deformation spaces. Indeed, the variety 
$\mathcal{B}(M,s)$ can move in positive-dimensional families as the Lie algebra element $s$ varies.

Any element $g\in G$ gives rise to an isomorphism 
\begin{equation}
  \label{e.adjiso}
  \varphi(g):\mathcal{B}(M,s)\xrightarrow{\;\cong\;}\mathcal{B}(M, \Ad(g)s)
\end{equation}
induced by the natural action of $G$ on $\mathcal{B}$.
From this, it is easy to see that each variety $\mathcal{B}(M,s)$ is isomorphic
to one where $s$ is taken to be in a fixed Cartan subalgebra of the Lie algebra of the derived subgroup of \(G\). 
Since $\mathcal{B}(M,s)=\mathcal{B}(M,as)$ for $a\in \bk^{\times}$, this makes it
clear that the regular semisimple Hessenberg varieties $\mathcal{B}(M,s)$ are
determined by \emph{at most}  $\rank G - 1$ parameters.   

Similarly, the centralizer $Z_{G^{\ad}}(s)$ acts on $\mathcal{B}(M,s)$.  
Since the action is faithful (except in a few degenerate cases), 
this makes it clear that (except in these cases), for $s\in\grs$, the maximal
torus $Z_{G^{\ad}}(s)$ is contained in $\Aut^\circ\mathcal{B}(M,s)$. 

\subsection{Deformations in type \texorpdfstring{$A$}{A}}\label{ss-defa} 
In type $A$ (i.e., for $G=\GL_n$), our theorems are easy to state precisely
without too much notation.  
Our results are also stronger in type $A$ because, while we are able to
determine the dimensions of the automorphism group and the deformation space in
all types, in type $A$, we are able to determine exactly when two such Hessenberg
varieties are isomorphic.
Moreover, although the automorphism group itself depends on $s$, we are 
able to 
determine it precisely.
For this reason, we state our theorems in this case first.

Recall that, for $G=\GL_n$, the Lie algebra $\mathfrak{g}$ is equal
to the space of $n\times n$-matrices over $\bk$, and $\grs$ is the Zariski
open subset consisting of matrices with $n$ distinct eigenvalues.
The variety $\mathcal{B}$ of Borel subgroups  is simply the variety $\Fl_n$ of 
complete flags $F_{\bullet}$ in $\bk^n$ given by subspaces $F_0\subset F_1\subset \cdots \subset F_n$ 
with $\dim F_i = i$. 
Write $\PFl_n$ for the variety of partial flags $F_{\bullet}$ given by 
subspaces $F_1\subset F_{n-1}$ of $\bk^n$ with $\dim F_i =i$.
Equivalently, if $\mathbb{P}$ denotes the projective space of lines in $\bk^n$
and $\mathbb{P}^{\vee}$ denotes the dual projective space of hyperplanes in $\bk^n$,
then $\PFl_n$ is the variety consisting of pairs $(\ell,h)\in \mathbb{P}\times \mathbb{P}^{\vee}$ 
such that $\ell\subset h$.
We have an obvious morphism $\Fl_n\to\PFl_n$ given by forgetting the $F_i$ 
for $i\neq 1$, $n-1$.
The morphism is smooth with fiber the variety of complete flags in 
$F_{n-1}/F_1$.  

Now, for $s\in\mathfrak{g}$, set  
\begin{align*}
  X(s)&:= \{F_{\bullet}\in\Fl_n: sF_1\subset F_{n-1}\},\\
  Y(s)&:= \{F_{\bullet}\in\PFl_n: sF_1\subset F_{n-1}\}.
\end{align*}
Write $\Aff$ for the group of affine transformations of $\mathbb{A}^1$.
We can view $\Aff$ as the subgroup of $\GL_2$ 
given by matrices of the form
$\displaystyle\begin{pmatrix} a & b \\ 0 & 1\end{pmatrix}$. 

We first give our results about deformations.
\begin{theorem}\label{t.isoA}
Let $s,s'\in\grs$ and suppose $n\geq 4$. Then the following hold.
   \begin{enumerate}
     \item\label{t.isoAX} The varieties $X(s)$ and $X(s')$ are isomorphic if and only if there exist 
        $g\in G$ and  $\displaystyle\begin{pmatrix} a & b \\ 0 &
    1\end{pmatrix}\in\Aff$ such that
        \[
        gs'g^{-1}= as + b.
        \]
      \item\label{t.isoAY} The varieties $Y(s)$ and $Y(s')$ are isomorphic if and only if there exist
      $g\in G$ and $\displaystyle\begin{pmatrix} a & b \\ c &
    d\end{pmatrix}\in\GL_2$ such that 
     $cs+d$ is invertible and 
     \[
     gs'g^{-1} =
(as+b)(cs+d)^{-1}.
\]
\end{enumerate}
\end{theorem}
Here, we write $b$ for $b\cdot \id$, omitting the identity matrix from the notation.
By the discussion around \eqref{e.adjiso}, one readily sees that
$gs'g^{-1}=as+b$ implies $X(s)\cong X(s')$, since $X(s)=X(s+b)$ for $b\in \bk$.
The content of \eqref{t.isoAX} is that the converse also holds.

The converse is deduced from \eqref{t.isoAY}.
The main difficulties lie in
\begin{itemize}
	\item showing that any isomorphism of $X(s)$ descends to $Y(s)$ (Proposition~\ref{prop:unique.ext.X});
	\item determining which isomorphisms of $Y(s)$ admit lifts to $X(s)$ (Theorem~\ref{thm:XtoY}).
\end{itemize}
The statement~\eqref{t.isoAY} itself follows from essentially linear-algebraic arguments,
since $Y(s)$ can be realized as a complete intersection of two $(1,1)$-divisors
in $\PP\times \PP^\vee$.

These results show that any isomorphism of $X(s)$
is induced by an automorphism of the flag variety $\Fl_n$,
and similarly any isomorphism of $Y(s)$
is induced by an automorphism of $\PP\times\PP^\vee$.

The same approach also applies to the study of automorphisms.

\subsection{Automorphisms in type \texorpdfstring{$A$}{A}}\label{subsection_intro_1.2}
To state our results on automorphisms in type~$A$, first note that, 
for $n>2$, the nontrivial automorphism of the Dynkin diagram 
induces an outer automorphism of $G$. 
We can describe this outer automorphism explicitly and lift it to an involution 
$\iota$ of $G$, by picking a nondegenerate quadratic form $Q$ on
$\bk^n$. 
Then we can set $\iota_Q(g) = (g^{t})^{-1}$, where $g^t$ denotes the 
transpose of $g$ with respect to $Q$.
When the nondegnerate quadratic form $Q$ is fixed, we write $\iota$ for $\iota_Q$.
The involution $\iota$  of $G$ induces an involution, which we will also call
$\iota$, of $\mathcal{B}$, sending a Borel subgroup $B$ to $\iota(B)$.
We can describe it explicitly as follows: 
if $B$ is the Borel subgroup stabilizing a complete flag $F_{\bullet}$,
then  $\iota(B)$ is the Borel subgroup
stabilizing the complete flag $F'_{\bullet}$ 
given by $F'_i = F_{n-i}^{\perp}$. 
Similarly, we get an involution $\iota$ of $\PFl_n$ given by
$(\ell,h)\mapsto (h^{\perp},\ell^{\perp})$. 
Although $\iota$ depends on the choice of $Q$, any two such choices give conjugate involutions. In~fact, $\Aut \Fl_n = \PGL_n \rtimes \langle \iota \rangle$, where $\langle \iota \rangle$ is the cyclic group of order~$2$ generated by~$\iota$.

Before stating our main result, we record two preparatory lemmas.

\begin{lemma}\label{l.inv}
Let $n>1$ and  $s\in\fg$.
There exists a nondegenerate quadratic form~$Q$ such that 
$s^t=s$. 
Moreover, if $\iota = \iota_Q$ is the corresponding involution of $\mathcal{B}$,
then $\iota(X(s))=X(s)$ and $\iota(Y(s))=Y(s)$.
In particular, $\iota$ defines a nontrivial involution of both $X(s)$ 
and $Y(s)$. 
\end{lemma}

Define subsets $H(s)$ and $K(s)$ of $G$ by 
\begin{align*}
  H(s)&:=\left\{g\in G: 
    \text{there exists~}\begin{pmatrix} a & b \\ 0 & 1\end{pmatrix}
    \in\Aff~\text{such that }
    gsg^{-1} = as+b\right\};\\
  K(s)&:=\bigg\{g\in G: 
    \text{there exists~}\begin{pmatrix}a & b\\ c & d \end{pmatrix}
                    \in\GL_2 
                    ~\text{such that }  \\
    & \qquad \qquad \qquad \qquad \quad cs+d \text{ is invertible \;and }\; 
                  gsg^{-1} = (as+b)(cs+d)^{-1} \bigg\}.
\end{align*}

\begin{lemma}\label{l.HandK}
  Let $s\in\grs$ and let $T(s)=Z_G(s)$ be the maximal torus of $G$. 
  Then $H(s)$ and $K(s)$ are 
  subgroups of $N_{G}(T(s))$ containing $T(s)$. 
  Moreover, $\iota(H(s))=H(s)$ and $\iota (K(s))= K(s)$ for the involution $\iota$ in Lemma~\ref{l.inv}.
\end{lemma}

The proofs of Lemmas~\ref{l.inv} and~\ref{l.HandK} are given in
Subsections~\ref{proof_Lem_l.inv} and~\ref{ss.proof.isoAX}, respectively.
We now state our main theorem on automorphisms in type~$A$.

\begin{theorem}\label{t.autsA}
  Let $G=\GL_n$, let $\mathbb{G}_m$ denote its center, and let $s\in\mathfrak{g}^{\mathrm{rs}}$ with $n\geq 4$. 
  Then the following statements hold:
  \begin{enumerate}
    \item\label{t.autsAX} $\Aut X(s) = (H(s)/\mathbb{G}_m)\rtimes \langle\iota\rangle$, 
    \item\label{t.autsAY} $\Aut Y(s) = (K(s)/{\mathbb{G}_m})\rtimes \langle\iota\rangle$.  
  \end{enumerate}
\end{theorem}
By Lemma~\ref{l.HandK}, the above theorem can be restated as follows:
\begin{itemize}
	\item The identity components of the automorphism groups are
	\[\Aut^\circ X(s)=\Aut^\circ Y(s)=T(s)/\mathbb{G}_m.\]
	\item The groups of connected components are 
	\[\begin{split}
	&\pi_0\Aut X(s)=\bar{H}(s)\times \langle \iota\rangle,\\
	&\pi_0\Aut Y(s)=\bar{K}(s)\times \langle \iota\rangle,
\end{split}\]
where $\bar{H}(s):=H(s)/T(s)$ and $\bar{K}(s):=K(s)/T(s)$ are subgroups of the Weyl group $W=N_{G}(T(s))/T(s)\cong S_n$, the symmetric group.
\end{itemize}
Thus $\Aut X(s)$ and $\Aut Y(s)$ are completely determined by the subgroups $\bar{H}(s)$ and $\bar{K}(s)$ of $S_n$.
In the next subsection, we describe these subgroups in detail, and relate them to the moduli space $M_{0,n}$ of genus $0$ curves with $n$ marked points.

\medskip
Finally, in Subsections~\ref{ss:group.actions} and~\ref{ss:group.actions.Y}, we place the above results into a uniform framework of group actions.
More precisely, we introduce the groups
\[
\begin{aligned}
	&\cG := \Aut G \times \Aff \;\cong\; \Aut \Fl_n \times \Aff,
	&&\text{acting on } \fg,\\
	&\tcG := \Aut \PP\times \PP^\vee \times \PGL_2,
	&&\text{acting on } \PP(\fg\oplus\fg),
\end{aligned}
\]
so that, in both cases, isomorphism classes correspond to orbits and
automorphism groups correspond to stabilizers.
See Corollaries~\ref{cor:interpret.AX1} and~\ref{cor:interpret.AY1}.

Moreover, there are equivalent interpretations in terms of 
divisors and pencils.
For $X(s)$, the relevant action is the $\Aut \Fl_n$-action on the space of
Hessenberg divisors in $\Fl_n$, while for $Y(s)$ it is the
$\Aut \PP\times \PP^\vee$-action on the space of pencils of $(1,1)$-divisors
in $\PP\times \PP^\vee$.
In both settings, isomorphism classes are identified with orbits and
automorphism groups with stabilizers; see
Corollaries~\ref{cor:interpret.AX} and~\ref{cor:interpret.AY}.

\medskip

\subsection{Relation to \texorpdfstring{$M_{0,n}$}{M0n}}\label{ss:M0n.intro} 

Let $M_{0,n}$ denote the moduli space of genus $0$ curves with $n$ distinct 
marked points.  By definition, it is isomorphic to the quotient variety 
\[M_{0,n}\cong ((\PP^1)^n-\Delta)/\PGL_2, \quad n\geq3\]
by the (free) diagonal action of $\Aut \PP^1\cong \PGL_2$, where $\Delta \subset (\PP^1)^n$ denotes the union of the pairwise diagonals,  that is, 
$\Delta = \bigcup_{i<j}\{(x_1,\ldots,x_n)\in(\PP^1)^n \mid x_i = x_j\}$.
It has a natural action of $S_n$ by permuting the $n$ points.

Let $s\in\grs$, and write 
$\lambda_1,\ldots, \lambda_n$ for the necessarily distinct eigenvalues of $s$.
Associated to the ordered $n$-tuple $(\lambda_1,\ldots, \lambda_n)$, we get a point
$(\lambda_1,\ldots, \lambda_n, \infty)\in M_{0,n+1}$ by adding an extra point in $\PP^1$ at the infinity. 
So, taking the quotient by the $S_n$-action permuting 
its first $n$ coordinates, 
we get a point 
$\phi_X(s)\in M_{0,n+1}/S_n$.
In this language, Theorem~\ref{t.isoA}(\ref{t.isoAX}) says that, for
$s,s'\in\grs$,  
\beq \label{1}X(s)\cong X(s') \iff \phi_X(s)=\phi_X(s').\eeq 

Omitting the extra point at infinity,
an ordering of the eigenvalues of $s\in\grs$  
gives us a point 
$(\lambda_1,\ldots, \lambda_n)\in M_{0,n}$. 
Forgetting the ordering gives us a point $\phi_Y(s)\in M_{0,n}/S_n$.
Theorem~\ref{t.isoA}(\ref{t.isoAY}) shows that 
\[Y(s)\cong Y(s') \iff  
\phi_Y(s)=\phi_Y(s').\] 

In the same vein,
Theorem~\ref{t.autsA} admits the following interpretation in terms of the
stabilizers in $S_n$ of the points $(\lambda_1,\ldots, \lambda_n,\infty)\in
M_{0,n+1}$ and $(\lambda_1,\ldots, \lambda_n)\in M_{0,n}$.  

\begin{theorem} \label{t.HKStab}
Let $s\in \grs$ with eigenvalues $\lambda_1,\ldots,\lambda_n$,
and identify $S_n$ 
with the 
subgroup of permutations in $G$. 
\begin{enumerate}
\item 
$\bar{H}(s)$
is conjugate in $S_n$ to the 
stabilizer of $(\lambda_1,\ldots,\lambda_n,\infty) \in M_{0,n+1}$;
\item 
$\bar{K}(s)$ is conjugate in $S_n$ to
  the stabilizer of $(\lambda_1,\ldots, \lambda_n) \in M_{0,n}$.
\end{enumerate}
\end{theorem}

There are simple classifications of such subgroups of $S_n$, see Corollary~\ref{cor:classification.H(s)} and Examples~\ref{e.Zn action} and~\ref{e.Dn action}.
 
An important consequence of Theorem~\ref{t.HKStab}(1) is that, for $n\geq 4$, Tymoczko's dot 
action on the cohomology $\coh^*(X(s))$ of $X(s)$, defined in~\cite{tymoczko} via GKM theory \cite{GKM}, does not lift to an action on the variety $X(s)$ itself. See Corollary~\ref{no-dot.body}.  
\begin{corollary}[No geometric lift of Tymoczko's dot action]
\label{no-dot}
Let $n\geq 4$ and $s\in \grs$. 
Any $S_n$-representation on $\coh^*(X(s))$ that is induced by an $S_n$-action on $X(s)$
	is a direct sum of copies of the trivial representation and the sign representation.
	
	In particular, there is no $S_n$-action on $X(s)$ that induces Tymoczko's dot action. 
\end{corollary}

\subsection{Moduli stacks and compactifications}
Building on the interpretation of Theorems~\ref{t.isoA} and~\ref{t.autsA} in terms
of group actions, we give a moduli-theoretic description of regular semisimple
Hessenberg varieties.
More precisely, the description of isomorphism classes and automorphism groups of
$X(s)$ in terms of the $\cG$-action on $\grs$ naturally suggests that the moduli stack
$\fM_n$ of regular semisimple Hessenberg varieties should be given by the quotient
stack $[\grs/\cG]$.
In Section~\ref{s.moduli}, we prove that this expectation is indeed correct, by
realizing $\fM_n$ as an open substack of the moduli stack of smooth Fano varieties.

We further identify $\fM_n$ with a quotient stack of $M_{0,n+1}$.
By restricting the $\cG$-action to a subgroup preserving a Cartan subalgebra of
$\fg$, and using the freeness of the $\Aff$-action, we obtain an equivalent
presentation of $[\grs/\cG]$ as a quotient of $M_{0,n+1}$ by an explicit extension of a maximal torus $T/\mathbb{G}_m\subset G/\mathbb{G}_m$
by the symmetric group $S_n$ and the involution $\mu_2=\langle\iota\rangle$.
This provides a concrete description of $\fM_n$ in terms of $M_{0,n+1}$. See Theorems~\ref{thm:moduli.stack} and~\ref{thm:moduli.stack.open}.

\begin{theorem}\label{tmsi}
    The moduli stack $\fM_n$  
    is represented by the quotient stack
    \[
    [\grs/\cG] \;\cong \;
   \bigl [M_{0,n+1}\big/\bigl((T/\mathbb{G}_m)\rtimes(S_n\times \mu_2)\bigr)\bigr].\]
    In particular, $\fM_n$ admits the good moduli space $M_{0,n+1}/S_n$.
\end{theorem}

Theorem~\ref{thm:moduli.stack.Y} gives a parallel description of the moduli stack $\fM_n^Y$ of generalized Hessenberg varieties $Y(s)$ with $s\in\grs$, realizing it as a quotient stack of $M_{0,n}$ by the same group, with good moduli space $M_{0,n}/S_n$.

Both moduli spaces admit natural GIT compactifications, namely the GIT moduli spaces of divisors in $\Fl_n$ and pencils of divisors in $\PP\times \PP^\vee$. See Subsection~\ref{ss:compact}.

\medskip
Since $X(s)$ and $Y(s)$ are Fano varieties (see Lemma~\ref{lem:XFano}), there may be another natural compactification provided by the theory of K-stability \cite{xu-book,xu-survery}.	

It is now known that for several classes of Fano hypersurfaces, complete intersections, and log Fano pairs, the corresponding K-moduli spaces coincide with classical GIT quotients.
Prominent examples include cubic surfaces, threefolds, and fourfolds \cite{odaka-spotti-sun,liu-xu,liu},
as well as del Pezzo surfaces of degree $4$, whose K-moduli spaces are identified with the GIT moduli space of pencils of quadrics in $\PP^4$ \cite{mukai-mabuchi,odaka-spotti-sun}.

In the log setting, it was shown that for pairs $(\PP^n,\epsilon D)$ with $\epsilon>0$ sufficiently small, the K-moduli spaces coincide with the GIT moduli spaces of hypersurfaces \cite{ascher-devleming-liu}
(cf.~\cite{gallardo-martinez-garcia-spotti}).
More generally, this was extended to log Fano pairs $(X,\epsilon D)$ with $X$ K-polystable and $D$ lying in a multiple of the anticanonical linear system \cite{zhou}.

Even when the two moduli spaces do not coincide, their relationship can often be described by explicit wall crossings or birational transformations; 
see, for instance,
\cite{abban-cheltsov-kasprzyk-liu-petracci,ascher-devleming-liu-quadric,ascher-devleming-liu-K3,ascher-devleming-liu,cheltsov-duarte-guerreiro-fujita-krylov-martinez-garcia,cheltsov-thompson,devleming-ji-kennedy-hunt-quek,kim-liu-wang,liu-zhao,martinez-garcia-papazachariou-zhao,odaka-spotti-sun,papazachariou} for a non-exhaustive list of related works.

The moduli spaces considered in this paper lie outside the scope of the results discussed above.
Indeed, the varieties $X(s)$ are Fano hypersurfaces in flag varieties, but they do not belong to linear systems proportional to the anticanonical divisor.
On the other hand, the varieties $Y(s)$ arise as complete intersections of two $(1,1)$-divisors in $\PP\times\PP^\vee$, which likewise fall outside the classes of examples previously treated in the literature.

This naturally raises the question of how the GIT compactifications appearing in this paper relate to K-stability.
Motivated by this perspective, we propose the following conjecture.
\begin{conjecture}\label{conj:K-polystability}
The Fano varieties $X(s)$ and $Y(s)$, as well as the Fano pairs 
$(\Fl_n, c X(s))$ with $0<c<1$, are K-polystable for all $s\in \grs$.
\end{conjecture}

We remark that the K-polystability of the pair $(\PFl_n,\epsilon Y(s))$ is known
for sufficiently small $\epsilon>0$ and $s\in\grs$, as it is equivalent to the
GIT-polystability of $Y(s)$ as a divisor in $\PFl_n$ by \cite[Theorem~1.1]{zhou}.
However, the associated K-moduli stack parametrizing these pairs is $\fM_n$ rather than
$\fM_n^Y$ (see Theorem~\ref{thm:XtoY}), and therefore does not describe the moduli of
the varieties $Y(s)$ themselves.

\medskip

\subsection{Hessenberg varieties in types \texorpdfstring{$A$--$G$}{A-G}}\label{ss.gentype}
To explain our results in types other than type $A$, we first recall the definition
of a Hessenberg subspace and the equivalent notion of a Hessenberg subset.
Here we follow the definitions in Section II of the foundational paper on
Hessenberg varieties by de Mari, Procesi and
Shayman~\cite{demari-procesi-shayman}.
However, we generalize the definitions in the paper slightly to include 
the case of Hessenberg subvarieties of $\mathcal{P}$ (rather than just $\mathcal{B}$).

Let $G$ be a  reductive group over $\bk$ whose adjoint group is simple.
Fix 
a maximal torus $T$ of $G$ and  
a Borel subgroup $B$ 
containing $T$,
and let  $\Phi^+$, $\Phi^-$   be the corresponding set of positive, negative roots of $G$, respectively.
We write $\mathfrak{t}$, $\mathfrak{b}$ and $\mathfrak{g}$ for the 
Lie algebras of $T$, $B$ and $G$ respectively. 
We assume that $\mathfrak b$ is generated by $\mathfrak t$ and the root spaces $\mathfrak g_{\alpha}$ with $\alpha \in \Phi^-$.

A parabolic subgroup $P\subset G$ is said to be \emph{standard} if $B\subset P$.
Each conjugacy class $\mathcal{P}$ of parabolic subgroups contains exactly one standard parabolic subgroup.
Moreover, if we write $\Delta\subset \Phi^+$ for the set of simple roots, 
then there is a well-known one-to-one correspondence between subsets $\Sigma \subset\Delta$ and standard
parabolic subgroups $P$ of $G$. 
This correspondence sends $P$ to $\Sigma_P :=\{\alpha\in\Delta: \mathfrak{g}_{\alpha} \subset \mathfrak{p}
\}$ 
 for the Lie algebra $\mathfrak{p}$ of $P$.
We have $\Sigma_B=\emptyset$ and $\Sigma_G=\Delta$.  

The next lemma is a slight generalization of~\cite[Lemma 1, page 530]{demari-procesi-shayman}  
which deals with the
case where $P=B$. 

\begin{lemma}
  \label{hess-spaces} 
  Suppose $P$ is a standard parabolic subgroup of $G$. 
  Then there is a one-to-one correspondence between the following two sets:
\[
\begin{split}
\mathcal{H} &:=\{H \subset \mathfrak{g} : H \text{ is an $\Ad(P)$-invariant subspace containing $\mathfrak{p}$}\}; \\
\mathcal{M} &:= \{ M \subset \Phi^+ :\text{if }\beta\in M, \alpha\in (-\Delta)\cup \Sigma_P, \text{ and }
\beta+\alpha\in \Phi^+, \text{ then }\beta+\alpha\in M\}.
\end{split}
\]
  This correspondence sends $H\in\mathcal{H}$ to $\{\alpha\in \Phi^+ : \mathfrak{g}_{\alpha}\subset H\} \in \mathcal{M}$.\footnote{Our convention is opposite to that in \cite{demari-procesi-shayman} so that $M$ is a subset of $\Phi^+$ rather than $\Phi^-$.} 
\end{lemma}

A subspace $H\in\mathcal{H}$ is called a \emph{Hessenberg subspace of $\mathfrak{g}$ for $P$} and a subset $M\in\mathcal{M}$ 
is called a \emph{Hessenberg subset for $P$}.
Note that any Hessenberg subspace for a standard parabolic subgroup $P$ of $G$ is automatically a Hessenberg subspace for 
the Borel subgroup $B$. 

Let $M$ be a Hessenberg subset for $P$ with corresponding 
Hessenberg subspace $H$. 
Let $P$ act on $G\times H$ by $p(g,h)=(gp^{-1},\Ad(p)h)$. 
Write $G\times^P H$ for the quotient, and let $\mathcal{P}$ denote the conjugacy class of the parabolic $P$. 
We then get morphisms
\[
\begin{tikzcd}
& G \times^P H \arrow[ld, "\pi_1"'] \arrow[rd, "\pi_2"]\\
\mathcal{P} && \mathfrak{g} 
\end{tikzcd}
\]
given by $\pi_1(g,h)=\Ad(g)P$ and $\pi_2(g,h)=\Ad(g)h$ for $g\in G$ and $h\in H$. 
We write 
\[\mathcal{P}(M,s) := \pi_2^{-1}(s)\]
for the scheme-theoretic fiber of $\pi_2$ over $s\in \fg$.
The restriction of the map $\pi_1:G\times^P H\to\mathcal{P}$ 
to $\mathcal{P}(M,s)$ is a closed immersion. 
So we can view $\mathcal{P}(M,s)$ as a closed subscheme of~$\mathcal{P}$. 
When $\mathcal P=\mathcal B$, we write $\mathcal B(M,s):=\mathcal P(M,s)$.

When $s\in \grs$ is regular semisimple, $\mathcal{P}(M,s)$ is smooth of pure codimension $\lvert \Phi^+ - M\rvert$ in $\mathcal P$ (\cite{demari-procesi-shayman}). 
This is the regular semisimple Hessenberg variety in $\cP$ associated to $M$ and $s$. When $\cP\neq \mathcal{B}$, 
we sometimes refer to $\cP(M,s)$ as a generalized Hessenberg variety. 
The variety $\cP(M,s)$ is connected if and only if $\Delta \subset M$ by~\cite[Corollary~9]{demari-procesi-shayman} (cf.~\cite[Proposition~A.1]{anderson-tymoczko}).

\medskip 
 Since the adjoint group $G^{\mathrm{ad}}$ of $G$ is simple, $\Phi$ has a unique maximal root $\theta$. See, for example, \cite[Lemma A, p.~52]{HuLie}. 
The set \[M := \Phi^+-\{\theta\}\]  is clearly a Hessenberg
subset, and, for 
  $s\in\grs$, the variety 
  \[X := \mathcal{B}(M,s)\] is a 
  divisor in $\mathcal{B}$.
Let $TX$ denote the tangent bundle of $X$.
  
Restricting $\pi_2$ to $\grs\subset \fg$ defines a family of regular semisimple Hessenberg varieties.
At a base point $s\in\grs$, this family induces the Kodaira--Spencer map 
\beq\label{eq:KSmap.intro}\fg\cong \tang\grs|_s\lra \coh^1(X,TX),\eeq
whose image consists of first-order deformations of $X$ obtained by varying $s$.

\begin{theorem}\label{t_comp}
  Suppose $G$ is a simple adjoint group of rank $r$, 
  which is neither of type~$A_1$ nor $A_2$,  
  Then, with the above notation, 
 \[
         \dim\coh^i(X,TX) = 
         \begin{cases}
            r  & \text{ if }i = 0;\\
          r-1 &  \text{ if }i =1;\\
            0  & \text{ otherwise.}
         \end{cases}
 \] 
 Moreover, 
 the Kodaira--Spencer map~\eqref{eq:KSmap.intro} is surjective.
\end{theorem}

Let 
  $P$ be the maximal proper parabolic subgroup of $G$ such that
  the Hessenberg subspace corresponding to $M$ is $\Ad(P)$-invariant. 
Our final theorem is an analogue of Theorem~\ref{t_comp} for
  Hessenberg subvarieties $\mathcal P(M,s)$  in $\mathcal P$.   
  
 Let $s\in\grs$
  and set 
  \[Y := \mathcal{P}(M,s).\] 
  Restricting $\pi_2$ to $\grs$ yields the Kodaira--Spencer map
	\beq\label{eq:KSmap.p.intro}
	\fg \cong T_s\grs \longrightarrow \coh^1(Y,TY).
	\eeq
\begin{theorem}\label{t_comp p}
  Let $G$ be as in Theorem~\ref{t_comp}.  Then,  %
  we get the following results. 
  
  If $G$ is neither of type $A_r$ with $r \geq 2$ nor of type $C_r$, 
  then
  \[
    \dim\coh^i (Y,TY)=
    \begin{cases}
      r & \text{ if } i=0;\\
      r-1 & \text{ if } i=1;\\
      0& \text{ otherwise.}
    \end{cases}
  \]
  
  If $G$ is of type $A_r$ with $r \geq 2$, then   \[
    \dim\coh^i (Y,TY)=
    \begin{cases}
      r & \text{ if }i=0;\\
      r-2 & \text{ if }i=1;\\
      0& \text{ otherwise.}
    \end{cases}
  \]
  
  If $G$ is of type $C_r$, then
    \[
    \dim\coh^i (Y,TY)=
    \begin{cases}
       r(2r-1)  & \text{ if } i=0;\\
      0& \text{ otherwise.}
    \end{cases}
  \]
 Moreover, the Kodaira--Spencer map \eqref{eq:KSmap.p.intro} is surjective.
\end{theorem}

Our proof of Theorem~\ref{t_comp} starts by using adjunction to reduce the
computation of $\coh^i(X,TX)$ to a cohomology computation on the flag variety
$\mathcal{B}$ itself, namely the vanishing of
$\coh^i(\mathcal B, T\mathcal B \otimes \mathcal I_X)$,
where $\mathcal I_X$ is the ideal sheaf of $X$ in $\mathcal B$.

We then give two independent proofs of the required vanishing.
The first, stated as Theorem~\ref{prop: cohomology vanishing},
is proved in Subsection~\ref{ss_van}.
It uses the Borel--Weil--Bott theorem together with information about root
systems taken from Akhiezer's book~\cite{Ak}.

The second proof can be gotten by combining  
Theorems~\ref{prop: alt cohom vanishing short} 
and~\ref{prop: alt cohom vanishing long}, which are stated and proved in 
Section~\ref{alt_coho}. 
These results in fact do more than prove the vanishing of
$\coh^i(\mathcal{B}, T\mathcal{B}\otimes \mathcal{I}_X)$.
They amount to complete computations of
$\coh^i(\mathcal{B}, S^*T\mathcal{B}\otimes \mathcal{I}_X)$
and, more generally,
$\coh^i(\mathcal{B}, S^*T\mathcal{B}\otimes\mathcal{L})$,
where $\mathcal{L}$ is a line bundle associated to a negative root of $G$
and $S^nT\mathcal B$ denotes the $n$-th symmetric power of $T\mathcal B$.
The method of proof of Theorems~\ref{prop: alt cohom vanishing short}
and~\ref{prop: alt cohom vanishing long}  
is inspired by Broer's paper~\cite{broer}.

Our proof of Theorem~\ref{t_comp p} is similar in that vanishing
results, this time for $\coh^i(\mathcal{P}, T\mathcal{P}\otimes\mathcal{I}_Y)$,
are crucial.
Our first version of the required vanishing theorem, which works in all types, is
 Theorem~\ref{prop: cohomology vanishing p}. 
In type $A$, Theorem~\ref{prop: type A parabolic cohomology} 
is a generalization of Theorem~\ref{prop: cohomology vanishing p},
which computes 
$\coh^i(\mathcal{P}, S^*T\mathcal{P}\otimes \mathcal{I}_Y)$
explicitly.

In this paper, we consider the case of codimension one only, but  
we expect that the results presented here will serve as a first step toward understanding automorphisms and deformations of all regular semisimple Hessenberg varieties.

\subsection{Overview of the paper}\label{overview}
Sections~\ref{sec-pb},~\ref{sec-autsa},~\ref{s:M0n} and~\ref{s.moduli} concern
Hessenberg varieties in 
type $A$, and the remainder of the paper, 
Sections~\ref{s:general} 
and~\ref{alt_coho}
contain results which are valid in all types. 

The complete type $A$ flag variety $\Fl_n$ comes equipped with two
distinguished morphisms to projective space: one morphism 
$\Fl_n\to\mathbb{P} = \mathbb{P}^{n-1}$ sending the flag~$F_{\bullet}$ to $F_1$, and the other 
$\Fl_n\to\mathbb{P}^{\vee}$ to the dual projective space,
which sends $F_{\bullet}$ to 
$F_{n-1}$.
Pulling back the tautological line bundle on projective space,
we get two distinguished line bundles on $\Fl_n$, 
which we call $\mathcal{L}_1$ and $\mathcal{L}_n$, respectively.
Moreover, the variety $Y(s)$ is the image of the Hessenberg
variety $X(s)$ under the morphism 
$\Fl_n\to\mathbb{P}\times\mathbb{P}^{\vee}$.
So we abuse notation and write 
$\mathcal{L}_1 = \mathcal{O}_{\mathbb{P}\times\mathbb{P}^{\vee}}(1,0)|_{Y(s)}$
and $\mathcal{L}_n = \mathcal{O}_{\mathbb{P}\times\mathbb{P}^{\vee}}(0,1)|_{Y(s)}$.
In Section~\ref{sec-pb}, we prove Lemma~\ref{l.pullback}, which shows
that isomorphisms $X(s)\to X(s')$ between type $A$ Hessenberg varieties in
the complete flag variety either fix the line bundles $\mathcal{L}_1$
and $\mathcal{L}_n$ or interchange them.
The lemma is proved by showing that $\mathcal{L}_1$
and $\mathcal{L}_n$ are the only line bundles with Hilbert 
polynomial $\displaystyle p(t) = \binom{n + t -1 }{n - 1}$.  

In Section~\ref{sec-autsa}, we use Lemma~\ref{l.pullback}  to prove 
Theorem~\ref{t.isoA} on the isomorphism classes of type~$A$
Hessenberg varieties and Theorem~\ref{t.autsA} on their
automorphisms.
In particular, Theorem~\ref{t.isoA}\eqref{t.isoAY} and Theorem~\ref{t.autsA}\eqref{t.autsAY},
which concern the varieties $Y(s)$, are both proved in 
Subsection~\ref{ss:unique.ext.Y}.
The main idea for the argument here is to show that an isomorphism 
$\varphi:Y(s)\to Y(s')$ fixing $\mathcal{L}_1$ and $\mathcal{L}_{n}$ must arise  
from a pair of elements 
$(g_1,g_2)\in G\times G$ acting on $\mathbb{P}\times \mathbb{P}^{\vee}$.
Here, $g_1$ and $g_2$ are determined by the action of $\varphi$ 
on the $\bk$-vector spaces $\coh^0(Y(s),\mathcal{L}_1)$ and 
$\coh^0(Y(s),\mathcal{L}_n)$, respectively.

Theorems~\ref{t.isoA}(1) and \ref{t.autsA}\eqref{t.autsAX} are proved in~Subsection~\ref{ss.proof.isoAX}.
What makes these theorems challenging to prove is that
the deformation space for $X(s)$ is larger than that for $Y(s)$. 
So we need to find an obstruction explaining why two varieties
$X(s)$ and $X(s')$ can be non-isomorphic even when $Y(s)$ and $Y(s')$ are isomorphic. 
Roughly speaking, the obstruction we found has to do with the isomorphism classes
of certain vector bundles on $Y(s)$. See Lemma~\ref{l.hom} for the precise statement.

A second key technical input is a general result on isomorphisms of flag bundles.
In Subsection~\ref{ss.proof.isoAX}, we show that any isomorphism between flag bundles
over a fixed base arises from an isomorphism of the underlying vector bundles,
up to tensoring by a line bundle or dualization.
Together with Lemma~\ref{l.hom}, this allows us to characterize precisely when an
isomorphism of $Y(s)$ lifts to an isomorphism of $X(s)$. 
We also include a generalization to partial flag bundles, which allows us to treat
the corresponding generalized Hessenberg varieties $X_I(s)$ in Subsection~\ref{ss:partial.A}.

In Section~\ref{s:M0n}, we prove Theorem~\ref{t.HKStab} and Corollary~\ref{no-dot}.
Subsections~\ref{ss:group.actions}--\ref{ss:ModCurve.main} interpret our main theorems in type $A$
in terms of group actions 
and
the moduli space 
$M_{0,n}$ of genus $0$ curves with $n$ marked points, respectively.
We prove Corollary~\ref{no-dot} in Subsection~\ref{ss:Tymoczko}, 
which was one of the initial motivations for this work. 

Section~\ref{s:general} proves Theorems~\ref{t_comp} and~\ref{t_comp p}. Subsection~\ref{s.adj} uses adjunction to reduce the computation 
of $\coh^i(X,TX)$ for $X = \mathcal B(M,s) \subset \mathcal{B}$ to the vanishing of 
$\coh^i(\mathcal B, T\mathcal B \otimes \mathcal I_X)$, 
which is the content of Theorem~\ref{prop: cohomology vanishing}. 
Our first proof of this vanishing is given in Subsection~\ref{ss_van}.
What makes the computation in Subsection~\ref{ss_van} feasible is that we are able
to filter the vector bundles on the flag variety whose cohomology groups
determine $\coh^i(\mathcal B, T\mathcal B \otimes \mathcal I_X)$ by line
bundles in such a way that the vast majority of the line bundles have vanishing
cohomology.

The proof of Theorem~\ref{t_comp p} uses the same general method as that of Theorem~\ref{t_comp},
but the answer is different   
due to the following properties:
The cohomology $\coh^1(\mathcal P, T\mathcal P \otimes \mathcal I_Y)$ does not vanish when $G$
is  of type $A_r$ and   
the inclusion $\mathfrak g\subset \coh^0(\mathcal P, T \mathcal P)$ is proper  when $G$ is \label{homogeneous_space_proper}
of type $C_r$ (Subsection~\ref{s.partial}). 

In Section~\ref{alt_coho}, we prove Theorems~\ref{prop: alt cohom vanishing short}
and~\ref{prop: alt cohom vanishing long}, which 
compute  
$\coh^i(\mathcal B, S^nT\mathcal B \otimes \mathcal{I}_X)$, 
and Theorem~\ref{prop: type A parabolic cohomology}, which 
computes 
$\coh^i(\mathcal P, S^nT\mathcal P \otimes \mathcal I_Y)$ for $G$ of type $A$.
From the theorems, it is clear that most of the cohomology groups vanish.
So these theorems yield alternate proofs of Theorem~\ref{t_comp} in general types and
Theorem~\ref{t_comp p} in type~$A$.

Section~\ref{s.moduli} proves theorems relating the moduli stacks
of the type $A$ Hessenberg varieties $X(s)$ and $Y(s)$ to certain 
stack-theoretic quotients of the moduli space of genus $0$ 
curves with marked points.
So it upgrades the results proved at the end of Section~\ref{sec-autsa} to 
theorems about the moduli stacks of Hessenberg varieties.
It appears at the end of the paper because it relies in part on 
the surjectivity of the Kodaira--Spencer map \eqref{eq:KSmap.intro} 
which is
contained in Theorem~\ref{t_comp}.
We also discuss natural GIT compactifications of their good moduli spaces.

\subsection{Acknowledgments}   Brosnan thanks the Simons Foundation for
partial support (Travel Grant \#: 638366) during the preparation of this
manuscript 
and the Center for Complex Geometry at the Institute for Basic
Science in Daejeon for hospitality during the summer of 2023. 
He also thanks Jeffrey Adams for pointers on how to use the Atlas of Lie Groups
and Representations software to calculate cohomology groups on $G/B$. 
(In the end, this paper does not use any computer computations, but the Atlas
was still a very useful tool.)
Escobar was supported by the National Science Foundation (NSF) Career grant
DMS-2142656.  Hong and
D. Lee were supported by the Institute for Basic Science (IBS-R032-D1). They
also thank the Department of Mathematics at the University of Maryland and the Brin
Mathematics Research Center for hospitality during the spring of 2024. 
D.~Lee thanks Hyeonjun Park for helpful discussions on the proofs of
Theorems~\ref{thm:moduli.stack} and~\ref{thm:moduli.stack.Y} 
and Kenneth Ascher, Harold Blum and Chuyu Zhou for kindly answering his questions about moduli stacks of polarized schemes and Fano varieties, especially regarding the proof of Theorem~\ref{thm:MFano}. 
He also thanks Peter Crooks and Young-Hoon Kiem for helpful comments.
E. Lee was supported by the National Research Foundation of Korea(NRF) grant funded by the Korea government(MSIT) (No.\ RS-2022-00165641, RS-2023-00239947) and by POSCO Science Fellowship of POSCO TJ Park Foundation. 
Hong and E. Lee thank Mikiya Masuda for valuable discussions on automorphism groups when they visited Himeji. 

This project began as part of a problem session in the Banff International
Research Station (BIRS) Workshop 
``Interactions between Hessenberg Varieties, Chromatic Functions, 
and LLT Polynomials (22w5143)."
We heartily thank both BIRS and the organizers of the workshop.


\bigskip

\section{Pullbacks of line bundles}\label{sec-pb}
In this section, we restrict ourselves to type $A_{n-1}$ for $n\geq 4$. Let $G=\GL_n$ and let $B$ be the Borel subgroup of $G$ consisting of lower triangular matrices. Then the flag variety $\mathcal{B}=G/B$ is isomorphic to the variety of complete flags in $\bk^n$
\[
\Fl_n=\{F_\bullet=(F_1\subset \cdots \subset F_{n-1} \subset \bk^n)~:~ \dim  F_i=i ~\text{ for all }i\}.
\]

For each regular semisimple element $s\in \grs$, the associated regular semisimple Hessenberg variety $X(s)$ of codimension one is the subvariety defined by
\[X(s)=\{F_\bullet \in \Fl_n~:~ sF_1\subset F_{n-1}\}.\]
This is a smooth connected divisor in $\Fl_n$ by the general results in \cite{demari-procesi-shayman}. 
So, as $\Fl_n$ is smooth, it is the Cartier divisor associated to a line bundle on $\Fl_n$. 
To describe this line bundle,
consider the natural projection map 
\begin{equation}\label{eq: projection}
	\Fl_n\longrightarrow \mathbb{P}\times \mathbb{P}^\vee, \quad F_\bullet \mapsto (F_1,F_{n-1}),
\end{equation}
where $\PP$ (resp. $\PP^{\vee}$) denotes the space of lines in $\bk^n$ (resp. $(\bk^n)^*$) 
as in~\ref{ss-defa}.
Let $\cL_1$ and $\cL_n$ be the line bundles on $\Fl_n$ obtained as the pullbacks
of the ample generators $\mathcal{O}_{\mathbb{P}}(1)$ and
$\mathcal{O}_{\mathbb{P}^\vee}(1)$ on $\mathbb{P}$ and $\mathbb{P}^\vee$
by the map obtained by composing~\eqref{eq: projection} with the first and second projections onto
$\PP$ or $\PP^\vee$ respectively. 
Then, the dual $\cL_1^\vee$ of $\cL_1$ is the tautological line bundle on
$\Fl_n$ whose fiber at a point representing the flag $(F_1\subset\cdots\subset
F_{n-1}\subset\bk^n)$ is $F_1$. Similarly, $\cL_n$ is the tautological line
bundle with fiber $\bk^n/F_{n-1}$.
The Hessenberg variety $X(s)$ is then the scheme-theoretic vanishing locus of the homomorphism $\cL_1^\vee\to \cL_n$ on $\Fl_n$, which is the composition 
\[F_1\hookrightarrow \bk^n\xrightarrow{~s~}\bk^n\twoheadrightarrow \bk^n/F_{n-1}\]
on the fibers.
 Here the map $s$ in the middle is the matrix multiplication by $s$. In particular, $X(s)$ is a divisor associated to the line bundle $\cL_1\otimes \cL_n$ on $\Fl_n$, and {hence $(\mathcal{L}_1 \otimes \mathcal{L}_n)^{\vee}$ is its ideal sheaf.}\label{X_is_a_divisor}

To simplify the notation, we write $\cL_1$ and $\cL_n$ for the restrictions $\cL_1|_{X(s)}$ 
and $\cL_n|_{X(s)}$ to $X(s)$. 
As the pullback $\mathrm{Pic}\Fl_n \to \mathrm{Pic}X(s)$ 
is an isomorphism for $n\geq 4$ (\cite[Theorem 6.1]{ayzenberg-masuda-sato}), this abuse of 
notation is relatively harmless.

\medskip

In~\ref{subsection_intro_1.2}, we stated Lemma~\ref{l.inv} 
without proof.  
That Lemma proved the existence of an involution
$\iota$ of $G=\mathbf{GL}_n$ preserving $X(s)$ and $Y(s)$. 
We need the existence of $\iota$ for the statement of the main result of this
section, which is Lemma~\ref{l.pullback} below.
So, before stating Lemma~\ref{l.pullback}, we provide a proof of Lemma~\ref{l.inv}.

\subsection{Symmetric matrices and the proof of Lemma~\ref{l.inv}}\label{proof_Lem_l.inv}
\begin{lemma}\label{symover}   
  Let $V$ be an $n$-dimensional $\bk$-vector space and let $s\in\End V$.
  Then there exists a nondegenerate symmetric bilinear form $B$ on 
  $V$ such that, with respect to $B$, we have $s = s^t$.
\end{lemma}
\begin{proof}
    To show this, we use a slight variant of the argument from the MathOverflow
    answer~\cite{ggsymm}, which we sketch for the convenience of the reader.
    (The argument in \cite{ggsymm} is actually more general in that it deals
    with non-algebraically closed fields of arbitrary characteristic.)

     First note that, as is easily seen by considering orthogonal sums, it
     suffices to consider the case when $s$ has only one Jordan block. Next,
     assuming that $s$ has only one Jordan block, we assume, as we can, that
     the unique eigenvalue of that Jordan block is $0$.

     We can then replace $V$ with the vector space $V = \bk[x]/(x^n)$ and
     assume that $s$ is given by multiplication by $x$ on $V$. 
     (Assume also that $n > 0$ to avoid thinking about trivialities.)
     If $\lambda\in V^*$ is a linear functional, we get a symmetric bilinear
     form $B = B_{\lambda}$ on $V$ by setting $B(a,b) = \lambda (ab)$ for $a, b \in V$.
     Moreover, since $B(xa,b) = \lambda(xab) = \lambda(axb) = B(a,xb)$, $s$ is
     symmetric with respect to the bilinear form $B$.

     So it suffices to find $\lambda$ such that $B_{\lambda}$ is nondegenerate.
     For this, pick $\lambda$ to be the form given by  
     $\lambda(f) = f^{(n-1)}(0)/(n-1)!$.
     This makes sense because, if $f\in x^{n}\bk[x]$, then $f^{(n-1)}(0) = 0$.
     Equivalently, $\lambda$ is the unique form determined by setting
     \[
      \lambda(x^i) = 
        \begin{cases}
            0, & 0\leq i < n-1;\\
            1, & i = n-1.
        \end{cases}
     \]

     With $B = B_{\lambda}$, 
     the matrix $b_{ij} = B(x^i,x^j)$ (with $0\leq i,j< n$) is then given by 
     \begin{equation*} 
      b_{ij} = 
        \begin{cases}
            1, & \text{if $i+j=n-1$;}\\
            0, & \text{otherwise.}
        \end{cases}
      \end{equation*}
    In other words, $b_{ij}$ is the anti-diagonal matrix.  
    It is not hard to see by induction that $\displaystyle\det (b_{ij}) = (-1)^{\binom{n}{2}}$.
    In particular, the determinant is nonzero. 
    So $B$ is nondegenerate.
\end{proof}

\begin{remark}
    Since $\bk$ has characteristic not equal to $2$, the map sending a symmetric 
    bilinear form $B$ to the quadratic form given by $Q(v) = B(v,v)$ is a one-one
    correspondence. 
    Since any two quadratic forms over $\bk$ are isometric and quadratic forms
    correspond bijectively to symmetric bilinear forms, 
    Lemma~\ref{symover} is equivalent to the statement that any $n\times n$-matrix
    $s$ over $\bk$ is conjugate to a symmetric matrix.
    This is the subject of the MathOverflow question~\cite{ggsymm}.
\end{remark}

To make the proof of Lemma~\ref{l.inv} as clear as possible, 
it will help to have another elementary linear algebra lemma.

\begin{lemma}\label{perps}
   Suppose $V$ is a finite dimensional $\bk$-vector space with a nondegenerate symmetric bilinear 
   form $B$, a subspace $U\subseteq V$ and an endomorphism 
   $s\in\End V$.
   Then 
   \[
       (sU)^{\perp} = (s^{t})^{-1}U^{\perp},
   \]
   where $(s^{t})^{-1}U^{\perp}$  denotes the \emph{preimage} of 
   $U^{\perp}$ under the (possibly noninvertible) endomorphism $s^t$ of $V$. 
\end{lemma}

\begin{proof}
   Suppose $v\in (sU)^{\perp}$ and $u\in U$.
   Then $B(s^tv, u) = B(v,su) = 0$.
   So $s^tv\in U^{\perp}$.  
   Therefore, $v\in (s^t)^{-1} U^{\perp}$.

   On the other hand, suppose $v\in (s^t)^{-1} U^{\perp}$ and $u\in U$.
   Then $s^t v \in U^{\perp}$. 
   So $B(v, su) = B(s^t v, u) = 0$.
   Therefore, $v\in (sU)^{\perp}$. 
\end{proof}

\begin{proof}[Proof of Lemma~\ref{l.inv}]
   Pick $s\in\fg$.
   Using Lemma~\ref{symover} and the identification of symmetric bilinear forms 
   over $\bk$ with quadratic forms, 
   we can find a nondegenerate quadratic form $Q$ such that, with respect to $Q$, 
   we have $s = s^t$.
   It remains to show that $\iota$ preserves the Hessenberg varieties $X(s)$ and $Y(s)$, 
   where $\iota$ is the involution of $\Fl_n$ corresponding to $Q$ in the case of $X(s)$ and
   of $\PFl_n$ in the case of $Y(s)$.

   To handle the case of $X(s)$, pick $F_\bullet\in X(s)$. 
   We have $sF_1\subseteq F_{n-1}$.
   Therefore, $F_{n-1}^{\perp}\subseteq (sF_1)^{\perp}$.
   So, using Lemma~\ref{perps}, we see that 
   \begin{align*}
      sF_{n-1}^{\perp} &\subseteq s (sF_1)^{\perp} =  s(s^t)^{-1} F_1^{\perp} \\
                       &= s(s^{-1}) F_1^{\perp} \subseteq F_1^{\perp}.
    \end{align*} 

   This shows that $\iota(X(s)) \subseteq X(s)$.
   But then, since $\iota$ is an involution, it also shows that $X(s)\subseteq \iota(X(s))$,
   and, therefore, that $\iota(X(s)) = X(s)$.
   Essentially the same argument shows that $\iota(Y(s))=Y(s)$. 
\end{proof}

By abuse of notation, we continue to write $\iota$ for the restriction
to $X(s)$ or $Y(s)$ of the involution $\iota$ in Lemma~\ref{l.inv}.

\smallskip
We now state our main result of this section, which is analogous to the fact
that inner automorphisms of $\Fl_n$ preserve both $\cL_1$ and $\cL_n$, while
 $\iota$ interchanges them.
Consequently, any automorphism of $\Fl_n$ preserves the unordered set
$\{\cL_1,\cL_n\}$.

\begin{lemma}[Pullback lemma] \label{l.pullback}
Let $s,s'\in \grs$ be two regular semisimple elements. Let $\varphi:X(s)\to
X(s')$ be an isomorphism. 
Then, we have either
\begin{enumerate}[label=(\roman*),  ref=\roman*] 
 \item $(\varphi^*\cL_1,\varphi^*\cL_n) = (\cL_1,\cL_n)$, or \label{item:pullback_1}
 \item $(\varphi^*\cL_1,\varphi^*\cL_n) =  (\cL_n, \cL_1)$ \label{item:pullback_2}
\end{enumerate}
on $X(s)$. 
That is, the set $\{\cL_1,\cL_n\}$ is preserved by $\varphi^*$. In particular, 
	\beq\label{eq:pullback}\Aut X(s)=\Aut^+ X(s)\rtimes \langle \iota \rangle \eeq
where $\Aut^+ X(s) \trianglelefteq \Aut X(s)$ is the normal subgroup consisting of
automorphisms of $X(s)$ satisfying~\eqref{item:pullback_1}, and $\iota$ is any involution in Lemma~\ref{l.inv}.

\end{lemma}
This lemma is essential in our proofs of Theorems~\ref{t.isoA} and \ref{t.autsA} in the next section.

\begin{proof}
The first assertion is a direct consequence of the cohomological characterization of $\cL_1$ and $\cL_n$
provided in Lemma~\ref{l.CohomChar} below, which says that these line bundles are uniquely determined on $X(s)$ by the dimensions of the cohomology groups of their sufficiently high tensor powers. 
The identification~\eqref{eq:pullback} is immediate,
since $\Aut^+ X(s)$ is preserved under
conjugation by $\iota$, and every automorphism satisfying~\eqref{item:pullback_2} differs
from an element of $\Aut^+ X(s)$ by composition with $\iota$.
\end{proof}

To complete the proof of Lemma~\ref{l.pullback}, we are left to prove the following lemma.
\begin{lemma}
\label{l.CohomChar}
For a line bundle $L\in \mathrm{Pic}\Fl_n$, the following are equivalent.
	\begin{enumerate}
		\item $L=\cL_1$ or $\cL_n$.
		\item For a sufficiently large integer $k$,
		\[ \dim \,\coh^i(X(s),L^k)=\begin{cases}
			\displaystyle\binom{n+k-1}{k} &\text{ for }i=0;\\
			0 &\text{ otherwise.}
		\end{cases}\]
		\item $L$ is globally generated on $\Fl_n$, and for every~$k$,
		\[\chi(X(s),L^k)=\binom{n+k-1}{k}\]
		where $\chi$ denotes the holomorphic Euler characteristic.
	\end{enumerate}
\end{lemma}

The rest of this section is devoted to the proof of Lemma~\ref{l.CohomChar}. In the first subsection, we recall the Borel--Weil--Bott theorem and the Weyl dimension formula in type $A$. Using this, we prove  Lemma~\ref{l.CohomChar} in the second subsection.

\subsection{Borel--Weil--Bott theorem and Weyl dimension formula in type \texorpdfstring{$A$}{A}}

The Borel--Weil--Bott theorem provides a complete classification of the sheaf cohomology of line bundles on $\Fl_n$ as explicit irreducible $G$-modules.
See, for example,~\cite[p.~103]{Ak} for the statement of the Borel--Weil--Bott
theorem and~\cite[p.~402]{fulton-harris} for the Weyl dimension formula.

\smallskip

Let $T\cong (\mathbb{G}_m)^n$ be the maximal torus in $B$.
Let $\lambda$ be an integral weight of~$T$. Under the natural isomorphism $\Hom(T,\mathbb{G}_m)\cong \mathbb{Z}^n$ as abelian groups, $\lambda$ 
can be regarded as an element of $\mathbb{Z}^n$. So we write $\lambda=(\lambda_1,\dots ,\lambda_n)$. To each integral weight~$\lambda$, associate a one-dimensional $B$-representation $\bk_{\lambda}$ with character $\chi_\lambda:B\to \mathbb{G}_m$, obtained by composing $\lambda$ with the canonical projection $B\to T$.
We denote by 
\[\cL(\lambda):=G\times ^B\bk_{\lambda}\]
the line bundle on $\Fl_n\cong G/B$ associated to $\lambda$, which is obtained as the quotient of $G\times \bk_{\lambda}$ by the $B$-action $b\times (g,v):=(gb^{-1},\chi_\lambda(b)v)$. 

Every line bundle on $\Fl_n$ is of this form. 
Moreover, $\cL(\lambda+\mu)\cong \cL(\lambda)\otimes \cL(\mu)$, and $\cL(\mathbbm{1})\cong \cO_{\Fl_n}$ for $\mathbbm{1}=(1,\dots,1)$.
The map $\lambda\mapsto \cL(\lambda)$ induces an isomorphism $\mathbb{Z}^n/\langle\mathbbm{1}\rangle\cong \mathrm{Pic}\Fl_n$.
Hence, we often consider $\lambda$ as an element in $\mathbb{Z}^n/\langle\mathbbm{1}\rangle$.

Let $e_i-e_j$ with $1\leq i<j\leq n$ be positive roots, where $e_i$ denotes the $i$-th standard basis vector $(0,\dots,1,\dots,0)$ with the $i$-th component 1 and 0 elsewhere. 
We say that $\lambda=(\lambda_1,\dots, \lambda_n)$ is \emph{dominant} (resp. \emph{regular}) if $\lambda_i\geq \lambda_{i+1}$  for all $1\leq i<n$ (resp. $\lambda_i$ are mutually distinct). In particular, $\lambda$ is regular dominant if and only if $\lambda_i>\lambda_{i+1}$ for all $1\leq i<n$. \label{def_regular_dominant} It is well known that $\cL(\lambda)$ is ample (resp. globally generated) 
if and only if $\lambda$ is regular dominant (resp. dominant).
We say that $\lambda$ is \emph{singular} if it is not regular.

We denote $\rho=(n-1,\dots, 1, 0)$. This is equal to the half of the sum of all the positive roots $e_i-e_j$ with $1\leq i<j\leq n$ modulo $\langle\mathbbm{1}\rangle$.

For every integral weight $\lambda$
with $\lambda+\rho$ regular, there exists a unique permutation $w\in S_n$ such that $w*\lambda:=w(\lambda+\rho)-\rho$ is dominant. Here, for $w \in S_n$ and $\lambda = (\lambda_1,\lambda_2,\dots,\lambda_n)$, we set $w(\lambda) = (\lambda_{w^{-1}(1)},\lambda_{w^{-1}(2)},\dots,\lambda_{w^{-1}(n)})$. Let $\ell(\lambda)$ denote the length $\ell(w)$ of $w$, which is by definition  the minimal number $\ell$ such that $w=s_{i_1}\cdots s_{i_\ell}$, where $s_i:=(i,i+1)$ denotes the $i$-th simple transposition for $1\leq i<n$.
\begin{theorem}[Borel--Weil--Bott theorem] \label{t.bwb}
	Let $\lambda$ be an integral weight of $T$. Then,
	\begin{enumerate}
		\item if $\lambda+\rho$ is singular, then $\coh^i(\Fl_n,\cL(\lambda))=0$ for all $i$;
		\item if $\lambda+\rho$ is regular, then $\coh^i(\Fl_n,\cL(\lambda))=0$ for all $i\neq \ell(\lambda)$ and
                \[\coh^{\ell(\lambda)}(\Fl_n,\cL(\lambda))\;\cong\; \coh^0(\Fl_n,\cL(w*\lambda))\;\cong\; V(w*\lambda)\]
		as a $G$-module, where $V(w*\lambda)$ is the irreducible
        $G$-module with the highest weight $w*\lambda$. 
	\end{enumerate}
	In particular, the sheaf cohomology of $\cL(\lambda)$ on $\Fl_n$ 
        is nonzero in at most one degree.
\end{theorem}

The Weyl dimension formula is a formula which computes the dimension of a given irreducible $G$-module.
\begin{theorem}[Weyl dimension formula] \label{t.weyl}
	Let $\lambda=(\lambda_1,\dots,\lambda_n)$ be an integral dominant weight of $T$. Then,
	\[\dim V(\lambda)= \prod_{1\leq i<j\leq n} \left(1+\frac{\lambda_i-\lambda_j}{j-i}\right).\]
\end{theorem}

\medskip

\subsection{Proof of Lemma~\ref{l.CohomChar}} \label{ss.proof_of_Lemma_l.CohomChar}
We now prove Lemma~\ref{l.CohomChar}. Throughout the proof, we let $\lambda=(\lambda_1,\dots, \lambda_n)$ be an integral weight of $T$ such that $L=\cL(\lambda)$.

Let $\theta:=e_1-e_n$. Since $\cL_1=\cL(e_1)$ and $\cL_n=\cL(-e_n)$ by definition, we have
\[\cL_1\otimes \cL_n=\cL(\theta).\]
Since $(\mathcal L_1 \otimes \mathcal L_n)^{\vee}=\cL(-\theta)$ is the ideal sheaf of $X(s)$,  
we get  
the long exact sequence
\beq\label{les.pullback}
\cdots \longrightarrow\coh^i(\Fl_n,L^k\otimes \cL(-\theta))\longrightarrow \coh^i(\Fl_n,L^k)\longrightarrow \coh^i(X(s),L^k)\longrightarrow \cdots\eeq
of cohomology groups induced by the short exact sequence
\beq \label{ses:closed}
0\longrightarrow \cL(-\theta) \longrightarrow \cO_{\Fl_n}\lra \cO_{X(s)}\lra 0.\eeq
Using Theorems~\ref{t.bwb} and \ref{t.weyl} and \eqref{les.pullback}, we  will prove $(1)\implies (2)\implies (3)\implies (1)$.

The implication $(1)\implies (2)$ directly follows from \eqref{les.pullback}. 
Indeed, if $\lambda=e_1$ or $-e_n$, then the associated weight $k\lambda -\theta $ of $L^k\otimes \cL(-\theta)$
is $(k-1)e_1+e_n$ or $-e_1-(k-1)e_n$ respectively. 
In either case, $(k\lambda -\theta )+\rho$ is singular. 
Hence, by Theorem~\ref{t.bwb},
$\coh^i(\Fl_n,L^k\otimes \cL(-\theta) )=0$ for all $i,k\geq 0$, and
\[\coh^i(X(s),L^k)\cong \coh^i(\Fl_n,L^k)\cong \begin{cases}
	V(ke_1) ~\text{ or }~V(ke_n) &\text{ if }i=0;\\
	0 &\text{ otherwise.}
\end{cases}\]
By Theorem~\ref{t.weyl}, we then have $\dim V(ke_1)=\dim V(ke_n)=\binom{n+k-1}{k}$.

\smallskip

For the implication $(2)\implies (3)$,  assuming that (2) holds, we will 
show that $\lambda$ is dominant (i.e. $\lambda_1\geq \cdots \geq \lambda_n)$. 
This implies that (3) holds, since the cohomology of $L^k$ vanishes in positive degrees for $k\geq 0$ by Theorem~\ref{t.bwb}, and since $k\in \mathbb{Z}\mapsto \chi(X(s),L^k)$ is a polynomial function due to the Hirzebruch--Riemann--Roch.

Suppose to the contrary that $\lambda$ is not dominant. This means that  $\lambda_i<\lambda_{i+1}$ for some $i<n$ and $k \lambda+\rho$ is not regular dominant for $k\geq 1$. So, we have $\coh^0(\Fl_n,L^k)=0$ by Theorem~\ref{t.bwb}, and in \eqref{les.pullback} we have a short exact sequence
\beq \label{ses:pullback} 
0\lra \coh^0(X(s),L^k)\lra \coh^1(\Fl_n, L^k\otimes \cL(-\theta))\lra \coh^1(\Fl_n,L^k)\lra 0
\eeq
for sufficiently large $k$,
because  $\coh^1(X(s), L^k)$  vanishes for sufficiently large $k$ by the assumption~(2). 
Moreover, since $\coh^0(X(s),L^k)\neq0 $ by (2), $\coh^1(\Fl_n,L^k\otimes \cL(-\theta))$ in the middle is also nonzero. 
So, by Theorem~\ref{t.bwb}, we have $\ell(k\lambda-\theta)=1$ and  there exists a simple transposition $s_r=(r,r+1)$ with $1\leq r<n$ such that
\beq\label{eq:mu}
\mu:=s_r*(k\lambda -\theta)
\eeq
is dominant and the cohomology group \(\coh^1(\Fl_n, L^k\otimes \cL(-\theta))\) is isomorphic to $V(\mu)$ as $G$-modules. Moreover, $r$ does not depend on $k$ whenever $k$ is sufficiently large.

By a general fact that  $s_r*(\lambda'+\lambda'')=s_r*\lambda'+s_r (\lambda'')$ for any weights $\lambda'$ and $\lambda''$, we can rewrite \eqref{eq:mu} as
\[
s_r*(k\lambda)=\mu + \alpha, \quad \text{ where }~\alpha:= s_r(\theta), 
\]
and $\mu$ is by definition dominant for sufficiently large $k$.

We claim that $s_r*(k\lambda)$ is also dominant for sufficiently large $k$. Indeed, when $2\leq r\leq n-2$, it is immediate from $\alpha=\theta$ being dominant.
When $r=1$, comparing
	\begin{equation*}
		\begin{split}
			\mu&=(k\lambda_2-1,~k\lambda_1,\;\;\;\;\;\;\;~k\lambda_3,~\dots,~ k\lambda_{n-1},~k\lambda_n+1) \quad \text{ and}\\
			s_r*(k\lambda)&=(k\lambda_2-1,~k\lambda_1+1,~k\lambda_3,~\dots,~k\lambda_{n-1},~k\lambda_n),
		\end{split}
	\end{equation*}
	one can see that  the dominance of $\mu$ 
    implies the dominance of $s_r*(k\lambda)$ for sufficiently large $k$. Similarly, when $r=n-1$, 
    comparing
	\begin{equation*}
		\begin{split}
			\mu&=(k\lambda_1-1,~k\lambda_2,~k\lambda_3,~\dots,~ k\lambda_{n},~\;\;\;\;\;k\lambda_{n-1}+1)\quad \text{ and}\\
			s_r*(k\lambda)&=(k\lambda_1,~\;\;\;\;\;\;k\lambda_2,~k\lambda_3,~\dots,~k\lambda_{n}-1,~k\lambda_{n-1}+1),
		\end{split}
	\end{equation*}
	one can see that the dominance of $\mu$ implies the dominance of $s_r*(k\lambda)$ for sufficiently large $k$. This proves the claim.

However, this contradicts the surjection given by \eqref{ses:pullback}
\[
V(\mu) \lra V(\mu+\alpha)\lra 0
\]
which is equivariant under the natural action of the maximal torus $T(s)=Z_G(s)$.
Indeed, the highest weight $\mu+\alpha$ of $V(\mu+\alpha)$ 
is higher than that of $V(\mu)$ since $\alpha$ is positive. 
Hence it cannot be surjective. Therefore, $\lambda$ is dominant and we obtain (3) as explained at the beginning of the proof of $(2) \implies (3)$.

\smallskip
For the implication $(3)\implies (1)$, we apply Theorem~\ref{t.weyl} to the equality
\beq \label{eq:EulerChar} \chi(X(s),L^k)=\chi(\Fl_n, L^k)-\chi(\Fl_n,L^k\otimes \cL(-\theta))\eeq
induced by \eqref{ses:closed}. To compute $\chi(\Fl_n,L^k\otimes \cL(-\theta))$ for $k>0$, observe that
$k\lambda+\rho -\theta$ is dominant since both $\lambda$ and $\rho-\theta$ are dominant,
    and moreover, it is regular if and only if $\lambda_1>\lambda_2$ and $\lambda_{n-1}>\lambda_n$.  Thus, by Theorem~\ref{t.bwb},
\[\chi(\Fl_n, L^k\otimes \cL(-\theta))=\begin{cases}
		\dim V(k\lambda -\theta) &\text{ if }\lambda_1>\lambda_2 \text{ and }\lambda_{n-1}>\lambda_n;\\
		0 &\text{ otherwise}
	\end{cases}\]
for $k>0$. The equality \eqref{eq:EulerChar}  now reads as 
\[\binom{n+k-1}{k}=\begin{cases}
		\dim V(k\lambda)-\dim V(k\lambda -\theta )&\text{if } \lambda_1>\lambda_2 ~\text{ and }~ \lambda_{n-1}>\lambda_n;\\
		\dim V(k\lambda) &\text{otherwise.}
	\end{cases} \]
The terms in the right-hand side can be computed by Theorem~\ref{t.weyl}: 
	\begin{equation*}
		\begin{split}
			\dim V(k\lambda)&=\prod_{1\leq i<j\leq n}\left(1+\frac{\lambda_i-\lambda_j}{j-i}k\right) \qquad \qquad \;\;=: A(k),\\
			\dim V(k\lambda -\theta )&=\prod_{1\leq i<j\leq n}\left(1+\frac{\lambda_i-\lambda_j}{j-i}k-
			\frac{\epsilon_{ij}}{j-i}\right) \quad =:B(k),
		\end{split}
	\end{equation*}
	where $\epsilon_{ij}$ are defined by 
	\[\epsilon_{ij}=\begin{cases}
		2 &\text{if }~(i,j)=(1,n);\\
		1 & \text{if }~1=i<j<n ~\text{ or }~1<i<j=n;\\
		0 & \text{otherwise}.
	\end{cases}\]
 As a polynomial in $k$, the degree of $A(k)$ is equal to
\[N(\lambda):=\#\{(i,j)~:~1\leq i<j\leq n,~\lambda_i>\lambda_j\}.\]

Assume that $\lambda_1>\lambda_2$ and $\lambda_{n-1}>\lambda_n$.
Then $B$ also has degree
$N(\lambda)$ with the same leading coefficient as $A(k)$,
since $\frac{\lambda_i-\lambda_j}{j-i}k=0$ implies $\epsilon_{ij}=0$.
However, the second leading coefficients, that is, the coefficients of
$k^{N(\lambda)-1}$ in $A(k)$ and $B(k)$ are not the same. 
This is because $\epsilon_{ij}\geq 0$ for all $i,j$ and 
$\epsilon_{1,2} = 1 >0$, $\lambda_1 > \lambda_2$ by the assumption. In particular, the degree of $A(k)-B(k)$ is $N(\lambda)-1$.
Consequently, 
\[
n-1=\deg_k \binom{n+k-1}{n-1} 
 =\begin{cases}
		N(\lambda)-1&\text{if } \lambda_1>\lambda_2 ~\text{ and }~ \lambda_{n-1}>\lambda_n;\\
		N(\lambda) \ &\text{otherwise.}\end{cases}
\]	

	The first case implies $n=N(\lambda)\geq 2n-3$ since $\lambda_1>\lambda_i$ and $\lambda_j>\lambda_n$ for $i\neq 1$ and $j\neq n$. This contradicts our assumption that $n>3$. In the second case, there exists $i$ with $\lambda_i>\lambda_{i+1}$ since otherwise $n-1=N(\lambda)=0$, a contradiction. If $\lambda_i>\lambda_{i+1}$ for some $i$, then we have $n-1=N(\lambda)\geq i(n-i)$, which forces either
	\[\lambda_1>\lambda_2=\cdots =\lambda_n \quad \text{ or }\quad \lambda_{1}=\cdots =\lambda_{n-1}>\lambda_n.\]
	In either case, we have
	\[\binom{n+k-1}{n-1}=A(k)\;\geq\; \prod_{i=1}^{n-1}\left(1+\frac{\lambda_1-\lambda_n}{i}k\right)=\binom{(\lambda_1-\lambda_n)k+n-1}{n-1}\]
	for sufficiently large $k$. So, $\lambda_1-\lambda_n=1$, or equivalently $L=\cL_1$ or $\cL_n$. 
    \qed

\bigskip

\section{Automorphisms and deformations in type \texorpdfstring{$A$}{A}}\label{sec-autsa}
We prove Theorems~\ref{t.isoA} and \ref{t.autsA} 
by analyzing the generalized Hessenberg variety
\[Y(s)=\{(F_1 \subset F_{n-1} \subset \bk^n)\in \PFl_n~:~sF_1 \subset F_{n-1}\},\]
which is a subvariety of the generalized flag variety $\PFl_n$ parametrizing partial flags
$F_1\subset F_{n-1}\subset \bk^n$, and by exploiting its relationship with $X(s)$.
For general results on generalized Hessenberg varieties in type~$A$ and their relation to Hessenberg varieties, we refer to~\cite{kiem-lee}.

In the codimension one case, the generalized Hessenberg variety in $\PFl_n$ is precisely $Y(s)$.
Moreover, it can be realized as the complete intersection of two $(1,1)$-divisors in
$\PP \times \PP^{\vee}$, cut out by the conditions $F_1\subset F_{n-1}$ and $sF_1 \subset F_{n-1}$.
This concrete description allows us to study isomorphisms of $Y(s)$ via their extensions to
automorphisms of the ambient space $\PP \times \PP^\vee$.

\medskip

The first part of this section is devoted to understanding isomorphisms and automorphisms of $Y(s)$.
In Subsection~\ref{ss:cohomY}, we compute the cohomology of line bundles on $Y(s)$.
In Subsection~\ref{ss:unique.ext.Y}, we show that every isomorphism $Y(s)\to Y(s')$ uniquely extends to an automorphism of $\PP\times\PP^\vee$ (Lemma~\ref{lem:unique.ext.Y}), and we use this to prove Theorems~\ref{t.isoA}\eqref{t.isoAY} and~\ref{t.autsA}\eqref{t.autsAY}.

\medskip

To study $X(s)$,
we consider the natural projection
\beq\label{eq:projection.pi}\pi:X(s)\longrightarrow Y(s), \quad F_\bullet \mapsto (F_1 \subset F_{n-1}),\eeq
induced by \eqref{eq: projection}, 
which is a fibration with fibers $\Fl(F_{n-1}/F_1)\cong \Fl_{n-2}$.
Our goal is to determine which isomorphisms of $Y(s)$ admit lifts to
isomorphisms of $X(s)$.

By the Pullback lemma,
every isomorphism $\varphi:X(s)\to X(s')$ induces a unique isomorphism $\varphi_Y:Y(s)\to Y(s')$ compatible with  $\pi$ 
(Proposition~\ref{prop:unique.ext.X}).
Thus the key problem becomes charactering those isomorphisms $\varphi_Y$
that admit such a lift. 

\medskip

We show that this lifting problem is governed by whether an isomorphism of $Y(s)$
preserves the incidence variety $\PFl_n$ as an automorphism of $\PP\times\PP^\vee$.
In Subsection~\ref{ss Lemma}, we make this criterion precise by computing
certain $\Hom$ spaces of vector bundles on $Y(s)$ (Lemma~\ref{l.hom}),
which shows that only those isomorphisms preserving the incidence variety
can lift to isomorphisms of $X(s)$.

Finally, in Subsection~\ref{ss.proof.isoAX}, we combine these results with
a general description of isomorphisms of flag bundles to complete the proofs
of Theorems~\ref{t.isoA}\eqref{t.isoAX} and~\ref{t.autsA}\eqref{t.autsAX}.

\medskip
Throughout this section, we continue to assume $n\geq 4$.

\subsection{Cohomology of line bundles on \texorpdfstring{$Y(s)$}{Y(s)}} \label{ss:cohomY}

By abuse of notation, we denote by $\cL_1$ and $\cL_n$ the line
bundles on $X(s)$, $Y(s)$, $\Fl_n$, and $\PFl_n$ obtained by pulling back
$\cO_\PP(1)$ and $\cO_{\PP^\vee}(1)$ along the natural projections to
$\PP$ and $\PP^\vee$, respectively.
In particular, we have $\cL_1=\pi^*\cL_1$ and $\cL_n=\pi^*\cL_n$ on $X(s)$.

In this subsection, we compute the cohomology of the line bundles
$\cL_1$, $\cL_n$, $\cL_1^\vee$, $\cL_n^\vee$,
as well as their tensor products
$\cL_1\otimes \cL_n$ and $\cL_1^\vee\otimes \cL_n^\vee$,
on $X(s)$ and $Y(s)$.
This will be used throughout this section.

\begin{lemma}[{\cite[Lemma~3.3.2]{BrionKumar}}]\label{lem:projection.formula} 
    Let $f: X\to Y$ be a morphism of algebraic varieties with $f_*\cO_X=\cO_Y$ and $R^if_*\cO_X=0$ for $i>0$. Then, 
    the pullback homomorphism $\coh^i(Y,\cV)\to\coh^i(X,f^*\cV)$ is an isomorphism
    for any vector bundle $\cV$ on $Y$ and for all $i \geq 0$.
\end{lemma}
\begin{proof}
   By the projection formula (\cite[Ex.~II.5.1.(d)]{hartshorne}), we have $f_*f^*\cV=\cV$ and $R^if_*f^*\cV=0$ for $i>0$. The assertion then holds by the Leray spectral sequence.
\end{proof}

For elements $x_1, \dots, x_r \in \mathfrak g$, denote by $\langle x_1, \dots, x_r \rangle$ the $\bk$-span of $x_1, \dots, x_r$ in $\mathfrak g$. 
\begin{lemma}\label{l.cohomY}
Let $n\geq 4$. Let $s\in\grs$ and let $Z=X(s)$ or $Y(s)$. Then:
\begin{enumerate}
  \item $\coh^i(Z,\cL_1^\vee)=\coh^i(Z,\cL_n^\vee)=0$ for all $i$.
  \item $\coh^i(Z,\cL_1^\vee\otimes\cL_n^\vee)=0$ for all $i$.
  \item $\coh^0(Z,\cL_1)\cong (\bk^n)^\vee$ and $\coh^0(Z,\cL_n)\cong \bk^n$ canonically.
  \item $\coh^0(Z,\cL_1\otimes\cL_n)\cong \End(\bk^n)/\langle \id,s\rangle$ canonically.
\end{enumerate}
\end{lemma}

\begin{proof}
The fibration $\pi:X(s)\to Y(s)$ in~\eqref{eq:projection.pi} satisfies the
assumptions of Lemma~\ref{lem:projection.formula}. Hence, it is enough to prove
each statement for one of $X(s)$ or $Y(s)$. We prove (1)--(3) on $X(s)$ using
\eqref{ses:closed} and Theorem~\ref{t.bwb}, and we prove (4) on $Y(s)$ using that
$Y(s)\subset \PP\times\PP^\vee$ is a complete intersection of two $(1,1)$-divisors.

\smallskip
For (1) and (2), let $\cW=\cL_1^\vee$, $\cL_n^\vee$, or $\cL_1^\vee\otimes\cL_n^\vee$ on $\Fl_n$. 
Let $\eta\in\{-e_1,\,e_n,-\theta\}$ (where $\theta=e_1-e_n$) so that $\cW=\cL(\eta)$. Then
$\cW\otimes\cL_1^\vee\otimes\cL_n^\vee\cong \cL(\eta-\theta)$.
For $n\ge 4$, both $\eta+\rho$ and $\eta-\theta+\rho$ are singular, so
Theorem~\ref{t.bwb} gives $\coh^i(\Fl_n,\cW)=0=\coh^i(\Fl_n,\cW\otimes\cL_1^\vee\otimes\cL_n^\vee)$ for all $i$.
Applying $R\Gamma(\Fl_n,\cW\otimes -)$ to~\eqref{ses:closed} yields
$\coh^i(X(s),\cW|_{X(s)})=0$ for all $i$, proving (1) and (2) for $X(s)$.

\smallskip
For (3), tensor \eqref{ses:closed} with $\cL_1$ and $\cL_n$ and take cohomology to obtain
\[\begin{split}
\coh^0(\Fl_n,\cL_n^\vee)\lra \coh^0(\Fl_n,\cL_1)\lra \coh^0(X(s),\cL_1)\lra \coh^1(\Fl_n,\cL_n^\vee),\\
\coh^0(\Fl_n,\cL_1^\vee)\lra \coh^0(\Fl_n,\cL_n)\lra \coh^0(X(s),\cL_n)\lra \coh^1(\Fl_n,\cL_1^\vee).
\end{split}\]
By Lemma~\ref{lem:projection.formula} applied to 
$\Fl_n\to\PP$ and $\Fl_n\to\PP^\vee$, 
we have canonical identifications
\beq\label{eq:O1}\begin{split}
	&\coh^0(\Fl_n,\cL_1)\cong \coh^0(\PP,\cO(1))=(\bk^n)^\vee,\\ 
	&\coh^0(\Fl_n,\cL_n)\cong \coh^0(\PP^\vee,\cO(1))=\bk^n,
\end{split}
\eeq 
and also
$\coh^i(\Fl_n,\cL_1^\vee)=0=\coh^i(\Fl_n,\cL_n^\vee)$ for all $i$.
Therefore the middle maps are isomorphisms, proving (3) for $X(s)$.

\smallskip
For (4), consider the exact sequences on $\PP\times\PP^\vee$ and on $\PFl_n$:
\beq \label{s eq}\begin{split}
0\lra \cO_{\PP\times\PP^\vee}(-1,-1)&\xrightarrow{~\id~}\cO_{\PP\times\PP^\vee}
\lra \cO_{\PFl_n}\lra 0,\\
0\lra \cL_1^\vee\otimes\cL_n^\vee&\xrightarrow{~s~}\cO_{\PFl_n}\lra \cO_{Y(s)}\lra 0.
\end{split}\eeq
Twisting the first sequence by $\cO(1,1)$ and taking $H^0$, we obtain
\beq\label{eq:pgln}\coh^0(\PFl_n,\cL_1\otimes\cL_n)\cong
\coh^0(\PP\times\PP^\vee,\cO(1,1))/\langle \id\rangle
\cong \End(\bk^n)/\langle \id\rangle.\eeq
The second sequence then shows that the further restriction to $Y(s)$ has kernel
$\langle s\rangle$. Hence $\coh^0(Y(s),\cL_1\otimes\cL_n)\cong \End(\bk^n)/\langle \id,s\rangle$,
proving (4) for $Y(s)$. 
\end{proof}

By Lemma~\ref{lem:projection.formula} applied to $\pi:X(s)\to Y(s)$, the cohomological characterization of the line bundles $\cL_i$
(Lemma~\ref{l.CohomChar}) remains valid after replacing $X(s)$ by $Y(s)$.
Consequently, the Pullback lemma  also holds for $Y(s)$,
although in this case it can also be proved directly using the fact that
$Y(s)$ is the complete intersection of two $(1,1)$-divisors in $\PP\times\PP^\vee$, and hence that its ample cone is generated by $\cL_1$ and $\cL_n$.

\begin{lemma}[Pullback lemma for $Y(s)$] \label{l.pullback.Y}
Let $s,s'\in \grs$ be two regular semisimple elements. Let $\varphi:Y(s)\to
Y(s')$ be an isomorphism. 
Then, we have either
\begin{enumerate}
 \item[(i)]	$(\varphi^*\cL_1,\varphi^*\cL_n) = (\cL_1,\cL_n)$, or 
 \item[(ii)]    $(\varphi^*\cL_1,\varphi^*\cL_n) =  (\cL_n, \cL_1)$
\end{enumerate}
on $Y(s)$. 
That is, the set $\{\cL_1,\cL_n\}$ is preserved by $\varphi^*$. In particular, 
	\[\Aut Y(s)=\Aut^+ Y(s)\rtimes \langle \iota \rangle \]
where $\Aut^+ Y(s) \trianglelefteq \Aut Y(s)$ is the normal subgroup consisting of
automorphisms of $Y(s)$ satisfying (i). \qed
\end{lemma}

\medskip

\subsection{Unique extension of isomorphisms}\label{ss:unique.ext.Y} 

In this subsection, we prove Theorems~\ref{t.isoA}\eqref{t.isoAY} and \ref{t.autsA}\eqref{t.autsAY}. The following lemma is the key ingredient. Note that 
\[\Aut \PP\times \PP^\vee = (\Aut \PP\times \Aut \PP^\vee)\rtimes \langle \iota \rangle\] for any involution $\iota:(F_1,F_{n-1})\mapsto (F_{n-1}^\perp, F_1^\perp)$ (with respect to any choice of a nondegenerate quadratic form $Q$). We write $\Aut^+ \PP\times\PP^\vee:=\Aut \PP\times\Aut \PP^\vee$. 

\begin{lemma}
\label{lem:unique.ext.Y}
	Let $s,s'\in \grs$. 
	Any isomorphism $Y(s)\to Y(s')$ extends uniquely to an automorphism of $\PP\times \PP^\vee$. 
	In particular, there exists an injective map
			\beq\label{eq:isom.ext.Y}\Isom(Y(s),Y(s')) \;\hookrightarrow\; \Aut \PP\times \PP^\vee.\eeq  
	When $s=s'$, the map~\eqref{eq:isom.ext.Y} is a group homomorphism
	\beq\label{eq:aut.ext.Y}\Aut Y(s) \;\hookrightarrow\; \Aut \PP\times \PP^\vee,\eeq
	which sends $\Aut^+ Y(s)$ into $\Aut^+ \PP\times \PP^\vee$.
\end{lemma}

\begin{proof}
	The inclusion $Y(s)\hookrightarrow \PP\times \PP^\vee$ is given by the complete linear systems of $\cL_1$ and $\cL_n$ by Lemma~\ref{l.cohomY}(3). 
	Hence, by Lemma~\ref{l.pullback.Y}, any isomorphism $\varphi:Y(s)\to Y(s')$ induces an automorphism $\psi$ of $\PP\times \PP^\vee$ such that the diagram 
	\[ 
	\begin{tikzcd}
		Y(s) \arrow[r,hook]\arrow[d,"\varphi"']&\PP\times \PP^\vee\arrow[d,"\psi"]\\
		Y(s')\arrow[r,hook]&\PP\times \PP^\vee
	\end{tikzcd}
	\]
	commutes. This is induced by the pullback isomorphisms $\varphi^*:\coh^0(Y(s'),\cL_i)\xrightarrow{\cong}\coh^0(Y(s),\varphi^*\cL_i)$ for $i=1,n$. 
	Composing with $\iota$ if necessary, we may assume that
	\beq \label{eq:isom.g1g2} 
	\psi(F_1,F_{n-1})=(g_1F_1,g_2F_{n-1})\eeq
	 for some $g_1, g_2\in G$, or equivalently, $\psi\in \Aut^+\PP\times\PP^\vee$.
	 
	 Since $Y(s)$ surjects onto $\PP$ and $\PP^\vee$ respectively, $g_1$ and $g_2$ are uniquely determined by $\varphi$ up to scaling, so $(g_1,g_2)$ is unique as an element of $\Aut \PP\times \PP^\vee $. 
	 This shows that
	 $\varphi$ extends uniquely to an automorphism $\psi$ of $\PP\times\PP^\vee$,
	 and therefore the map \eqref{eq:isom.ext.Y} is well-defined. 
	 Injectivity follows immediately.
	
	When $s=s'$, 
    it follows from the uniqueness of $\psi$ and the form \eqref{eq:isom.g1g2} that the map \eqref{eq:aut.ext.Y} is a group homomorphism with respect to composition. Moreover, 
    it is clear that $\varphi\in \Aut^+Y(s)$ if and only if $\psi\in \Aut^+\PP\times\PP^\vee$.
\end{proof}

\medskip

Let $\Isom^+(Y(s),Y(s'))$ denote the preimage of $\Aut^+\PP\times\PP^\vee$ under \eqref{eq:isom.ext.Y}. Every element of this set is of the form \eqref{eq:isom.g1g2}. We characterize this set as follows.

\begin{proposition}\label{prop:ext.Y}
Let $s,s'\in \grs$, and let $\psi$ be an automorphism of $\PP\times\PP^\vee$
of the form \eqref{eq:isom.g1g2}.
Then $\psi\in \Isom^+(Y(s),Y(s'))$ if and only if there exists
$\begin{pmatrix} a & b \\ c & d \end{pmatrix}\in \GL_2$
such that $cs+d$ is invertible and
\[g_1 = g_2(cs+d), \qquad g_2^{-1}s'g_2 = (as+b)(cs+d)^{-1}\]
as $n\times n$ matrices.
\end{proposition}

\begin{proof} 
	The subvariety $Y(s)\subset \PP\times \PP^\vee$ is cut out by the condition
	$\langle \id, s\rangle F_1 \subset F_{n-1}$.
	Its image under the automorphism \eqref{eq:isom.g1g2} is therefore cut out by
	\[g_2\langle \mathrm{id}, s\rangle g_1^{-1} F_1 \subset F _{n-1},\]
    where, as before, $\langle x_1,\dots,x_r \rangle$ denotes the $\bk$-span of $x_1,\dots,x_r$ in $\mathfrak{g}$. 
	Hence, this image coincides with $Y(s')$ if and only if $\langle\mathrm{id}, s'\rangle=g_2\langle \mathrm{id}, s\rangle g_1^{-1}$, or equivalently,
    \[ 
    \langle\mathrm{id}, s\rangle=g_2^{-1}\langle \mathrm{id}, s'\rangle g_1=\langle g_2^{-1}g_1,~ g_2^{-1}s'g_1 \rangle\] 
    in $\mathrm{End}(\bk^n)$.
	This is equivalent to the existence of $\begin{pmatrix} a & b \\ c & d \end{pmatrix}\in \GL_2$
	such that
	\[ g_2^{-1}g_1 = cs+d \quad\text{and}\quad g_2^{-1}s'g_1 = as+b,\]
	that is, to the stated condition.
\end{proof}

We now prove Theorems~\ref{t.isoA}\eqref{t.isoAY} and \ref{t.autsA}\eqref{t.autsAY}.

\begin{proof}[Proof of Theorem~\ref{t.isoA}\eqref{t.isoAY}]
	The `if' direction is an immediate consequence of the `if' direction of  Proposition~\ref{prop:ext.Y}.
	
	Conversely, suppose that $Y(s)\cong Y(s')$.
	After composing with the involution~$\iota$ from Lemma~\ref{l.inv} if necessary,
	this isomorphism extends to an automorphism of $\PP\times\PP^\vee$
	of the form~\eqref{eq:isom.g1g2}.
	Proposition~\ref{prop:ext.Y} then shows that
	\[g_2^{-1}s'g_2 = (as+b)(cs+d)^{-1}\]
	for some $\begin{pmatrix} a & b \\ c & d \end{pmatrix}\in \GL_2$.
	This is exactly the desired condition.
\end{proof}

\begin{proof}[Proof of Theorem~\ref{t.autsA}\eqref{t.autsAY}]
	We show that $\Aut^+ Y(s)\cong K(s)/\mathbb{G}_m$.
	
	Consider the composition
	\[
	\Aut^+Y(s)\hookrightarrow \Aut^+(\PP\times\PP^\vee)=\Aut\PP\times\Aut\PP^\vee \xrightarrow{\,\mathrm{pr}_2\,} \Aut\PP^\vee=G/\mathbb{G}_m.
	\]
	By Proposition~\ref{prop:ext.Y}, the image of this map is precisely $K(s)/\mathbb{G}_m$.
	Thus it suffices to show that the composition is injective.
	
	Let $\varphi\in\Aut^+Y(s)$ be in the kernel.
	Then $\varphi$ is of the form
	\[\varphi(F_1,F_{n-1})=(g_1F_1,F_{n-1})\]
	for some $g_1\in G$.
	By Proposition~\ref{prop:ext.Y}, there exist $a,b,c,d\in\bk$ such that
	\[g_1 = cs+d \quad\text{and}\quad sg_1 = as+b.\]
	Since $n\geq 3$ and $s$ is regular, the elements $\id$, $s$, and $s^2$ are linearly independent in $\End(\bk^n)$.
	It follows that $a=d$ and $b=c=0$, hence $g_1$ is a scalar.
	Therefore $\varphi$ is trivial, proving injectivity.
\end{proof}

By the same line of argument as in the proof of Lemma~\ref{lem:unique.ext.Y}, we obtain the following result, which will be used
in the proofs of Theorems~\ref{t.isoA}\eqref{t.isoAX} and \ref{t.autsA}\eqref{t.autsAX}.

\begin{proposition}
\label{prop:unique.ext.X}
	Let $s,s'\in \grs$. For any isomorphism $\varphi:X(s)\to X(s')$, there exist a unique isomorphism $\varphi_Y\colon Y(s)\to Y(s')$ and a unique automorphism $\psi$ of $\PP\times \PP^\vee$ which make the diagram 
	\beq\label{eq:diagram.XY}
	\begin{tikzcd}
		X(s) \arrow[r,"\pi"]\arrow[d,"\cong","\varphi"'] 
                &Y(s) \arrow[r,hook]\arrow[d,"\cong", "\varphi_Y"']
                &\PP\times \PP^\vee\arrow[d,"\cong","\psi"']\\
		X(s')\arrow[r,"\pi"]&Y(s')\arrow[r,hook]&\PP\times \PP^\vee
	\end{tikzcd}
	\eeq
	commute. In particular, there is a natural map
	\beq\label{eq:isom.ext.X}\Isom(X(s),X(s')) \;\lra\; \Isom(Y(s),Y(s')).\eeq 
	When $s=s'$, the map \eqref{eq:isom.ext.X} is a group homomorphism 
	\beq \label{eq:aut.ext.X}
	\Aut X(s)\;\lra\;  \Aut Y(s)\subset  \Aut \PP\times \PP^\vee, \eeq
	which sends $\Aut^+X(s)$ into $\Aut^+Y(s)$,
	where the inclusion on the right is \eqref{eq:aut.ext.Y}.
\end{proposition}

\begin{proof}
    We first note that $\varphi_Y$ and $\psi$, if they exist, are uniquely determined by the commutativity of~\eqref{eq:diagram.XY}, since the morphism $\pi$ and the projections maps from $Y(s)$ to $\PP$ and $\PP^\vee$ are all surjective.
    Moreover, by definition, $\varphi\in \Aut^+ X(s)$ if and only if $\varphi_Y\in \Aut^+ Y(s)$. Thus it remains to prove the existence of $\varphi_Y$.

	Since $\cL_1$ and $\cL_n$ are globally generated, their complete linear systems define morphisms $X(s)\to \PP\times \PP^\vee$ and $X(s')\to \PP\times\PP^\vee$ by Lemma~\ref{l.cohomY}(3). 
	By the Pullback lemma (Lemma~\ref{l.pullback}), the isomorphism $\varphi:X(s)\to X(s')$ induces an automorphism $\psi\in \Aut\PP\times \PP^\vee$ such that the diagram
	\[
	\begin{tikzcd}
		X(s) \arrow[r]\arrow[d,"\cong","\varphi"'] &\PP\times \PP^\vee\arrow[d,"\cong","\psi"']\\
		X(s')\arrow[r]&\PP\times \PP^\vee
	\end{tikzcd}
	\]
	commutes.
	Since the restriction maps $\coh^0(\Fl_n,\cL_i)\to \coh^0(X(s),\cL_i)$ are isomorphisms for $i=1,n$, the morphism $X(s)\to \PP\times\PP^\vee$ is the restriction of the natural map $\Fl_n\to \PP\times\PP^\vee$ in \eqref{eq: projection}. In particular,
	it factors through $\pi$, and its image is precisely $Y(s)$. The same holds for $X(s')$, whose image is $Y(s')$.  
	Therefore, $\psi$ restricts to the desired isomorphism $\varphi_Y:Y(s)\to Y(s')$, completing the proof.
\end{proof}

By the same argument applied to $\Fl_n$, we obtain a  group homomorphism
\beq \label{eq:aut.ext.Fl} \Aut \Fl_n \;\lra \; \Aut \PP\times \PP^\vee. \eeq
Since $\Aut \Fl_n=(G/\mathbb{G}_m)\rtimes \langle \iota\rangle$, it follows that \eqref{eq:aut.ext.Fl} is, up to the involution $\iota$, induced by the diagonal map of $G/\mathbb{G}_m$. In particular, \eqref{eq:aut.ext.Fl} is injective,
and its image consists precisely of those automorphisms of $\PP\times \PP^\vee$ that preserve the incidence variety $\PFl_n\subset \PP\times \PP^\vee$.

Later, we will see that an analogous statement holds for $\Aut X(s)$: the homomorphism \eqref{eq:aut.ext.X} is injective, and its image consists exactly of those automorphisms of $Y(s)$ that preserve $\PFl_n$.
See Theorem~\ref{thm:XtoY} and Remark~\ref{rem:XtoY}.

\begin{remark} \label{rem A}
	Proposition~\ref{prop:ext.Y} shows that an isomorphism $\varphi_Y:Y(s)\to Y(s')$ 
	extends to an automorphism of $\PFl_n\subset \PP\times\PP^\vee$ if and only if $c=0$.
	In the next subsections, we show that this condition is also necessary for $\varphi_Y$ to admit a lift to an isomorphism $\varphi:X(s)\to X(s')$ in~\eqref{eq:diagram.XY}.
	See Lemma~\ref{l.hom} and Theorem~\ref{thm:XtoY}.
\end{remark}

\bigskip

\subsection{$\mathrm{Hom}$ spaces of vector bundles on $Y(s)$} \label{ss Lemma} 
Let $\mathcal F_1$ and $\mathcal F_{n-1}$ be the tautological vector bundles  on $\PFl_n$ whose fibers at $(F_1\subset F_{n-1})\in \PFl_n$ 
are $F_1$ and $F_{n-1}$, respectively.  We use the same notation for their restrictions to $Y(s)$. In particular,
\beq\label{eq:taut}\cL_1^\vee\cong \cF_1 \quad \text{ and }\quad \cL_n\cong \bk^n/\cF_{n-1}\eeq 
where $\bk^n$ denotes the trivial vector bundle of rank $n$.  

The following lemma will be crucial in the proofs of Theorems~\ref{t.isoA}\eqref{t.isoAX} and  \ref{t.autsA}\eqref{t.autsAX}. 
Let $\cE$ denote the quotient bundle $\cF_{n-1}/\cF_1$ on $\PFl_n$, or its restriction to $Y(s)$.

\begin{definition}
	We say that an isomorphism $\varphi_Y:Y(s)\to Y(s')$ \emph{preserves $\PFl_n$} if its
(unique) extension to an automorphism of $\PP\times\PP^\vee$ sends $\PFl_n$ to itself.
\end{definition}

\begin{lemma} \label{l.hom}
	Let $s,s'\in \grs$, and let $\varphi_Y\in \Isom^+(Y(s),Y(s'))$. Then
	\[\mathrm{Hom}_{Y(s)}\big(\cE,\;\varphi_Y^*\cE\big)=\begin{cases}
		\bk & \text{ if }\varphi_Y \text{ preserves }\PFl_n,\\ 
		0 & \text{ otherwise.}
	\end{cases}\]
	In particular, $\cE\cong \varphi_Y^*\cE$ if and only if $\varphi_Y$ preserves $\PFl_n$. 
	Moreover,  $\cE$ is simple. 
\end{lemma}
\begin{proof}
	By Proposition~\ref{prop:ext.Y},
	$\varphi_Y$ is
	of the form
	\[\varphi_Y(F_1, F_{n-1})=(g(cs+d)F_1, gF_{n-1}).\]
	As the diagonal action of $g\in G$ preserves $\cE$ and $\PFl_n$,
we may assume $g=\id$. 

Let $\alpha:=cs+d$. Then $\alpha \cF_1\subset \cF_{n-1}$ on $Y(s)$. Let $\cE_\alpha:=\cF_{n-1}/\alpha \cF_1$ denote the corresponding quotient. 
It suffices to show the following.
	\[\mathrm{Hom}_{Y(s)}(\cE,\;\cE_\alpha)=\begin{cases}
		\bk & \text{ if }\alpha\in \langle \id\rangle, \\
		0 & \text{ otherwise,}
	\end{cases}\]
	and $\cE$ and $\cE_\alpha$ are isomorphic if and only if $\alpha\in \langle \id \rangle $.
The last assertion follows immediately from the computation of $\Hom_{Y(s)}(\cE,\cE_\alpha)$.

In the remainder of the proof, all $\Hom$-spaces are taken over $Y(s)$.

We compute the $\Hom$-space using the long exact sequences of $\Ext$-groups associated to the
short exact sequences
	\beq \label{ses:hom1} 0 \lra \cF_{n-1}\lra \bk^n \lra \cL_n\lra 0, \eeq 
	\beq \label{ses:hom2} 0 \lra \cL_1^\vee \lra \cF_{n-1}\lra \cE\lra 0, \eeq
	\beq \label{ses:hom3} 0 \lra \cL_1^\vee \xrightarrow{~\alpha~} \cF_{n-1}\lra \cE_{\alpha}\lra 0 \eeq
	on $Y(s)$.
	Applying $\Hom(-, \cE_{\alpha})$ to \eqref{ses:hom2}, we obtain
	\beq \label{les1}
	0\lra \Hom(\cE,\cE_{\alpha})\lra \Hom(\cF_{n-1},\cE_{\alpha})\lra \Hom(\cL_1^\vee,\cE_{\alpha}).
	\eeq
	
	We claim that:
	\begin{enumerate}
		\item[(i)] $\Hom(\cF_{n-1},\cE_{\alpha}) \cong \bk$, generated by the canonical quotient map $\cF_{n-1}\to \cE_\alpha$; 
		\item[(ii)] there is a canonical isomorphism $\Hom(\cL_1^\vee,\cE_{\alpha})\cong\langle \id,s\rangle/\langle \alpha\rangle$; so that 
		\item[(iii)] under these identifications, the last map in \eqref{les1} sends the generator of $\Hom(\cF_{n-1},\cE_{\alpha})$ to the class of $\id$. 
	\end{enumerate}
	
It follows that $\Hom(\cE,\cE_\alpha)$ is nonzero if and only if
$\langle\alpha\rangle=\langle\id\rangle$, in which case it is one-dimensional.
This completes the proof.
	\medskip

	(i) To show that $\Hom(\cF_{n-1},\cE_{\alpha})=\bk$, 
	we apply $\Hom(\cF_{n-1},-)$ to \eqref{ses:hom3} and obtain the exact sequence
	\[\Hom(\cF_{n-1},\cL_1^\vee)\lra \Hom(\cF_{n-1},\cF_{n-1})\lra \Hom(\cF_{n-1},\cE_{\alpha})\lra \Ext^1(\cF_{n-1},\cL_1^\vee).\]
	Then it suffices to show that $\Ext^i(\cF_{n-1},\cL_1^\vee)=0$ for all $i$ and $\Hom(\cF_{n-1},\cF_{n-1})=\bk$.
	We first show the vanishing. Applying $R\Hom(-,\cL_1^\vee)$ to \eqref{ses:hom1}, we obtain
	\[\Ext^i(\bk^n,\cL_1^\vee) \lra \Ext^i(\cF_{n-1},\cL_1^\vee)\lra \Ext^{i+1}(\cL_n,\cL_1^\vee).\]
	Both outer terms vanish for all $i$ 
	by Lemma~\ref{l.cohomY}, thus $\Ext^i(\cF_{n-1},\cL_1^\vee)=0$ for all $i$.
	
	Next, to compute $\Hom(\cF_{n-1},\cF_{n-1})$, we apply $\Hom(\cF_{n-1},-)$ to \eqref{ses:hom1} and get
	\[\begin{split}
		0&\lra \Hom(\cF_{n-1},\cF_{n-1})\lra \Hom(\cF_{n-1},\bk^n)\lra \Hom(\cF_{n-1},\cL_n).
	\end{split}
	\]
	Thus it suffices to show that the last map identifies the natural surjection
	\beq\label{eq: surj} \mathrm{End}(\bk^n)\lra \mathrm{End}(\bk^n)/\langle\id\rangle.\eeq
	To see that $\Hom(\cF_{n-1},\cL_n)\cong \mathrm{End}(\bk^n)/\langle \id \rangle$, we take $\Hom(-,\cL_n)$ to \eqref{ses:hom1}:
	\[0\lra \Hom(\cL_n,\cL_n)\lra \Hom(\bk^n,\cL_n)\lra \Hom(\cF_{n-1},\cL_n)\lra \Ext^1(\cL_n,\cL_n)\]
	which then reads as
	\[0\lra \bk\cdot \id\lra \mathrm{End}(\bk^n)\lra \Hom(\cF_{n-1},\cL_n)\lra 0 \]
	since $\Hom(\bk^n,\cL_n)\cong \mathrm{End}(\bk^n)$  
	and $\Ext^i(\cL_n,\cL_n)$ is $\bk$ for $i=0$ and $0$ otherwise.
	
	To see that $\Hom(\cF_{n-1},\bk^n)\cong \mathrm{End}(\bk^n)$, we apply $\Hom(-,\bk^n)$ to \eqref{ses:hom1} and get
    \beq\label{les2} 
    \begin{tikzcd}[row sep = 0.2cm, column sep = 1.5em]
        \Hom(\cL_n,\bk^n) \rar \arrow[d, equal] & \Hom(\bk^n,\bk^n) \rar & \Hom(\cF_{n-1},\bk^n) \rar & \Ext^1(\cL_n,\bk^n) \arrow[d, equal] \\
        0 &&& 0
    \end{tikzcd}
    \eeq
	where the terms at the left and right ends vanish by Lemma~\ref{l.cohomY}(1). It follows that
	\[\Hom(\cF_{n-1},\bk^n)\;\cong\;\End(\bk^n).\]
	Combining this with the identification
	$\Hom(\cF_{n-1},\cL_n)\cong \End(\bk^n)/\langle \id\rangle$,
	we see that $\Hom(\cF_{n-1},\cF_{n-1})$ is the kernel of
	\eqref{eq: surj}, and hence is one-dimensional.
	This proves
	\beq \label{eq:hom} \Hom(\cF_{n-1},\cE_{\alpha})\cong \Hom(\cF_{n-1},\cF_{n-1})=\bk,\eeq
	which is generated by the canonical quotient map $\cF_{n-1}\to \cF_{n-1}/\alpha \cF_1$.
	This completes the proof of (i).

	\medskip
	
	(ii) Applying $\Gamma(\cL_1\otimes -)$ to \eqref{ses:hom3}, we obtain the exact sequence
	\beq \label{22}0\lra \coh^0(Y(s),\cO_{Y(s)})\xrightarrow{~\alpha~} \coh^0(Y(s),\cL_1\otimes \cF_{n-1})\lra \coh^0(Y(s),\cL_1\otimes \cE_{\alpha})\lra 0.\eeq
	Thus, it suffices to show that there is a canonical isomorphism
	\beq\label{eq:hom2} \coh^0(Y(s),\cL_1\otimes \cF_{n-1} )\cong \langle \id, s\rangle \subset \mathrm{End}(\bk^n)\eeq
	under which $\coh^0(\cO_{Y(s)})=\bk$ embeds as the subspace $\langle \alpha\rangle$. 
	
	Applying $\Gamma(\cL_1\otimes -)$ to \eqref{ses:hom1}, we obtain
	\[0\lra \coh^0(Y(s),\cL_1\otimes \cF_{n-1})\lra \coh^0(Y(s),\cL_1\otimes \bk^n) \lra \coh^0(Y(s),\cL_1\otimes \cL_n),\]
	where the last map is canonically identified, by Lemma~\ref{l.cohomY}, with
	\[\mathrm{End}(\bk^n)\lra \mathrm{End}(\bk^n)/\langle \id, s\rangle.\]
	Therefore, the kernel $\coh^0(\cL_1\otimes \cF_{n-1})$ is canonically isomorphic to $\langle \id, s\rangle\subset \mathrm{End}(\bk^n)$.
	
	(iii) Consider the canonical identification
	\beq \label{20}\End(\bk^n)\xrightarrow{~\cong~}\Hom (\cF_{n-1},\bk^n)\xrightarrow{~\cong~}\Hom(\cL_1^\vee,\bk^n)\eeq
	induced by the inclusions $\cL_1^\vee\cong \cF_1 \subset \cF_{n-1}\subset \bk^n$. Here the first map is an isomorphism by~\eqref{les2}, and the composition is an isomorphism by Lemma~\ref{l.cohomY}(3). 
	 In particular, the second isomorphism preserves the subspace $\langle\id\rangle\subset \End(\bk^n)$.
	  
	 By \eqref{eq:hom} and \eqref{eq:hom2}, under the identification~\eqref{20} we have
	 \[\Hom(\cF_{n-1},\cE_\alpha)=\langle \id\rangle  \quad \text{ and } \quad \Hom(\cL_1^\vee,\cF_{n-1})=\langle\id,s\rangle.\] 
	 Therefore, the image of the natural inclusion
	\beq \label{21}\Hom(\cF_{n-1},\cE_\alpha)\cong \Hom(\cF_{n-1},\cF_{n-1})\lra \Hom(\cL_1^\vee,F_{n-1})\eeq
	is precisely the subspace $\langle \id\rangle$.
	Finally, the rightmost map in~\eqref{les1} factors as
	\[\Hom(\cF_{n-1},\cE_\alpha)\xrightarrow{\eqref{21}}\Hom(\cL_1^\vee,\cF_{n-1})\lra \Hom(\cL_1^\vee,\cE_\alpha)\cong \langle \id,s\rangle/\langle\alpha\rangle\]
	where the second map is the third map in \eqref{22}. Hence its image is generated by the class of $\id$, as claimed.
	\end{proof}

\subsection{Flag bundles and their isomorphisms}  
\label{ss.proof.isoAX} 
We begin by recalling some basic facts about flag bundles and their isomorphisms.
 
For a vector bundle $\cV$ of rank $r$ over a scheme $Z$, we write
\[\pi_\cV:\Fl(\cV)\lra Z\]
for the associated flag bundle, whose fiber over a point $z\in Z$
is the flag variety $\Fl(\cV|_z)\cong \Fl_r$.
(See~\cite[Section~4.2]{AndersonFulton} or \cite[Chapter~14]{fulton-intersection-theory}.)
For example, $X(s)\to Y(s)$ in \eqref{eq:projection.pi} is the flag bundle associated to the vector bundle $\cE=\cF_{n-1}/\cF_1$ on $Y(s)$.
\smallskip

The projection map $\pi_\cV$ admits two natural sequences of projective bundles
\[\Fl(\cV)\lra \cdots\lra \PP(\cV)\lra Z \quad \text{and}\quad \Fl(\cV)\lra \cdots \lra \PP(\cV^\vee)\lra Z\]
by successively forgetting the subspaces of rank $r-1,r-2\dots,1$ and $1,2,\dots,r-1$ respectively. We denote by 
\[\cL_1(\cV):=\cO_{\PP(\cV)/Z}(1)|_{\Fl(\cV)} \quad \text{and} \quad \cL_r(\cV):=\cO_{\PP(\cV^\vee)/Z}(1)|_{\Fl(\cV)}\]
the line bundles on $\Fl(\cV)$ that are obtained as the pullbacks of $\cO_{\PP(\cV)/Z}(1)$ and $\cO_{\PP(\cV^\vee)/Z}(1)$ respectively.
 Among the universal subbundles and quotient bundles over $\Fl(\cV)$,  $\cL_1(\cV)^\vee$ and $\cL_r(\cV)$ are precisely those of rank one.
These line bundles recover $\cV$ in the sense that there are canonical isomorphisms
\beq\label{eq:pushforward.flag}\pi_{\cV*}\cL_1(\cV)\cong\cV^\vee \quad \text{and}  \quad\pi_{\cV*}\cL_r(\cV)\cong\cV,\eeq 
generalizing \eqref{eq:O1}. More generally, we have the following. Let $\cF_j(\cV)$ be the universal subbundles of $\pi_\cV^*\cV$ of rank $j$ and let $\cQ_j(\cV):=\pi_\cV^*\cV/\cF_{r-j}(\cV)$ be the universal quotient bundles of rank $j$. Then there are canonical isomorphisms
\beq\label{eq:pushforward.flag2}\pi_{\cV*}(\cF_j(\cV)^\vee)\cong \cV^\vee\quad\text{and}\quad\pi_{\cV*}\cQ_j(\cV)\cong \cV\eeq
for any $1\leq j\leq r$,
due to the Borel--Weil--Bott theorem iteratedly applied to the short exact sequences 
\[
\begin{split}
	&0\lra(\cF_j(\cV)/\cF_{j-1}(\cV))^\vee\lra \cF_j(\cV)^\vee\lra \cF_{j-1}(\cV)^\vee\lra 0 \quad \text{ and}\\
	&0\lra \cF_{r-j+1}(\cV)/\cF_{r-j}(\cV)\lra \cQ_j(\cV)\lra \cQ_{j-1}(\cV)\lra 0,
\end{split}\]
where $\cF_j(\cV)/\cF_{j-1}(\cV)$ is a line bundle isomorphic to $\cL(-e_j)$ on fibers $\Fl_r$. Indeed, one can immediately check that $\coh^i(\Fl_r,\cL(e_j))=0=\coh^i(\Fl_r,\cL(-e_{r-j+1}))$ for all $i$ and all $j>1$, and $\coh^0(\Fl_r,\cL(e_1))=(\bk^r)^\vee$ and $\coh^0(\Fl_r,\cL(-e_r))\cong \bk^r$ canonically. 
The isomorphisms in \eqref{eq:pushforward.flag} are then precisely those in \eqref{eq:pushforward.flag2} when $j=1$.

\medskip
Given two vector bundles $\cV$ and $\cV'$ of rank $r$ over $Z$, there are two
natural ways in which an isomorphism $\Fl(\cV)\to \Fl(\cV')$ over $Z$ arises.

First, any isomorphism $f:\cV\to \cV'\otimes L$ for a line bundle $L$ on $Z$
induces an isomorphism $\varphi_f:\Fl(\cV)\to \Fl(\cV')$ over $Z$, obtained
fiberwise by applying $f$ to each subspace of a flag.

Second, any isomorphism $f:\cV^\vee\to \cV'\otimes L$ induces an isomorphism
$\varphi_f^\vee:\Fl(\cV)\to \Fl(\cV')$, obtained by sending a flag to the flag
of orthogonal complements with respect to the corresponding nondegenerate pairing $\cV\otimes \cV'\to L^\vee$.

We now show that these are the only possible ways in which an isomorphism
$\Fl(\cV)\to \Fl(\cV')$ over $Z$ can arise, using the above properties of
$\cL_1(\cV)$ and $\cL_r(\cV)$.

\begin{proposition}\label{prop:flagbundles}
	Let $Z$ be a connected reduced locally Notherian scheme. Let $\cV$ and $\cV'$ be vector
	bundles of rank $r$ over $Z$.
	Then there exists an isomorphism
	\[\varphi:\Fl(\cV)\lra \Fl(\cV')\] over $Z$
	if and only if there exist a line bundle $L$ on $Z$ and an isomorphism
	\beq\label{eq:isom.vb}
	f:\cV\lra\cV'\otimes L \quad \text{or} \quad f: \cV^\vee\lra \cV'\otimes L\eeq
	that induces $\varphi$,
	that is, $\varphi=\varphi_f$ or $\varphi=\varphi_f^\vee$.
\end{proposition}

\begin{proof}
	As the `if' direction is vacuous, we prove the `only if' direction. 
	
	Suppose that we have an isomorphism $\varphi:\Fl(\cV)\lra \Fl(\cV')$ over $Z$. 
    For each point $z\in Z$, $\varphi$ restricts to an isomorphism $\Fl(\cV|_z) \to \Fl(\cV'|_z)$.
	
	By definition, the line bundle $\cL_1(\cV')$ restricts to $\cL_1$
	on the flag variety $\Fl(\cV'|_z)=\Fl_r$.
	Therefore, the pullback $\varphi^*\cL_1(\cV')$ restricts to either
	$\cL_1$ or $\cL_r$ on each fiber $\Fl(\cV|_z)=\Fl_r$, since any
	automorphism of $\Fl_r$ preserves the set
	$\{\cL_1,\cL_r\}$.
	
	Since $Z$ is connected, this choice is constant over $Z$.
	It then follows from the general fact that $\mathrm{Pic}(X)\cong \mathrm{Pic}(Y)\oplus \mathbb{Z}\cdot \cO_\rho(1)$ for a projective bundle $\rho:X\to Y$ (see \cite[Exercise~25.1.K]{vakil}) 
	that there exists a line bundle $L$ on $Z$ such that either
	\[
	\varphi^*\cL_1(\cV') \cong \cL_1(\cV)\otimes \pi_\cV^*L
	\quad \text{or} \quad
	\varphi^*\cL_1(\cV') \cong \cL_r(\cV)\otimes \pi_\cV^*L.
	\]
	
	Pushing forward to $Z$ and using \eqref{eq:pushforward.flag}, we obtain
	an isomorphism $f$ as in \eqref{eq:isom.vb}.
	Let $\psi$ be the induced isomorphism $\varphi_f$ or $\varphi_f^\vee$.
	By construction, $\varphi$ and $\psi$ agree on every geometric fiber of $\Fl(\cV)\to Z$,
	and hence on all geometric points of $\Fl(\cV)$.
	Since $\Fl(\cV')$ is separated over $Z$, the equalizer of $\varphi$ and $\psi$
	is a closed subscheme of $\Fl(\cV)$ (\cite[\href{https://stacks.math.columbia.edu/tag/01KM}{Tag 01KM}]{stacks-project}).
	As $\Fl(\cV)$ is reduced and the equalizer contains all geometric points,
	it follows that $\varphi=\psi$ (cf.~\cite[Exercise~11.3.B]{vakil}).
\end{proof}

We use Lemma~\ref{l.hom} and Proposition~\ref{prop:flagbundles} to characterize the map \eqref{eq:isom.ext.X}.

\begin{theorem}\label{thm:XtoY}
	Let $s,s'\in \grs$. The map $\Isom(X(s),X(s'))\to \Isom(Y(s),Y(s'))$ defined in \eqref{eq:isom.ext.X} is injective, and its image consists precisely of those isomorphisms $Y(s)\to Y(s')$ that preserves $\PFl_n$.
	In other words, we have the fiber diagram
	\[
	\begin{tikzcd}
		\Isom(X(s),X(s')) \arrow[r,hook]\arrow[d,hook'] &\Aut\Fl_n \arrow[d,hook']\\
		\Isom(Y(s),Y(s'))\arrow[r,hook]&\Aut\PP\times \PP^\vee.
	\end{tikzcd}
	\]
	Moreover, given an isomorphism $\varphi:X(s)\to X(s')$, there exists a unique automorphism of $\Fl_n$ that restricts to $\varphi$.
\end{theorem}
\begin{proof}
	Let $\Isom^+(X(s),X(s'))$ be the preimage of $\Isom^+(Y(s),Y(s'))$. 
	It suffices to prove the theorem for the induced map 
	\beq\label{eq:isom.ext+}\Isom^+(X(s),X(s'))\to \Isom^+(Y(s),Y(s'))\eeq 
	since the general case follows by composing with $\iota$.
	
	Let $\varphi_Y$ be an element in the image of \eqref{eq:isom.ext+}. This means that there exists $\varphi\in \Isom^+(X(s),X(s'))$ fitting into a commutative diagram
	\[\begin{tikzcd}
		X(s) \arrow[d,"\pi"']\arrow[r,"\varphi"] 
                &X(s') \arrow[d,"\pi"]\\
		Y(s)\arrow[r,"\varphi_Y"]&Y(s')
	\end{tikzcd}\]
	where $\pi$ are the flag bundle maps in \eqref{eq:projection.pi}.
	Let $\cE=\cF_{n-1}/\cF_1$ as before. 
	Pulling back along $\varphi_Y$ and using the above diagram,
	we obtain an isomorphism $\Fl(\cE)\to \Fl(\varphi_Y^*\cE)$ of flag bundles over $Y(s)$. 
	By Proposition~\ref{prop:flagbundles}, we have either
	\[
	\cE \;\cong\; \varphi_Y^*\cE \otimes L
	\quad \text{or} \quad
	\cE^\vee \;\cong\; \varphi_Y^*\cE \otimes L
	\]
	for some line bundle $L$ on $Y(s)$.

	On the other hand, $\varphi_Y^*\cE$ and $\cE$ have the same Chern classes.
	By \eqref{eq:taut}, we have
	\beq\label{eq:chern}
	c_1(\cE) = c_1(\cL_1)-c_1(\cL_n).
	\eeq
	If $\cE^\vee \cong \varphi_Y^*\cE \otimes L$, then comparing first Chern classes yields
	\[-2c_1(\cE) = \rank(\cE)\cdot c_1(L)= (n-2)\cdot c_1(L).\]
	Since $n \ge 4$ and $c_1(\cE)$ is primitive by \eqref{eq:chern}, this forces $n=4$.
	In that case, $\cE$ has rank two, and hence
	$\cE^\vee \cong \cE \otimes \det(\cE)^\vee$.
	Therefore, whether $n=4$ or $n>4$, we may assume that
	$\cE \cong \varphi_Y^*\cE \otimes L$.

	Comparing first Chern classes again, we conclude $c_1(L)=0$.
	Since $Y(s)$ is cellular, this implies that $L$ is trivial.
	Thus we obtain $\cE\cong \varphi_Y^*\cE$.
	By Lemma~\ref{l.hom}, this implies that $\varphi_Y$ preserves $\PFl_n$.
	Moreover, Lemma~\ref{l.hom} also shows that $\cE$ is simple, so the isomorphism
	$\cE\cong \varphi_Y^*\cE$ is unique up to scaling.

	Therefore, by Proposition~\ref{prop:flagbundles}, the induced isomorphism
	$\Fl(\cE)\to \Fl(\varphi_Y^*\cE)$ is unique, and consequently the lift $\varphi$
	is unique. This proves that \eqref{eq:isom.ext+} is injective and that its image
	is contained in the set of $\varphi_Y$ preserving $\PFl_n$.

	Conversely, suppose that $\varphi_Y:Y(s)\to Y(s')$ preserves $\PFl_n$.
	Then $\varphi_Y$ extends to an automorphism $\hat\varphi$ of $\PFl_n$, hence it lies in the image of \eqref{eq:aut.ext.Fl}.
	Let $\widetilde\varphi\in\Aut\Fl_n$ be the unique automorphism inducing $\hat\varphi$,
	so that $\pi\circ\widetilde\varphi=\hat\varphi\circ\pi$.
	Since $X(s)=\pi^{-1}(Y(s))$ and $X(s')=\pi^{-1}(Y(s'))$, it follows that
	$\widetilde\varphi(X(s))=X(s')$.
	So, restricting $\widetilde\varphi$ to $X(s)$ yields an isomorphism
	$X(s)\to X(s')$, which induces $\varphi_Y$. 

	Finally, the last assertion follows from the surjectivity of $X(s)\to \PP$, since
	an isomorphism $\varphi\in \Isom^+(X(s),X(s'))\subset \Aut^+\Fl_n=G/\mathbb{G}_m$ is determined by the image of $F_1$
	for $F_\bullet\in X(s)$.	
\end{proof}
\begin{remark}\label{rem:XtoY}
	When $s=s'$, this shows that $\Aut X(s)\to \Aut Y(s)$ in \eqref{eq:aut.ext.X} is an injective group homomorphism that fits into the fiber diagram
	\[
	\begin{tikzcd}
		\Aut X(s) \arrow[r,hook]\arrow[d,hook'] &\Aut\Fl_n \arrow[d,hook']\\
		\Aut Y(s)\arrow[r,hook]&\Aut\PP\times \PP^\vee.
	\end{tikzcd}
	\]
\end{remark}

\medskip

Now we prove Theorems~\ref{t.isoA}\eqref{t.isoAX}~and~\ref{t.autsA}\eqref{t.autsAX}.

\begin{proof}[Proof of Theorem~\ref{t.isoA}\eqref{t.isoAX}] 
The `if' direction follows from the discussion immediately following Theorem~\ref{t.isoA}.
For the `only if' direction, let $\varphi\colon X(s) \to X(s')$ be an isomorphism. 
Composing $\varphi$ with $\iota$ if necessary, we may assume that $\varphi\in \Isom^+(X(s),X(s'))$. 

By Theorem~\ref{thm:XtoY}, the induced isomorphism $\varphi_Y:Y(s)\to Y(s')$ preserves $\PFl_n$.
By Remark~\ref{rem A}, this means that in the description of
Proposition~\ref{prop:ext.Y} we must have $c=0$.
Hence $g^{-1}s'g=as+b$ for some $g\in G$ and $(a,b)\in \Aff$.
\end{proof}

\medskip

\begin{proof}[Proof of Theorem~\ref{t.autsA}\eqref{t.autsAX}]
By Theorem~\ref{thm:XtoY}, $\Aut X(s)\to \Aut Y(s)$ is injective.
It suffices to identify its image under the identification
$\Aut Y(s)=(K(s)/\mathbb{G}_m)\rtimes\langle \iota\rangle$.

By Lemma~\ref{l.hom}, the image is contained in $(H(s)/\mathbb{G}_m)\rtimes \langle\iota\rangle$.
On the other hand, by the discussion immediately following Theorem~\ref{t.isoA} and
by Lemma~\ref{l.inv}, restriction defines a homomorphism
\[
\mathrm{res}:(H(s)/\mathbb{G}_m)\rtimes \langle \iota \rangle \lra \Aut X(s),
\qquad \varphi \mapsto \varphi|_{X(s)},
\]
whose composition with \eqref{eq:aut.ext.X} is the inclusion into $\Aut Y(s)$.
Therefore the homomorphism $\mathrm{res}$ is injective, and the image of \eqref{eq:aut.ext.X}
is exactly $(H(s)/\mathbb{G}_m)\rtimes \langle\iota\rangle$.
\end{proof}

These complete the proofs of Theorems~\ref{t.isoA} and~\ref{t.autsA}.

\subsection{Extension to other generalized flag varieties}\label{ss:partial.A}
Our main results for $X(s)$ and $Y(s)$ extend to regular semisimple Hessenberg varieties of codimension one in other generalized (or partial) flag varieties.
For the interested reader, we describe such extended results in this subsection.
We fix $n\geq4$ throughout this subsection.

Let $I\subset[n-1]=\{1,\dots,n-1\}$ be a nonempty subset. Explicitly, we write
\[I=\{i_1,\dots, i_m\} \quad \text{with }~1\leq i_1<\cdots<i_m<n.\]
We denote the corresponding partial flag variety by 
\[\Fl_I:=\{F_{i_1}\subset\cdots \subset F_{i_m}\subset \bk^n:\dim F_{i}=i ~\text{ for }i\in I\}.\]
In our notation, $\Fl_I=\Fl_n$ when $I=[n-1]$ and $\Fl_I=\PFl_n$ when $I=\{1,n-1\}$. In general, when $I\subset J\subset[n-1]$, there is a natural projection $\Fl_J\to \Fl_I$. 

The Picard group $\Fl_I$ can be identified with a subgroup of $\mathrm{Pic}\Fl_n$ via the pullback homomorphism, that consists of $\cL(\lambda)$ with $\lambda=(\lambda_1,\dots,\lambda_n)$ such that
\[\lambda_i=\lambda_{i+1} \quad \text{ for }~i\in [n-1]-I.\]
In other words,
$\mathrm{Pic}\Fl_I= \left\{\cL(\lambda):s_i(\lambda)=\lambda ~\text{ for }i\in [n-1]-I\right\}$.

We say that $I$ is \emph{Dynkin-symmetric} if the following holds:
\[i\in I \quad \text{if and only if}\quad n-i\in I.\]
Then, by the work of Demazure \cite{demazure}, we know that
\[\Aut \Fl_I =\begin{cases}
	(G/\mathbb{G}_m) \rtimes \langle\iota\rangle &\text{if $I$ is Dynkin-symmetric}\\
	G/\mathbb{G}_m &\text{otherwise,}
\end{cases}\]
where $G/\mathbb{G}_m$ consists of the inner automorphisms and $\iota$ sends the partial flag $F_{i_1}\subset \cdots \subset F_{i_m}\subset \bk^n$ to $F_{i_m}^\perp\subset \cdots \subset F_{i_1}^\perp\subset \bk^n$. Here, the orthogonal complement is taken with respect to any given nondegenerate bilinear form.
In particular, in the latter case, any automorphism of $\Fl_I$ preserves every line bundle on $\Fl_I$.

When $I$ is Dynkin-symmetric, $\iota^*$ interchanges the line bundles $\cL(\lambda_1,\dots,\lambda_n)$ and $\cL(-\lambda_n,\dots,-\lambda_1)$.
In particular, when $1,n-1\in I$, it interchanges $\cL_1$ and $\cL_n$.  

When $I$ is Dynkin-symmetric and $i,n-i\in I$, one can also check that the pullback $\iota^*$ interchanges the dual of the universal subbundle $\cF_i$ of rank $i$ and the universal quotient bundle $\cQ_i=\bk^n/\cF_{n-i}$ of rank $i$.

From now on, we suppose that $1,n-1\in I$.
For $s\in \fg$, the associated Hessenberg variety in $\Fl_I$ of codimension one is
defined as
\[X_I(s):=\{F_\bullet\in \Fl_I:sF_1\subset F_{n-1}\}.\]
This variety is a Cartier divisor in $\Fl_I$ with the associated line bundle $\cL(\theta)$ and is smooth when $s\in \grs$, as $X(s)$ is in $\Fl_n$. When $I$ is Dynkin-symmetric, the involution $\iota\in \Aut\Fl_I$ can be chosen to preserve $X_I(s)$ (cf.~Lemma~\ref{l.inv}).

\smallskip

We prove the following generalization of Theorems~\ref{t.isoA}\eqref{t.isoAX} and~\ref{t.autsA}\eqref{t.autsAX}.
\begin{theorem}\label{thm:partial.A}
Let $n\geq4$ and let $I\subset[n-1]$ be a subset properly containing $\{1,n-1\}$, that is, $\{1,n-1\}\subsetneq I$. Let $s,s'\in \grs$. Then, the varieties $X_I(s)$ and $X_I(s')$ are isomorphic if and only if there exist $g\in G$ and $(a,b)\in \Aff$ such that \[gs'g^{-1}=as+b.\]
Moreover, every isomorphism $X(s)\to X(s')$ extends uniquely to an automorphism of $\Fl_I$.
In particular, the automorphism group of $X_I(s)$ is
	\[\Aut X_I(s)=\begin{cases}
		(H(s)/\mathbb{G}_m)\rtimes\langle\iota\rangle &\text{if $I$ is Dynkin-symmetric}\\
		H(s)/\mathbb{G}_m &\text{otherwise}.
	\end{cases}\]
\end{theorem}
The proof is structurally the same as those of Theorems~\ref{t.isoA}\eqref{t.isoAX} and~\ref{t.autsA}\eqref{t.autsAX}. We need generalizations of Lemma~\ref{l.pullback} and Proposition~\ref{prop:unique.ext.X} to $X_I(s)$, and of Proposition~\ref{prop:flagbundles} to partial flag bundles.
We state and prove them below.

Let $\Aut^+X_I(s)\subset \Aut X_I(s)$ denote the subgroup consisting of automorphisms that preserve $\cL_1$ and $\cL_n$. The following lemma generalizes Lemmas~\ref{l.pullback} and~\ref{l.pullback.Y}.
\begin{lemma}\label{lem:pullback.partial}
	Let $n\geq4$ and let $\{1,n-1\}\subset I\subset [n-1]$. Let $s\in\grs$. Then,
	\[\Aut X_I(s)=\begin{cases}
		\Aut^+ X_I(s)\rtimes \langle\iota\rangle &\text{if $I$ is Dynkin-symmetric}\\
		\Aut^+ X_I(s)&\text{otherwise.}
	\end{cases}\]
\end{lemma}
\begin{proof}
	The proof is essentially the same as that of Lemma~\ref{l.pullback.Y}.
	
	The Picard groups of $X_I(s)$ and $\Fl_I$ are isomorphic via the pullback, and they are embedded into $\mathrm{Pic}X(s)=\mathrm{Pic}\Fl_n$ via pullbacks. Moreover, by Lemma~\ref{lem:projection.formula} applied to $X(s)\to X_I(s)$, there is a canonical isomorphism 
	\beq\label{eq:cohom.XtoXI}\coh^i(X_I(s),\cL(\lambda))\cong \coh^i(X(s),\cL(\lambda))\eeq 
	for any $i$ and line bundle $\cL(\lambda)\in \mathrm{Pic}X_I(s)=\mathrm{Pic}\Fl_I$.
	Hence the assertion follows immediately from the cohomological characterization of $\cL_1$ and $\cL_n$ in Lemma~\ref{l.CohomChar}, by the same argument as in the proof of Lemma~\ref{l.pullback}.
\end{proof}

The following is a generalization of Proposition~\ref{prop:unique.ext.X}. We denote by $\pi_I:X_I(s)\to Y(s)$ the natural projection map $(F_{i_1}\subset \cdots \subset F_{i_m})\mapsto (F_1\subset F_{n-1})$. 
\begin{proposition}\label{prop:unique.ext.XI}
	Let $n\geq4$  and let $\{1,n-1\}\subset I\subset [n-1]$. Let $s,s'\in \grs$. For an isomorphism $\varphi:X_I(s)\to X_I(s')$, there exist a unique isomorphism $\varphi_Y\colon Y(s)\to Y(s')$ and a unique automorphism $\psi$ of $\PP\times \PP^\vee$ which make the diagram 
	\[
	\begin{tikzcd}
		X_I(s) \arrow[r,"\pi_I"]\arrow[d,"\cong","\varphi"'] 
                &Y(s) \arrow[r,hook]\arrow[d,"\cong", "\varphi_Y"']
                &\PP\times \PP^\vee\arrow[d,"\cong","\psi"']\\
		X_I(s')\arrow[r,"\pi_I"]&Y(s')\arrow[r,hook]&\PP\times \PP^\vee
	\end{tikzcd}
	\]
	commute. In particular, there is a natural map
	\beq\label{eq:isom.ext.XI}
	\Isom(X_I(s),X_I(s')) \;\lra\; \Isom(Y(s),Y(s')).\eeq 
	When $s=s'$, the map \eqref{eq:isom.ext.XI} 
	is a group homomorphism 
	\[
	\Aut X_I(s)\;\lra\;  \Aut Y(s)\subset  \Aut \PP\times \PP^\vee, \]
	which sends $\Aut_I^+X(s)$ into $\Aut^+Y(s)$,
	where the inclusion on the right is \eqref{eq:aut.ext.Y}.
\end{proposition}
\begin{proof}
	Recall that we have canonical isomorphisms \eqref{eq:cohom.XtoXI}. 
	In particular, we have $\coh^0(X_I(s),\cL_1)\cong (\bk^n)^\vee$ and $\coh^0(X_I(s),\cL_n)\cong \bk^n$ canonically, inducing the commutative diagram. The rest of the assertions also follow from the same argument as in the proof of Proposition~\ref{prop:unique.ext.X}.
\end{proof}
We also need a generalization of Proposition~\ref{prop:flagbundles} to partial flag bundles. Let $\cV$ be a vector bundle of rank $r$ over a scheme $Z$. We denote by 
\[\pi_\cV^I:\Fl_I(\cV)\lra Z\] 
the associated partial flag bundle over $Z$ with fibers $\Fl_I$. This is an iterated Grassmann bundle. See \cite[Section~4.2]{AndersonFulton} or \cite[Chapter~14]{fulton-intersection-theory}. Let $\cF_j(\cV)$ and $\cQ_j(\cV)$ denote the universal subbundles and quotient bundles of $(\pi_\cV^I)^*\cV$ of rank $1\leq j\leq r$, respectively. Then, as in the case of the complete flag bundle $\Fl(\cV)$ in \eqref{eq:pushforward.flag2}, there are natural isomorphisms (cf.~\eqref{eq:pushforward.flag2})
\beq\label{eq:pushforward.flag3}\begin{split}
	\pi_{\cV*}^I(\cF_j(\cV)^\vee)&\cong \cV^\vee  \quad \text{when }j\in I\\
	\pi_{\cV*}^I\cQ_j(\cV)&\cong \cV \quad\;\; \text{when }n-j\in I.
\end{split}\eeq
Using these, one can prove the following generalization of Proposition~\ref{prop:flagbundles}.
\begin{proposition}\label{prop:flagbundles.partial}
	Let $Z$ be a connected reduced locally Noetherian scheme. Let $\cV$ and $\cV'$ be vector bundles of rank $r$ over $Z$. Let $I\subset [r-1]$ be a nonempty subset. 
	Then, there exists an isomorphism $\varphi:\Fl_I(\cV)\to \Fl_I(\cV')$ over $Z$ if and only if there exist a line bundle $L$ on $Z$ and an isomorphism 
	\[f:\cV\to \cV'\otimes L\quad\text{or}\quad g:\cV^\vee\to \cV'\otimes L\] that induces $\varphi$, that is, $\varphi=\varphi_f$ or $\varphi=\varphi_g^\vee$, where $\varphi_f$ and $\varphi_g^\vee$ are defined analogously.
	
	When $I$ is not Dynkin-symmetric, only the isomorphism $f$ (not $g$) is allowed.
\end{proposition}
\begin{proof}
	As the `if' direction is vacuous, we prove the `only if' direction.

	Pick any $j\in [r-1]$ with $r-j\in I$, so that $\cQ_{r-j}(\cV')$ is well-defined. 
	We know that any automorphism of $\Fl_r$ pulls back $\cQ_j$ to either $\cF_j^\vee$ or $\cQ_j$. 
	Since $Z$ is connected, this choice is constant. In the first case, we set $\cW:=\cF_j^\vee$ and, in the second case, we set $\cW:=\cQ_j$. We claim that $\varphi^*\cQ_j(\cV')\cong \cW\otimes (\pi_\cV^I)^*L$ 
	for some line bundle $L$ on $Z$. To prove this, 
	observe that the vector bundle $\pi_{\cV*}^I\cHom(\cW,\varphi^*\cQ_j(\cV'))$
	has fibers $\End(F_j^\vee)$ or $\End(Q_j)$ respectively, which are spanned by the identity by Lemma~\ref{lem:simple.tautological} below and Lemma~\ref{lem:projection.formula}.
	By \cite[Theorem~25.1.6]{vakil} (cf.~\cite[Exercise~25.1.K]{vakil}), we have
	\[\pi_{\cV*}^I\cHom(\cW,\varphi^*\cQ_j(\cV'))\cong L\] for some line bundle $L$ on $Z$. By adjunction, there is a canonical homomorphism
	\[(\pi_\cV^I)^*L \lra \cHom(\cW,\varphi^*\cQ_j(\cV'))\]
	which is $\bk\to \End(\cF_j^\vee)$ or $\bk\to \End(\cQ_j)$ sending 1 to the identity. Therefore this forces $\cW\otimes (\pi_\cV^I)^*L\to \varphi^*\cQ_j(\cV')$ to be an isomorphism. This proves the claim. 
	
	The first assertion then follows from \eqref{eq:pushforward.flag3} and the same argument as in the proof of Proposition~\ref{prop:flagbundles}
	
	The last assertion follows immediately from choosing $j\in [r-1]-I$ with $r-j\in I$, whose existence is guaranteed by the assumption that $I$ is not Dynkin-symmetric. In this case, there does not exist $\cF_j(\cV)$, hence it is only possible to have $\varphi^*\cQ_j(\cV)\cong \cQ_j(\cV)\otimes (\pi_\cV^I)^*L$, which induces an isomorphism $f$.
\end{proof}
\begin{lemma}\label{lem:simple.tautological}
	Let $1\leq j<r$. 
	Let $\cF_j$ and $\cQ_j:=\bk^r/\cF_j$ be the universal subbundle and quotient bundle of rank $j$ on the flag variety $\Fl_r$. Then, $\cF_j$ and $\cQ_j$ are simple.
\end{lemma}
\begin{proof}
	Since $\iota^*\cQ_j=\cF_j$, it suffices to show that $\End\cF_j=\bk\cdot \id$. 
	
	We first recall that the universal subbundles fit into short exact sequences
	\[0\lra \cL(e_j)\lra \cF_j^\vee\lra \cF_{j-1}^\vee\lra 0.\]
	Iterating this sequence and 
	using the fact that $\coh^i(\Fl_r,\cL(e_j))=0$ for all $i$ and all $j>1$ by the Borel--Weil--Bott theorem, we obtain canonical identifications
	\[\coh^0(\Fl_r,\cF_j^\vee)=\cdots =\coh^0(\Fl_r,\cF_1^\vee)=\coh^0(\Fl_r,\cL_1)=(\bk^r)^\vee.\]
	
	Next, tensoring the universal exact sequence
	$0\to \cF_j\to \bk^r\to \cQ_{r-j}\to 0$ with $\cF_j^\vee$ and taking global sections yields a left exact sequence
	\[0\lra \End(\cF_j)\lra \End(\bk^r)\lra \Hom(\cF_j,\cQ_{r-j}).\]
	Under this identification, an endomorphism $A\in \End(\bk^r)$ maps to the homomorphism $\cF_j\to \bk^r/\cF_j$ sending a subspace $F_j$ to $AF_j$ modulo $F_j$.
	
	Therefore, $\End(\cF_j)$ consists precisely of those linear maps
	$A\in\End(\bk^r)$ such that $A(F)\subset F$ for every $j$-dimensional subspace
	$F\subset \bk^r$. This condition forces $A$ to be a scalar multiple of the identity. Hence $\End(\cF_j)=\bk\cdot\id$, as claimed.
\end{proof}

Using Lemma~\ref{lem:pullback.partial} and Propositions~\ref{prop:unique.ext.XI} and~\ref{prop:flagbundles.partial}, we prove Theorem~\ref{thm:partial.A}.
\begin{proof}[Proof of Theorem~\ref{thm:partial.A}]
	Let $\Isom^+(X_I(s),X_I(s'))$ denote the subset of isomorphisms preserving $\cL_1$ and $\cL_n$, that is, the preimage of $\Isom^+(Y(s),Y(s'))$ under the map \eqref{eq:isom.ext.XI}.
	It suffices to show that the map 
	\[\Isom^+(X_I(s),X_I(s'))\lra \Isom^+(Y(s),Y(s'))\] 
	is injective and its image is the same as that of \eqref{eq:isom.ext+}. This in particular shows that $\Isom^+(X(s),X(s'))=\Isom^+(X_I(s),X_I(s'))$ as a subset of $\Aut^+\Fl=\Aut^+\Fl_I=G/\mathbb{G}_m$.
	But the injectivity and the assertion on the image now follow from the same argument as in the proofs of Theorems~\ref{t.isoA}\eqref{t.isoAX} and~\ref{t.autsA}\eqref{t.autsAX}.
\end{proof}

We end this section with the proof of Lemma~\ref{l.HandK}.

\begin{proof}[Proof of Lemma~\ref{l.HandK}] 
It is clear that both $H(s)$ and $K(s)$ contain the centralizer $T(s)$ of~$s$. We prove that $H(s)$ and $K(s)$ are contained in the normalizer $N_G(T(s))$ of~$T(s)$. Since $H(s)\subset K(s)$, it suffices to prove that $K(s)\subset N_G(T(s))$. 

Let $g\in K(s)$. By definition, there exists $\begin{pmatrix}
	a&b\\c&d
\end{pmatrix}\in \GL_2$ such that 
\beq \label{eq:cond3}g^{-1}sg=(as+b)(cs+d)^{-1}.\eeq 
Here the right-hand side is well-defined, as we identify $\fg\cong \End(\bk^n)$ with the space of $n\times n$ matrices.

Since $as+b$ and $cs+d$ lie in the Lie algebra of the maximal torus $T(s)$, it follows from \eqref{eq:cond3} that $g^{-1}sg$ also lies in the Lie algebra of $T(s)$.  
On the other hand, $g^{-1}sg$ has the same eigenvalues as $s$. 
Therefore, there exists $\sigma \in N_G(T(s))$ such that $g^{-1}sg=\sigma^{-1}s\sigma$. This implies that $g\sigma^{-1}\in T(s)$, and hence $g\in N_G(T(s))$.

Finally, since the involution $\iota$ is chosen so as to preserve $s$, 
we have $(\iota g)^{-1} s(\iota g)=(as+b)(cs+d)^{-1}$ by \eqref{eq:cond3}. Thus $\iota g\in K(s)$, and consequently $\iota (K(s))=K(s)$. 
The same argument with $c=0$ shows that $\iota(H(s))=H(s)$. 
\end{proof}

\bigskip

\section{Relation to $M_{0,n}$}\label{s:M0n}

In this section, we interpret our main results in type~$A$,
namely Theorems~\ref{t.isoA} and~\ref{t.autsA},
in terms of group actions
(Corollaries~\ref{cor:interpret.AX1},~\ref{cor:interpret.AX},~\ref{cor:interpret.AY1} and~\ref{cor:interpret.AY})
and pointed rational curves
(Corollaries~\ref{cor:main.M0n.X} and~\ref{cor:main.M0n.Y}).

We further show that Tymoczko's dot action of $S_n$ on the cohomology
$\coh^*(X(s))$ does not lift to an action on the variety $X(s)$
(Corollary~\ref{no-dot.body}).

\smallskip
Throughout this section, we continue to assume that $n\geq 4$.
Moreover, using the identification
$\Aff \cong \mathbb{G}_m \ltimes \mathbb{G}_a \cong \bk^\times \ltimes \bk$,
we will write elements of $\Aff$ as $(a,b)$.

\subsection{Main theorems of $X(s)$ in terms of group actions}\label{ss:group.actions}
We first interpret the main results for $X(s)$ in terms of group actions.
Recall that
\[
\Aut G=(G/\mathbb{G}_m)\rtimes \mu_2,
\]
where $G/\mathbb{G}_m$ is the group of inner automorphisms of $G$
and $\mu_2$ is generated by an involution $\iota$ arising from the Dynkin diagram.
Let
\[
\cG:=\Aut G\times \Aff = \big((G/\mathbb{G}_m)\rtimes \mu_2\big)\times \Aff.
\]
This group acts naturally on $\fg$:
$\Aut G$ acts via the differential of automorphisms (for example, $g\cdot s=gsg^{-1}$ and $\iota\cdot s=-s^t$ for $g\in G$ and $s\in \fg$), while $\Aff$ acts by
\[
(a,b)\cdot s = as+b \qquad \text{for } (a,b)\in \Aff \text{ and } s\in \fg.
\]
This action preserves $\grs$ and $\fg-\langle\id\rangle$.

\begin{corollary}\label{cor:interpret.AX1}
	Let $s,s'\in\grs$. Then the following hold.
	\begin{enumerate}
		\item $X(s)\cong X(s')$ if and only if $s$ and $s'$ lie in the same $\cG$-orbit.
		\item $\Aut X(s)$ is isomorphic to the stabilizer of $s$ in $\cG$.
\end{enumerate}
\end{corollary}
\begin{proof}
(1) By Theorem~\ref{t.isoA}\eqref{t.isoAX}, $X(s)\cong X(s')$ if and only if
there exist $g\in G$ and $(a,b)\in\Aff$ such that $gs'g^{-1}=as+b$. This is precisely the condition that $s$ and $s'$ lie in the same $\cG$-orbit in $\grs$.

\smallskip

(2) The statement follows directly from Theorem~\ref{t.autsA}\eqref{t.autsAX}.
\end{proof}

\medskip

By Theorem~\ref{thm:XtoY}, every isomorphism $X(s)\to X(s')$ extends uniquely
to an automorphism of $\Fl_n$.
Thus, the isomorphism classes and automorphism groups of $X(s)$ can equivalently
be read off from the natural $\Aut \Fl_n$-action on
$\PP\coh^0(\Fl_n,\cL_1\otimes\cL_n)$, the space of Hessenberg divisors.
Let $[s]$ denote the class of the divisor $X(s)$.

\begin{corollary}\label{cor:interpret.AX}
	Let $s,s'\in\grs$. Then the following hold.
	\begin{enumerate}
		\item $X(s)\cong X(s')$ if and only if $[s]$ and $[s']$ lie in the same $\Aut \Fl_n$-orbit.
		\item $\Aut X(s)$ is isomorphic to the stabilizer of $[s]$ in $\Aut \Fl_n$.
	\end{enumerate}
\end{corollary}

Corollaries~\ref{cor:interpret.AX1} and~\ref{cor:interpret.AX} are indeed equivalent,
since both describe the same orbits and stabilizers, expressed via two different but
canonically equivalent group actions.

Note that the $\Aff$-action on $\fg - \langle \id \rangle$ is free, with quotient
\[
\PP(\pgl_n) = (\fg - \langle \id \rangle)/\Aff,
\]
where  $\pgl_n$ denotes the Lie algebra of $\PGL_n = G/\mathbb{G}_m$.
Consequently, the $\cG$-action on $\fg - \langle \id \rangle$ descends to an
$\Aut G$-action on $\PP(\pgl_n)$.
This action can be identified with the natural $\Aut \Fl_n$-action on the complete
linear system $\PP \coh^0(\Fl_n,\cL_1\otimes \cL_n)$.
Indeed, the natural homomorphism $\Aut G\to \Aut \Fl_n$ is an isomorphism, and the
canonical identification
\[
\coh^0(\Fl_n,\cL_1\otimes \cL_n)
\;\cong\;
\End(\bk^n)/\langle \id \rangle
\;\cong\;
\pgl_n,
\]
given by Lemma~\ref{lem:projection.formula} and~\eqref{eq:pgln}, is equivariant with
respect to these actions.

Let
\[
\PP \coh^0(\Fl_n,\cL_1\otimes \cL_n)^{\rs}
\;\subset\;
\PP \coh^0(\Fl_n,\cL_1\otimes \cL_n)
\]
denote the open subset corresponding to $\grs/\Aff \subset \PP(\pgl_n)$.
Under the above identifications, the $\cG$-action on $\grs$ and the $\Aut \Fl_n$-action
on $\PP \coh^0(\Fl_n,\cL_1\otimes \cL_n)^{\rs}$ have canonically identified orbits and
stabilizers.

Equivalently, the corresponding quotient stacks are canonically isomorphic:
\beq\label{eq:quot.stack.X}[\grs/\cG]\;\cong\;[\PP \coh^0(\Fl_n,\cL_1\otimes \cL_n)^{\rs}/\Aut \Fl_n].\eeq

\medskip

\subsection{Main theorems of $Y(s)$ in terms of group actions}\label{ss:group.actions.Y}
We now give a parallel interpretation for $Y(s)$.
In this case, the role of $\cG$ is played by the automorphism group of $\PP\times\PP^\vee$,
together with the natural $\PGL_2$-action on pairs of $(1,1)$-divisors.

Let 
\[\tcG:= \Aut \PP\times\PP^\vee\times \PGL_2=\left((G/\mathbb{G}_m)^2\rtimes \mu_2\right) \times \PGL_2,\]
where $\Aut \PP\times \Aut\PP^\vee=(G/\mathbb{G}_m)^2$ and $\mu_2$ acts by the involution $(g,h)\mapsto (\iota(h),\iota(g))$, corresponding to the Dynkin involution exchanging $\PP$ and $\PP^\vee$.

Recall that there are natural isomorphisms
\beq\label{eq:fg}\coh^0(\PP\times\PP^\vee,\cO(1,1))\cong \End(\bk^n)\cong \fg.\eeq

Under this identification, $\PP(\fg\oplus \fg)$ parametrizes pairs of $(1,1)$-divisors in $\PP\times\PP^\vee$.
The group $\tcG$ acts on this space as follows.
The factor $(G/\mathbb{G}_m)^2$ acts by 
\[(g_1,g_2)\cdot [s:u]=[g_2sg_1^{-1}:g_2ug_1^{-1}]\]
for $g_1,g_2\in G$ and $s,u\in\fg$, 
while $\mu_2$ acts by \[[s:u]\mapsto [s^t:u^t].\]
Finally, $\PGL_2$ acts by linear recombination of two components:
\[
\begin{pmatrix}a&b\\ c&d\end{pmatrix}\cdot[s:u]
=[as+bu:cs+du] \qquad \text{for }\begin{pmatrix}a&b\\ c&d\end{pmatrix}\in \PGL_2 .
\]
 
We denote by $\PP(\fg\oplus \fg)^{\rs}\subset \PP(\fg\oplus \fg)$ the open subset 
consisting of points lying in the $\tcG$-orbits of elements of the form $[s:\id]$ with $s\in\grs$.

\begin{corollary}\label{cor:interpret.AY1}
	Let $s,s'\in \grs$. Then, the following hold.
	\begin{enumerate}
		\item $Y(s)\cong Y(s')$ if and only if $[s:\id]$ and $[s':\id]$ lie in the same $\tcG$-orbit.
		
		\item $\Aut Y(s)$ is isomorphic to the stabilizer of $[s:\id]$ in $\tcG$.
	\end{enumerate}
\end{corollary}
\begin{proof}
(1) By Theorem~\ref{t.isoA}\eqref{t.isoAY}, we have $Y(s)\cong Y(s')$ if and only if
there exist $g_1,g_2\in G$ and
$A=\begin{pmatrix} a & b \\ c & d \end{pmatrix}\in \PGL_2$
such that
\[g_2s'g_1^{-1}=(as+b)(cs+d)^{-1}.\]
Equivalently, this can be written as
$[g_2s'g_1^{-1} : \id] = A\cdot [s:\id]$,
which is precisely the condition that $[s:\id]$ and $[s':\id]$ lie in the same
$\tcG$-orbit in $\PP(\fg\oplus\fg)^{\rs}$.

\smallskip

(2) This follows analogously from Theorem~\ref{t.autsA}\eqref{t.autsAY}.
\end{proof}

Let $\Gr_2(-)$ denote the Grassmannian of two-dimensional linear subspaces.
Consider the natural $\Aut\PP\times\PP^\vee$-action on
\[
\Gr_2\bigl(\coh^0(\PP\times \PP^\vee,\cO(1,1))\bigr)
\;\cong\;
\Gr_2(\fg),
\]
which parametrizes pencils of $(1,1)$-divisors, or equivalently,
complete intersections of two such divisors in $\PP\times\PP^\vee$. Here the isomorphism is given by \eqref{eq:fg}.

Since every isomorphism $Y(s)\to Y(s')$ is induced by an automorphism of
$\PP\times \PP^\vee$ (Lemma~\ref{lem:unique.ext.Y}),
the isomorphism classes and automorphism groups of $Y(s)$ can also be read off
directly from this action,
without passing through the $\tcG$-action.

\begin{corollary}\label{cor:interpret.AY}
	Let $s,s'\in \grs$. Then, the following hold.
	\begin{enumerate}
		\item $Y(s)\cong Y(s')$ if and only if $\langle s,\id\rangle$ and $\langle s',\id\rangle$ lie in the same
		$\Aut\PP\times\PP^\vee$-orbit.
		
		\item
		$\Aut Y(s)$ is isomorphic to the stabilizer of $\langle s,\id\rangle$
		in $\Aut\PP\times\PP^\vee$.
	\end{enumerate}
    Here, as before, \(\langle x_1,\dots, x_r \rangle \) denotes the \(\bk\)-span of \(x_1,  \dots , x_r\) in \(\fg\).
\end{corollary}
Corollaries~\ref{cor:interpret.AY1} and~\ref{cor:interpret.AY} are equivalent,
as the $\tcG$-action on pairs $[s:u]$ in $\PP(\fg\oplus\fg)^{\rs}$ and the $\Aut\PP\times\PP^\vee$-action
on the associated pencils $\langle s,u\rangle$ have canonically identified
orbits and stabilizers.

\smallskip

We first identify the set of such associated pencils.
For any $[s:u]\in \PP(\fg\oplus \fg)^{\rs}$, the elements  $s$ and $u$ are linearly independent in $\fg$. 
Hence there is a natural morphism
\[\PP(\fg\oplus \fg)^{\rs}\lra \Gr_2(\fg), \quad [s:u]\mapsto \langle s,u\rangle.\]
We denote the image of this morphism by 
\[
\Gr_2(\coh^0(\PP\times \PP^\vee,\cO(1,1)))^{\rs} 
\]
as a subvariety of $\Gr_2(\coh^0(\PP\times \PP^\vee,\cO(1,1)))\cong \Gr_2(\fg)$.

This image is nothing but the quotient by the free action of $\PGL_2$ on
$\PP(\fg\oplus \fg)^{\rs}$, which acts by changing the ordered basis of the
two-dimensional subspace $\langle s,u\rangle$.
Consequently, the $\tcG$-action on $\PP(\fg\oplus \fg)^{\rs}$ descends to an
$\Aut\PP\times\PP^\vee$-action on the quotient
$\Gr_2(\coh^0(\PP\times \PP^\vee,\cO(1,1)))^{\rs}$.
Under the identification~\eqref{eq:fg}, this induced action coincides with the
natural action obtained by transforming $(1,1)$-divisors.

In short, the corresponding quotient stacks are canonically isomorphic:
\beq\label{eq:quot.stack.Y}[\PP(\fg\oplus\fg)^{\rs}/\tcG]\;\cong\;[\Gr_2(\coh^0(\PP\times\PP^\vee,\cO(1,1)))^{\rs}/\Aut \PP\times\PP^\vee].\eeq

\medskip

It is worth noting that there exist natural morphisms relating \eqref{eq:quot.stack.X} and \eqref{eq:quot.stack.Y}:
\beq\label{eq:inclusion.pgln}\grs \lra \PP(\fg\oplus\fg)^{\rs} \quad \text{and}\quad \PP\coh^0(\Fl_n,\cL_1\otimes\cL_n)^{\rs}\lra \Gr_2(\coh^0(\PP\times\PP^\vee,\cO(1,1)))^{\rs}\eeq
which send $s$ to $[s:\id]$, and the class of $X(s)$ to that of $Y(s)$, respectively. These maps are equivariant under the actions of $\cG$, $\tcG$, $\Aut \Fl_n$, and $\Aut \PP\times\PP^\vee$, compatibly with the natural group homomorphisms $\cG\to \tcG$ and $\Aut \Fl_n\to \Aut \PP\times\PP^\vee$.

\medskip

\subsection{Main theorems in terms of pointed rational curves (Theorem~\ref{t.HKStab})}\label{ss:ModCurve.main}
We provide an interpretation of our main theorems in terms of pointed rational curves. 
The moduli space $M_{0,n}$ of $n$-pointed smooth rational curves is the quotient
\beq \label{eq:M0n.def}
M_{0,n}:=\left((\PP^1)^n-\Delta\right) /\Aut \PP^1
\eeq
by the diagonal action of $\Aut \PP^1=\PGL_2$, 
where $\Delta$ is the set of $n$ points $(p_1,\dots,p_n)$ with $p_i=p_j$ for some $i\neq j$.
Note that the action is free on $(\PP^1)^n-\Delta$, since any projective linear transform of $\PP^1$ fixing three distinct points must be trivial.

By fixing the last point at the infinity, we get an isomorphism
\beq \label{eq:M0n+1.Aff}M_{0,n+1}\cong \left( \mathbb{A}^n-\Delta\right)/\Aff.
\eeq
Here, by abuse of notation, $\Delta$ is the subset where at least two coordinates coincide. 
Then, the natural action of the symmetric group $S_n$ on $M_{0,n+1}$ permuting the first $n$ points coincides with the action on $\left( \mathbb{A}^n-\Delta\right)/\Aff$ 
permuting the coordinates. 

Combining this description with the interpretation of
Subsection~\ref{ss:group.actions}, we obtain the following identification.
\begin{proposition}\label{prop:isom.quotients}
	$\grs/\cG\cong M_{0,n+1}/S_n$.
\end{proposition}
\begin{proof}
	This is obtained from the isomorphism
\beq \label{eq:isom.quotients2}\grs/(G/\mathbb{G}_m)\cong \mathfrak{t}^\rs/(N_G(T)/\mathbb{G}_m)=\mathfrak{t}^\rs/W \cong (\mathbb{A}^n-\Delta)/S_n\eeq
by further taking quotients of the both sides by the induced $\mu_2\times\Aff$-actions.
Here, $T\subset G$ is a maximal torus with Lie algebra $\mathfrak t$, 
$N_G(T)\subset G$ is the normalizer subgroup of $T$, $W\cong S_n$ is the Weyl group, and $\mathfrak{t}^{\rs}:=\grs \cap \mathfrak{t}$. The first isomorphism in \eqref{eq:isom.quotients2} is given by diagonalization and the last isomorphism is given by an isomorphism $\mathfrak{t}\cong \mathbb{A}^n$ choosing coordinates. 
Moreover, we may choose any $\mu_2$ that preserves $\mathfrak{t}^{\rs}$ among its conjugates. Then the induced $\mu_2$-actions on the quotients are trivial.
\end{proof}

The quotient $\grs/\cG$ classifies the isomorphism classes of $X(s)$ for $s\in \grs$ by Corollary~\ref{cor:interpret.AX1}(1). 
Hence, Proposition~\ref{prop:isom.quotients} 
reads as the bijection
\[
\{\text{isomorphism classes of }X(s):s\in \grs\}\;\lra\; M_{0,n+1}/S_n\]
	which 
	associates to the isomorphism class of $X(s)$ the set of eigenvalues of $s$ modulo the $\Aff$-action, 
	that is, 
	$[(\lambda_1,\dots,\lambda_n,\infty)]\in M_{0,n+1}/S_n$ for the eigenvalues $\lambda_i$ of $s$. 
	
	\smallskip
	
	Further, $\bar{H}(s)$ 
	is the stabilizer of $(\lambda_1,\dots,\lambda_n,\infty)\in M_{0,n+1}$  under the $S_n$-action.	

\begin{corollary}[Theorem~\ref{t.HKStab}(1)]\label{cor:main.M0n.X}
	Let $n\geq 4$, and let $s\in \grs$ with eigenvalues $\lambda_1,\dots,\lambda_n$.
	Then $\bar{H}(s)$ is conjugate in $S_n$ to the stabilizer of $(\lambda_1,\dots,\lambda_n,\infty)\in M_{0,n+1}$ under the $S_n$-action permuting the first $n$ marked points.
\end{corollary}

\begin{proof}
	By Lemma~\ref{l.HandK} and Theorem~\ref{t.autsA}\eqref{t.autsAX}, $\bar{H}(s)$ is conjugate to the subgroup of $S_n$ consisting of $\sigma\in S_n$ such that 
	$(\lambda_{\sigma^{-1}(1)},\dots, \lambda_{\sigma^{-1}(n)})=(\lambda_1,\dots, \lambda_n)$ modulo affine transformations.
	This is precisely the stabilizer of the point represented by $(\lambda_1,\dots, \lambda_n)$ under the $S_n$-action on $(\mathbb{A}^n-\Delta)/\Aff\cong M_{0,n+1}$.
\end{proof}

We next explain how the results for $Y(s)$ fit into the same framework.
For $Y(s)$, we have an analogous result.
\begin{proposition}
	$\PP(\fg\oplus\fg)^{\rs}/\tcG\;\cong \; M_{0,n}/S_n$.
\end{proposition}
\begin{proof}
	As in the case of $X(s)$, we restrict to a Cartan subalgebra and then pass to the quotients by the Weyl group $W\cong S_n$.
	Let $\PP(\ft\oplus\ft)^{\rs}:=\PP(\fg\oplus \fg)^{\rs}\cap \PP(\ft\oplus \ft)$.
Then we obtain a chain of natural isomorphisms
\[\PP(\fg\oplus \fg)^{\rs}/(G/\mathbb{G}_m)^2\cong \PP(\ft\oplus\ft)^{\rs}/(N_G(T)/\mathbb{G}_m)=\PP(\ft\oplus\ft)^{\rs}/W\cong ((\PP^1)^n-\Delta)/S_n \]
analogous to \eqref{eq:isom.quotients2},
where $N_G(T)$ is the normalizer in the diagonal $G\subset G\times G$. The desired isomorphism then follows by passing to the quotients by the induced
$\mu_2\times \PGL_2$-actions on both sides.
\end{proof}

Consequently, for $Y(s)$, we obtain an analogous bijection
\[
\{\text{isomorphism classes of }Y(s):s\in \grs\}\;\lra\; M_{0,n}/S_n.\] 
Under this correspondence, the isomorphism class of $Y(s)$ is sent to the
configuration of eigenvalues of $s$ modulo the action of $\Aut \PP^1=\PGL_2$, that is, to the point $[(\lambda_1,\dots,\lambda_n)]\in M_{0,n}/S_n$,
where $\lambda_1,\dots,\lambda_n$ are the eigenvalues of $s$.
Moreover, $\bar{K}(s)$ is identified with the stabilizer in $S_n$ of the
corresponding point of $M_{0,n}$.

\begin{corollary}[Theorem~\ref{t.HKStab}(2)]\label{cor:main.M0n.Y}
	Let $n\geq 4$, and let $s\in \grs$ with eigenvalues $\lambda_1,\dots,\lambda_n$.
	Then $\bar{K}(s)$ is conjugate in $S_n$ to the stabilizer of $(\lambda_1,\dots,\lambda_n)\in M_{0,n}$ under the $S_n$-action permuting the $n$ marked points.
\end{corollary}
\begin{proof}
	The proof is the same as that of Corollary~\ref{cor:main.M0n.X}. By Lemma~\ref{l.HandK} and Theorem~\ref{t.autsA}\eqref{t.autsAY}, $\bar{K}(s)$ is conjugate to the subgroup of $S_n$ consisting of $\sigma \in S_n$ such that $(\lambda_{\sigma^{-1}(1)},\dots, \lambda_{\sigma^{-1}(n)})=(\lambda_1,\dots, \lambda_n)$ 
	modulo linear fractional transformations. This is precisely the stabilizer of the point $(\lambda_1,\dots, \lambda_n)\in M_{0,n}$ under the $S_n$-action.
\end{proof}

\begin{remark}
    By the corollaries in this section, it is natural to expect that
\begin{itemize}
	\item the moduli stack of $X(s)$ is isomorphic to the quotient stack \eqref{eq:quot.stack.X}
		and it is a $((T/\mathbb{G}_m)\rtimes \mu_2)$-gerbe over $[M_{0,n+1}/S_n]$ for a maximal torus $T\subset G$;
	\item the moduli stack of $Y(s)$ is isomorphic to the quotient stack \eqref{eq:quot.stack.Y}
		and it is a $((T/\mathbb{G}_m)\rtimes \mu_2)$-gerbe over $[M_{0,n}/S_n]$. 
\end{itemize}
These will be established Theorems~\ref{thm:moduli.stack} and~\ref{thm:moduli.stack.Y} respectively, in Section~\ref{s.moduli}.
\end{remark}

\medskip

\subsection{Nonexistence of a geometric lift of Tymoczko's dot action}\label{ss:Tymoczko}
An important consequence of Corollary~\ref{cor:main.M0n.X} is that $\bar{H}(s)$ is cyclic.

\begin{corollary}\label{cor:classification.H(s)}
	Let $n\geq 4$ and $s\in \grs$. Then $\bar{H}(s)$ is a cyclic subgroup of $S_n$. Moreover, $\bar{K}(s)$ is a subgroup of $S_n$ isomorphic to one of the following:
	\begin{enumerate}
		\item the cyclic group $C_k$ of order $k$;
		\item the dihedral group $D_k$ of order $2k$;
		\item one of the groups $A_4$, $S_4$, or $A_5$.
	\end{enumerate}
	Here $k$ is a positive integer, and $A_k$ denotes the alternating subgroup of $S_k$.
\end{corollary}
\begin{proof}
	Any subgroup of $S_n$ fixing a point in $M_{0,n+1}$ (resp. $M_{0,n}$) is a finite subgroup of $\Aff$ (resp. $\PGL_2$). Indeed, if $\sigma \in S_n$ stabilizes a point in $M_{0,n+1}$ (resp. $M_{0,n}$) represented by $n$ distinct points $\lambda_1,\dots,\lambda_n$ in $\mathbb{A}^1$ (resp. $\PP^1$), then there exists a unique element $\phi$ in $\Aff$ (resp. $\PGL_2$) such that $\lambda_{\sigma^{-1}(i)}=\phi(\lambda_i)$ for $1\leq i\leq n$. This induces an injective homomorphism from the stabilizer group to $\Aff$ (resp. $\PGL_2$).
	The assertion on $\bar{K}(s)$ then follows from the classical fact that any finite subgroup of $\PGL_2$ is among (1)--(3) (\cite[\S1.1]{dolgachev}). 
	
	Moreover, any finite subgroup of $\Aff$ embeds into $\mathbb{G}_m$ by composing 
	with the projection $\Aff\to \mathbb{G}_m$, which sends $(a,b)$ to $a$. Indeed, the kernel of this projection is the additive group $\mathbb{G}_a\cong \bk$, which has no elements of finite order. This shows that $\bar{H}(s)$ is a finite subgroup of $\mathbb{G}_m\cong \bk^\times$, implying that it is a cyclic group.
\end{proof}

\begin{remark}
When $\bar H(s)$ or $\bar K(s)$ is cyclic of order $k>1$, one can say more about
the possible values of $k$, following \cite[Lemma~2.4]{moon-swinarski}.

If $\bar H(s)\cong C_k$ with $k>1$, then $\bar H(s)$ is generated by a permutation
of the form $\sigma_1\cdots\sigma_\ell$, where the $\sigma_i\in S_n$ are mutually
disjoint cycles, each of length $k$, and satisfy $n-k\ell\leq 1$.
In particular, $k$ divides $n$ or $n-1$.

If $\bar K(s)\cong C_k$ with $k>1$, an analogous statement holds, except that
the inequality becomes $n-k\ell\leq 2$.
In particular, $k$ divides $n$, $n-1$, or $n-2$.

Moreover, any group appearing in (1)--(3) of Corollary~\ref{cor:classification.H(s)} occurs as $\bar K(s)$ for some $n$ and some $s\in\grs$ \cite[Theorem~71]{WuXu}.
\end{remark}

From this, one can confirm
that Tymoczko's dot action (\cite{tymoczko}; cf.~\cite[Section~8]{AHTMMS20}) on the cohomology $\coh^*(X(s))$ of $X(s)$
does not lift to an action on $X(s)$ itself. 

\begin{corollary}[Corollary~\ref{no-dot}] 
\label{no-dot.body}
	Let $n\geq 4$ and $s\in \grs$.
	Any $S_n$-representation on $\coh^*(X(s))$ that is induced by an $S_n$-action on $X(s)$
	is a direct sum of copies of the trivial representation and the sign representation.
	
	In particular, there is no $S_n$-action on $X(s)$ that induces Tymoczko's dot action. 
\end{corollary}
\begin{proof}
	Let $\eta: S_n \to \Aut X(s)$ be a group homomorphism, inducing an $S_n$-action on the cohomology $\coh^*(X(s))$.
	Since $T(s)=Z_G(s)$ acts trivially on $\coh^*(X(s))$, by the $T(s)$-invariant cellular decomposition provided by the theorem of Bia{\l}ynicki-Birula, the induced $S_n$-action on cohomology factors through the finite quotient
	\[\Aut X(s)/(T(s)/\mathbb{G}_m)=\bar{H}(s)\times \mu_2.\]
	By Corollary~\ref{cor:classification.H(s)}, we have $\bar{H}(s)\cong C_k$.
	Let $H_\eta$ denote the image of the composition
	\[S_n \xrightarrow{\;\eta\;} \Aut X(s) \longrightarrow C_k \times \mu_2.\]
	Since $H_\eta$ is a quotient of $S_n$, it is either trivial or isomorphic to $S_n/A_n\cong C_2$.
	For $n\geq 5$, this follows from the fact that $A_n$ is the unique nontrivial proper normal subgroup of $S_n$.
	When $n=4$, although $S_3$ may also occur as a quotient of $S_4$, it cannot embed into $C_k \times \mu_2$.

	This shows that any $S_n$-representation on $\coh^*(X(s))$ arising from such an action $\eta$ decomposes as a direct sum of the trivial and sign representations.

	On the other hand, the $S_n$-representation on $\coh^*(X(s))$ given by Tymoczko's dot action contains the $(n-1)$-dimensional standard representation of $S_n$, which is irreducible and is neither trivial nor sign \cite[Corollary~4.6]{ahm} (cf.~\cite[Example~5.12]{kiem-lee}).
	Therefore, the dot action cannot arise from any action of $S_n$ on the variety $X(s)$.
\end{proof}

In contrast, restricting Tymoczko's dot action along $C_n\subset S_n$ yields a $C_n$-action on $\coh^*(X(s))$ which does lift to an action on $X(s)$ for certain choices of $s\in \grs$, as illustrated in the following example.

\begin{example}[A $C_n$-action on $X(s)$] \label{e.Zn action}
	Let $\xi$ be a primitive $n$-th root of unity. Let $s=\mathrm{diag}(\xi,\xi^2,\dots,\xi^i,\dots,1)$ be the $n\times n$ diagonal matrix with (ordered) diagonal entries $\xi, \xi^2,\dots,\xi^i,\dots,\xi^{n}=1$. 
	Then, $\bar{H}(s)\cong C_n$ generated by $\sigma=(1 2 \dots n)\in S_n$.  Indeed,
	the permutation matrix 
	$g=g_\sigma\in G$ associated to $\sigma$ 
	satisfies 
	\[gsg^{-1}=\mathrm{diag}(1,\xi,\dots,\xi^{n-1})=\xi^{-1}s.\] 
	In particular, $g\in H(s)$ and $\sigma \in \bar{H}(s)$. 
	Conversely, suppose that $\tau\in \bar{H}(s)$. Then, there exists $(a,b)\in \Aff$ such that $\tau s \tau^{-1}=as+b$, that is, 
	\[\xi^{\tau^{-1}(i)}=a\xi^{i}+b \quad \text{ for }~1\leq i\leq n.\]
	By summing (resp. multiplying) these up for $1\leq i\leq n$, we conclude $b=0$ (resp. $a^n=1$). In particular, $a=\xi^j$ for some $j$, and $\tau\in \langle \sigma\rangle=C_n$.

	Consider the action of $C_n$ on $X(s)$ sending $F_\bullet$ to $gF_{\bullet}$, which shifts the coordinates of $\bk^n$. One can easily check that this action does lift the restriction (to $C_n$) of Tymoczko's dot action on $\coh^*(X(s))$. In other words, the induced $C_n$-action on $\coh^*(X(s))$ coincides with the restriction of Tymoczko's dot action of $S_n$ to $C_n$.
\end{example}
For the definition of Tymoczko's dot action for generalized Hessenberg varieties and its relation to the dot action for Hessenberg varieties, see \cite[\S2]{kiem-lee}.

 \begin{example} [A $D_n$-action on $Y(s)$] \label{e.Dn action}
 Let $s$ and $g_{\sigma}$ be as in Example~\ref{e.Zn action}. Then $g_{\sigma}$ is contained in $K(s)$ and acts on $Y(s)$ by mapping $(F_1, F_{n-1})$ to $(g_{\sigma}F_1, g_{\sigma}F_{n-1})$. Let $\tau$ be the permutation   mapping $i$ to $n-i$ for $1 \leq i \leq n-1$ and fixing $n$, and $g_{\tau}  $ be the element  in $G$ associated to $\tau$.   Then  
 \[  g_{\tau} s g_{\tau}^{-1} = s^{-1}.
 \] 
Thus $g_{\tau}$ is contained in $K(s)$ and acts on $Y(s)$ by sending $(F_1, F_{n-1})$ to $(g_{\tau}sF_1, g_{\tau}F_{n-1})$. 
 The subgroup generated by  $g_{\sigma} $ and $g_{\tau}$ defines an action of the dihedral group $D_n$ on $Y(s)$.     In fact, $\bar{K}(s)$ is $D_n$ (\cite[Theorem 22]{WuXu}). 
\end{example}

The following holds in general. 

\begin{lemma}\label{lem:restriction.Tymoczko}

Let $s\in \grs$. 
The action of $H(s)\subset \Aut X(s)$ on $X(s)$ induces a natural $\bar H(s)$-action on the $T(s)$-equivariant cohomology $\coh^*_{T(s)}(X(s))$. 
This $\bar H(s)$-action descends to an action on $\coh^*(X(s))$, which is precisely Tymoczko's dot action on $\coh^*(X(s))$ restricted to the subgroup $\bar H(s)\subset S_n$.

An analogous statement holds for $Y(s)$, with $\bar H(s)$ replaced by $\bar K(s)$.
\end{lemma}

\begin{proof}
	Since $T(s)$ is conjugate to $T$, the maximal torus in $B$, we replace $s$ by a suitable conjugate and assume $T=T(s)$.
	Write $H:=\bar H(s)$ for simplicity.
	
	Let $ET\to BT$ be the universal principal $T$-bundle over the classifying  space $BT$.
	The $T\rtimes H$-action on $X(s)$ induces an $H$-action on $ET\times^TX(s)$, and hence an $H$-action on the equivariant cohomology $\coh^*_T(X(s))= \coh^*(ET\times^TX(s))$. Since the inclusion $ET\times^TX(s)^T\hookrightarrow ET\times^TX(s)$ is $H$-equivariant,
	the restriction map 
	\[\coh^*_T(X(s)) ~ \hookrightarrow ~\coh^*_T(X(s)^T) \cong \coh^*(BT)\otimes \coh^*(X(s)^T)\]
	is $H$-equivariant. 
	On the other hand, the $H$-actions on 
	\[\coh^*(BT)\cong \mathbb{Q}[t_1,\dots,t_n] ~ \text{ and } ~  \coh^*(X(s)^T)\cong \bigoplus_{\sigma \in S_n}\mathbb{Q}\]
	coincide with the restrictions of the canonical $S_n$-actions used in the definition of Tymoczko's dot action on $\coh^*_T(X(s))$. It follows that the $H$-action on $\coh^*_T(X(s))$ coincides with the restriction of Tymoczko's dot action. Consequently, the same holds for $\coh^*(X(s))\cong \coh^*_T(X(s))/(t_1,\dots,t_n)\coh^*_T(X(s))$.
	
	The same argument applies to $Y(s)$ and $\bar{K}(s)$, in which case $Y(s)^T$ is isomorphic to $S_n/(S_1\times S_{n-2}\times S_1)$. 		
\end{proof}

\begin{remark}
	The assertion on $X(s)$ in Lemma~\ref{lem:restriction.Tymoczko} remains valid when $X(s)$ is replaced by any regular semisimple generalized Hessenberg variety of arbitrary codimension. 
	It also holds in other types, provided that one defines $H(s)$ as the set of $g\in G$ such that $\Ad(g^{-1})s\in \langle s\rangle$ in the Lie algebra of $G^{\ad}$, and defines $\bar H(s)$ as the image of $H(s)$ in the Weyl group $W$ under the projection $N_G(T(s))\to W$.
 \end{remark}

\begin{remark}
    It is proved in \cite{Aut-GKM} that for a connected regular semisimple Hessenberg variety $X\subsetneq \Fl_n$ of type $A$, the reductive part of $\Aut^\circ X$ is $T(s)/\mathbb{G}_m$, and $\pi_0(\Aut X)=\Aut X/\Aut^\circ X$ is finite, by studying the GKM graph of $X$. 
\end{remark}

\bigskip

\section{Hessenberg varieties in arbitrary types}\label{s:general}
In this  section  and the next,  let  $G$ be a simple algebraic  group over $\bk$. 
Take a Borel subgroup $B$ of
$G$  and the maximal torus $T$   in it.  
We write $\mathfrak{g}$, $\mathfrak{b}$ and $\mathfrak{t}$ for the Lie algebras of $G$, $B$ and $T$ respectively.   Let $\Phi$  (resp. $\Phi^+$, $\Phi^-$)   be the    set of (resp. positive, negative) roots of $G$. Let $\Delta=\{\alpha_1, \dots, \alpha_r\} \subset \Phi$ denote the set of simple roots. 
We assume that $\mathfrak b$ is generated by $\mathfrak t$ and the root spaces $\mathfrak g_{\alpha}$ with $\alpha \in \Phi^-$.

A subspace $H \subset \mathfrak g$ is said to be a \emph{Hessenberg space} if $H$ is $\mathrm{Ad}(B)$-invariant and contains $\mathfrak b$. Then it corresponds a \emph{Hessenberg subset}, a subset $M \subset \Phi^+$ satisfying that for any $\beta \in M$ and $\alpha \in -\Delta$ with $\beta+\alpha \in \Phi^+$, we have $\beta+\alpha \in M$. The correspondence satisfies $H=\mathfrak b \oplus \bigoplus_{\alpha \in M}\mathfrak g_{\alpha}$.

Write $G\times^B H$ for the quotient of $G \times H$ by the diagonal $B$-action,
and let $\mathcal{B}$ denote the conjugacy class of $B$. 
We then get a morphism $\pi_2:G\times^B H\to\mathfrak{g}$ given by $\pi_2(g,h)=\Ad(g)h$. 
We denote the scheme-theoretic fiber over $s\in \mathfrak g$ by
\[\mathcal{B}(M,s) := \pi_2^{-1}(s).\]
Equivalently, $\mathcal B(M,s)$ is the scheme-theoretic vanishing locus of the section of the homogeneous vector bundle $G\times^B(\fg/H)\to \mathcal B$ given by $gB\mapsto [g,\Ad(g^{-1})s]$.

When $s \in \grs$ is regular semisimple, 
it is proved in \cite{demari-procesi-shayman} that this section intersects transversely with the zero section, so
$\mathcal B(M,s)$ is a smooth subvariety of $\mathcal B$ of pure codimension $\dim\,\fg/H = \lvert \Phi^+- M\rvert$. 
In this case, we call $\mathcal B(M,s)$ the \emph{regular semisimple Hessenberg variety associated with} $M$ (or $H$) and $s$.  
The variety $\mathcal{B}(M,s)$ is connected if and only if $\Delta\subset M$ (\cite[Corollary~9]{demari-procesi-shayman} or \cite[Proposition~A.1]{anderson-tymoczko}).

\subsection{Tangent and normal sequences}
\label{s.adj} 
From now on, let $M$ be the Hessenberg subset consisting of 
all the nonzero positive roots except for the maximal root $\theta$, let $H$ be the corresponding Hessenberg space, and $X:= {\mathcal{B}}(M,s)$ for $s\in \fg- \{0\}$. 

For each character $\lambda$ of the maximal torus $T$, we write $\cL(\lambda):=G\times^B\bk_\lambda$ for the corresponding line bundle on $\mathcal B$.
Then $X$ is a divisor in $\mathcal B$ associated to the line bundle $\cL(\theta)$. So, its ideal sheaf is $\mathcal{I}_X=\mathcal{L}(-\theta)$
and its normal bundle is $\mathscr{N}=\mathcal{L}(\theta)|_{X}$.

Write $\tang \mathcal{B}$ and $\tang X$ for the tangent bundles of $\mathcal{B}$ and $X$, respectively.
We get three related exact sequences:
\begin{align}
   &0\to \mathcal{O}_{\mathcal{B}}\to \mathcal{L}(\theta)\to\mathscr{N}\to 0;\label{e.nr}\\
   &0\to \tang X \to \tang \mathcal{B}|_{X} \to \mathscr{N}\to 0;\label{e.tn}\\
   &0\to \tang \mathcal{B}\otimes\mathcal{I}_X\to \tang \mathcal{B}\to\tang \mathcal{B}|_{X}\to 0.
   \label{e.tr}
\end{align}
Here, the sequences \eqref{e.nr} and \eqref{e.tr} arise from $0 \to \mathcal I_X \to\mathcal O_{\mathcal B} \to \mathcal O_X \to 0$, the former by tensoring with $\mathcal L(\theta)$ and the latter by tensoring with $T\mathcal B$. 
We view each of the sheaves above as sheaves on $\mathcal{B}$ and, for each such
sheaf $\mathcal{F}$, we write
\[
h^i(\mathcal{F}) := \dim\coh^i(\mathcal{B},\mathcal{F}).
\]

\begin{proposition}\label{seqs}
   Write $D$ for the dimension of the image of the map
   $\coh^0(X,\mathscr{N})\to\coh^1(X,\tang X)$ induced by the exact sequence~\eqref{e.tn}.
   Then we have
   \[
     D = h^0(\tang X) + h^0(\tang \mathcal{B}\otimes\mathcal{I}_X)
     - h^1(\tang \mathcal{B}\otimes\mathcal{I}_X) - 1.
   \]
\end{proposition}
\begin{proof}
  By~\cite{demazure}, we know that
  \begin{equation}\label{e.Dcalc}
    h^i(\mathcal{L}(\theta)) = h^i(\tang \mathcal{B}) =
    \begin{cases}
      \dim G & \text{ if }  i =0;\\
      0 &       \text{ otherwise}.
    \end{cases}
  \end{equation}
  From the long exact sequence for~\eqref{e.tn}, we then get that
  \begin{equation}\label{e.Dim}
    D = h^0(\mathscr{N}) - h^0(\tang \mathcal{B}|_{X}) + h^0(\tang X).
  \end{equation}
  From the exact sequence~\eqref{e.nr}, the computation in~\eqref{e.Dcalc}, and the fact that
  $h^i(\mathcal{O}_{\mathcal{B}}) = 0$ for $i>0$, we get that
  \begin{equation}\label{e.Ncalc}
  h^i(X,\mathscr{N}) =
  \begin{cases}
    \dim G -1 & \text{ if }i=0;\\
    0 &\text{ otherwise.}
  \end{cases}
  \end{equation}
  And from the exact sequence~\eqref{e.tr} and the computation in~\eqref{e.Dcalc}, we get that
  \begin{align*}
    h^0(\tang \mathcal{B}|_{X}) &= h^1(\tang \mathcal{B}\otimes\mathcal{I}_X) +
    h^0(\tang \mathcal{B}) - h^0(\tang \mathcal{B}\otimes\mathcal{I}_X)\\
      &= h^1(\tang \mathcal{B}\otimes\mathcal{I}_X) +
    \dim G - h^0(\tang \mathcal{B}\otimes\mathcal{I}_X). 
  \end{align*}
Substituting the last equation as well as~\eqref{e.Dcalc}
and \eqref{e.Ncalc} into~\eqref{e.Dim}, yields the desired result.
\end{proof}

\begin{theorem}\label{prop: cohomology vanishing}
  Let $G$ and $X$ be as above. Then, we have the following.
  \[
    \coh^i(\mathcal{B}, \tang \mathcal{B}\otimes\mathcal{I}_X) \cong
    \begin{cases}
      \bk & \text{ if $G$ of type $A_1$ and $i=0$};\\
      \bk & \text{ if $G$ of type $A_2$ and $i=1$};\\
      0 & \text{ otherwise.}
    \end{cases}
  \]
\end{theorem}

We prove Theorem~\ref{prop: cohomology vanishing} in Subsection~\ref{ss_van}
below.
Assuming this cohomology vanishing  we compute the terms appearing in
Proposition~\ref{seqs} obtaining a proof of 
Theorem~\ref{t_comp} from the introduction.

Before starting the proof of Theorem~\ref{t_comp}, we make a few observations about the spaces and bundles involved.

Let us recall how~\eqref{e.Dcalc} was proved in~\cite{demazure}.
Consider the exact sequence
\begin{equation}\label{txtheta}
  0\to E_{\theta}\to T\mathcal{B}\to \mathcal L(\theta)\to 0
\end{equation}
that appears in the
proof (on page 184) of 
\cite[Proposition 2]{demazure}. 
Demazure shows that $\coh^*(\mathcal{B},E_{\theta})=0$ and 
that $\coh^*(\mathcal{B}, \mathcal L(\theta))$
is concentrated in degree $0$ with $ \coh^0(\mathcal{B}, \mathcal L(\theta))=\mathfrak g$ as a $G$-representation. 
In particular, $\coh^0(\mathcal B, \tang \mathcal B)$, $\coh^0(\mathcal B, \cL(\theta))$ and $\fg$ are all isomorphic as $G$-representations. Since these are irreducible, the isomorphisms among them are unique up to scalar multiplication by Schur's lemma. 

Consequently, there exists a $G$-equivariant isomorphism $\fg \cong \coh^0(\cB, \tang\cB)$, which is unique up to scalar. We now give an explicit description of such an isomorphism.
 
\begin{lemma}\label{c_act}
Consider the map $\mathfrak{g} \to \coh^0(\mathcal{B}, T\mathcal{B})$ that sends each $s\in\mathfrak{g}$ to the vector field on~$\mathcal{B}$ generated by~$s$ under the adjoint action of~$G$. 
Then this map is an isomorphism of $G$-representations, where both sides carry the adjoint action.
\end{lemma}
\begin{proof}
The map is nonzero and $G$-equivariant.
Since both sides are irreducible $G$-representations,
it is an isomorphism by Schur's lemma.
\end{proof}

\begin{lemma}
  \label{c_div}
  Write $\varphi:\mathfrak{g}\to\coh^0(\mathcal{B}, \mathcal L(\theta))$.
  Then, for $s \in \mathfrak{g}$, $\mathcal{B} (M,s)$
  is precisely the vanishing locus
  of the section $\varphi(s)$.
\end{lemma}
\begin{proof} 
The map $\mathfrak g\to \coh^0(\mathcal B, \tang \mathcal B)$ sends $s$ to the section of $\tang \mathcal B=G\times^B(\fg/\mathfrak{b})$ given by $gB\mapsto [g,\Ad(g^{-1})s]$, hence $\varphi$ sends $s$ to the section of $\mathcal L(\theta)=G\times^B(\mathfrak g/H)$ of the same form.
Thus the vanishing locus of $\varphi(s)$ is $\mathcal B(M,s)$ by definition.
\end{proof}

\begin{lemma}\label{ss_fact}
  Suppose $x,s\in\mathfrak{g}$ with $s$ semisimple, and suppose $[x,s]=c s$ for some $c\in\bk$.
  Then $c=0$.
\end{lemma}
\begin{proof}
By assumption, $\ad(s)x=-cs$, hence $\ad^2(s)x=0$.
Thus $x$ is in the generalized $0$-eigenspace of $\ad(s)$.
Since $s$ is semisimple, $\ad(s)$ is semisimple, and therefore $\ad(s)x=0$. This forces $c=0$.
\end{proof}

From now on, we assume that $s\in \grs$ is regular semisimple, as in Theorem~\ref{t_comp}.
\begin{proposition}\label{p_tan}
  Let $G$ and $X$ be as in Theorem~\ref{t_comp}.
  Then, $h^0(X,TX)= \rank G$.
\end{proposition}

\begin{proof}
  Recall that we have $X= {\mathcal{B}}(M,s)$ for $s$ a regular semisimple element of $\mathfrak{g}$.
  Set $T=Z_G(s)$, the centralizer of $s$ in $G$, and let $\mathfrak{t}$ denote the Lie algebra of $T$, which is the
  centralizer of $s$ in $\mathfrak{g}$.
  Then $T$ acts on $X$ and the derivative of that action gives a map $\mathfrak{t}\to\coh^0(X,TX)$. 
Considering the moment graph of $X$ (see~\cite{AHTMMS20}), the weights appearing on the edges  of the moment graph of $X$ span $\mathfrak{t}$. We have only to check this at the vertex corresponding to the identity. This is the case  because $M$ contains every simple root. 
  Therefore, the action of $T$ on $X$ is effective.
  In other words, the homomorphism $T\to \Aut X$ is injective.
  So we can regard $\mathfrak{t}$ as a subspace of $\coh^0(X,TX)$.

  Now, we have an isomorphism $\mathfrak{g}=\coh^0(\mathcal{B},T\mathcal{B})$ by Lemma~\ref{c_act}, and, from
  Theorem~\ref{prop: cohomology vanishing}, we have another isomorphism
  $\coh^0(\mathcal{B},T\mathcal{B})\cong\coh^0(X,T\mathcal{B}|_{X})$.
  It follows that any vector field in $\coh^0(X,TX)$ is the restriction of a
  vector field in $\coh^0(\mathcal{B},T\mathcal{B})=\mathfrak{g}$.

  So suppose $g=e^h\in \Aut^\circ X$ is the exponential of a vector field $h\in\mathfrak{g}$.
  Then, viewing $g$ as an automorphism of $\mathcal{B}$, we see that $g(X)=X$.
  But then it follows from Lemma~\ref{c_div} that $\Ad(g)s$ must be a multiple of $s$.
  So $\ad(h)s$ is a multiple of~$s$.
  But, since $s$ is semisimple, this implies that $\ad(h)s=0$ by Lemma~\ref{ss_fact}, thus
  that $h\in\mathfrak{t}$.
  This shows that $\coh^0(X,TX)=\mathfrak{t}$, which proves Proposition~\ref{p_tan} as $\rank G =\dim\mathfrak{t}$.
\end{proof}

\begin{proof}[Proof of Theorem~\ref{t_comp}]
  From Theorem~\ref{prop: cohomology vanishing}, it follows that the map
  $\coh^i(\mathcal{B},T\mathcal{B})\to \coh^i(\mathcal{B},T\mathcal{B}|_{X})$ is an isomorphism for all $i$.
  But we
  know from~\cite[Proposition~2 (on page 182)]{demazure} that $\coh^i(\mathcal{B},T\mathcal{B})=0$ for
  $i>0$.
  It then follows from the short exact sequence \eqref{e.tn} that the connecting homomorphism
  \beq\label{eq:connecting.homo}\coh^i(X, \mathscr N)\lra\coh^{i+1}(X,TX)\eeq
  is surjective for $i=0$ and an isomorphism for $i>0$.
  Using the computation in~\eqref{e.Ncalc}, we conclude that $h^i(X,TX)=0$ for $i>1$.

  We now show that the Kodaira--Spencer map~\eqref{eq:KSmap.intro} is surjective.
	By construction, it factors as
  \[\fg\xrightarrow{~\varphi~}\coh^0(\cB,\cL(\theta))\lra \coh^0(X,\mathscr{N})\lra \coh^1(X,\tang X)\]
  where the second map is the restriction induced by the short exact sequence~\eqref{e.nr},
  and the third map is the connecting homomorphism~\eqref{eq:connecting.homo}.
  Indeed, by Lemma~\ref{c_div}, the map $\varphi$ is precisely the Kodaira--Spencer map
  associated to the family of divisors in $\mathcal{B}$ given by $\pi_2$.
  The connecting homomorphism~\eqref{eq:connecting.homo} corresponds to the natural map from embedded deformations to abstract deformations.

	The restriction map
	$\coh^0(\mathcal{B},\mathcal{L}(\theta))\to \coh^0(X,\mathscr N)$ is surjective,
	since $\coh^1(\mathcal{B},\cO_\cB)=0$.
	We have already shown that $\varphi$ is an isomorphism and  the connecting homomorphism is surjective.
	Therefore, their composition, which coincides with the Kodaira--Spencer map
	\eqref{eq:KSmap.intro}, is surjective. 

  Finally Proposition~\ref{p_tan}, together with Propositions~\ref{seqs} and Theorem~\ref{prop: cohomology vanishing}, implies that $h^1(X,TX)=\rank G-1$.
  This completes the proof of Theorem~\ref{t_comp}.
\end{proof}

It remains to prove Theorem~\ref{prop: cohomology vanishing}. 
When $G$ is of type $A_1$, then $\mathcal{B}$ is $\mathbb P^1$ and $\tang \mathcal{B} \otimes \mathcal{I}_X$ is a trivial vector bundle of rank one. Thus
 \[
    \coh^i(\mathcal{B}, \tang \mathcal{B}\otimes\mathcal{I}_X) \cong
    \begin{cases}
        \bk &  \text{ if } i=0;\\
        0 & \text{ otherwise.}
    \end{cases}
  \]
So, from now on, we will assume that $G$ is not of type $A_1$.  
 
\subsection{Regular weights} \label{s.notations}

We assume that $\Phi$ is irreducible (as implied by the hypotheses on $G$ in
Theorem~\ref{prop: cohomology vanishing}).
For $\alpha \in \Phi$, let $\alpha^{\vee}$ be the coroot associated with~$\alpha$. 
We denote by $\langle \,,\,\rangle$ the pairing between the weight lattice and the coroot lattice. The Weyl group $W$ of $G$ acts on the weight lattice: For a simple reflection $s_{\alpha}$ and a weight $\lambda$, $s_{\alpha}(\lambda)=\lambda - \langle \lambda, \alpha^{\vee} \rangle \alpha$. 
Denote
by $\rho $ the half sum of positive roots. Then $\langle \rho,\alpha^{\vee}\rangle =1$ for any simple root $\alpha$. 

We say that an integral weight $\lambda$ is \emph{dominant} (\emph{regular}, respectively) if $\langle\lambda, \alpha^{\vee} \rangle \geq 0$ ($\langle\lambda, \alpha^{\vee} \rangle\not= 0$, respectively) for all $\alpha \in \Phi^+$. 
We call a weight $\lambda$ \emph{singular} if it is not regular. 
For a regular dominant weight $\lambda$, denote by $V_{\lambda}$ the irreducible representation of $G$ with highest weight $ \lambda$.

\begin{theorem}[{Borel--Weil--Bott theorem for $\mathcal B$ \cite[Theorem in Section 4.3]{Ak}}]  \label{thm BWB} 
Let $\lambda$ be an integral weight of $G$. If $\lambda + \rho$ is   regular, then there is a unique $w \in W$ such that $w(\lambda + \rho)$ is regular dominant and  
\[
\coh^{i}(\mathcal B, \mathcal L(\lambda)) \cong
\begin{cases}
V_{w(\lambda+\rho)-\rho} & \text{ if }i=\ell(w); \\
0 & \text{ otherwise.}
\end{cases}
\]
Otherwise, $\coh^i(\mathcal B, \mathcal L(\lambda))$ vanishes for all $i$.  

\end{theorem}
  
Recall that  $\theta$  denotes the maximal root of the root system $\Phi$. 
When $\Phi$ has roots of different lengths, let $\theta^+$ denote the maximal
short root, the maximal root among all short roots.
By Proposition 6 in Section 4.7 of \cite{Ak}, there is a unique $k$ such that
$\theta^+ + \alpha_k$ is a root. 
Set $\theta^{++}=s_k(\theta^+ +
\alpha_k)$.
Then $\theta^+, k, \theta^{++}$ are given as follows.

\begin{center}
\begin{tabular}{c||c|c|c}
\toprule
 $\Delta$ & $\theta^+$ & $k$ & $\theta^{++}$  \\
\midrule
  $B_r$ & $\alpha_1 + \alpha_2 + \dots + \alpha_{r-1} + \alpha_r$ & $r$ 
    & $\alpha_1 + \alpha_2 + \dots + \alpha_{r-1}$ \\
  \hline
  $C_r$ & $\alpha_1 + 2 \alpha_2 \dots + 2 \alpha_{r-1} + \alpha_r$ & $1$ 
    & $2 \alpha_2 + \dots + 2 \alpha_{r-1} + \alpha_r$ \\
   \hline
  $F_4$ & $\alpha_1 + 2 \alpha_2 + 3 \alpha_3 + 2 \alpha_4$ & $3$ 
    & $\alpha_1 + 2 \alpha_2 + 2 \alpha_3 + 2\alpha_4$ \\
   \hline
  $G_2$ & $2 \alpha_1 + \alpha_2$ & $1$ &$\alpha_2$ \\
\bottomrule
\end{tabular}
\end{center}
\noindent
We follow the conventions in \cite{bourbaki-lie-4-6} for the indices of simple roots. See Table~\ref{table_Dynkin_diagram}.
\begin{table}[ht]
\tikzstyle{Dnode}=[draw, circle, inner sep = 0.07cm]
\tikzstyle{double line} = [
decoration={
    markings,
    mark=at position 0.55 with {\arrow[line width = 0.5pt,scale=1]{angle 90}}},
	double distance = 1.5pt, 
	double=\pgfkeysvalueof{/tikz/commutative diagrams/background color},
	postaction={decorate}
]

\tikzstyle{triple line} = [
decoration={
    markings,
    mark=at position 0.55 with {\arrow[line width = 0.5pt,scale=1]{angle 90}}},
	double distance = 2pt, 
	double=\pgfkeysvalueof{/tikz/commutative diagrams/background color},
	postaction={decorate}
]

\begin{center}
\begin{tabular}{l|l}
\toprule
$\Phi$ & Extended Dynkin diagram \\
\midrule
$A_{r}$ $(r \geq 2)$  &
\begin{tikzpicture}[scale=.5, baseline=-.5ex]
\tikzset{every node/.style={scale=0.7}}
\node[Dnode, label=below:$1$] (1) {};
\node[Dnode, label=below:$2$] (2) [right = of 1] {};
\node[Dnode, label=below:$3$] (3) [right = of 2] {};
\node[Dnode, label=below:$r-1$] (4) [right =of 3] {};
\node[Dnode, label=below:$r$] (5) [right =of 4] {};			
\node[Dnode] (6) [above =of 3] {};			

\draw (1)--(2)--(3)
(4)--(5)
(1)--(6)--(5);
\draw[dotted] (3)--(4);

\end{tikzpicture}  \\

$B_2$ & 
\begin{tikzpicture}[scale=.5, baseline=-.5ex]
\tikzset{every node/.style={scale=0.7}}
    \node[Dnode] (0) {};
    \node[Dnode, label=below:$2$] (2) [right = of 0] {}; 
    \node[Dnode, label=below:$1$] (1) [right = of 2] {};
    \draw[double line] (0)--(2); 
    \draw[double line] (1)--(2);
\end{tikzpicture}
\\ 

$B_{r}$ $(r \geq 3)$  &
\begin{tikzpicture}[scale=.5, baseline=-.5ex]
\tikzset{every node/.style={scale=0.7}}

\node[Dnode, label=below:$2$] (1) {};
\node[Dnode, label=below:$3$] (2) [right = of 1] {};
\node[Dnode,label=below:$r-2$] (3) [right = of 2] {};
\node[Dnode,label=below:$r-1$] (4) [right =of 3] {};
\node[Dnode, label=below:$r$] (5) [right =of 4] {};
\node[Dnode] (6) [below left = 0.6cm and 0.6cm of 1] {};
\node[Dnode, label=right:$1$] (7) [above left = 0.6cm and 0.6cm of 1] {};			

\draw (1)--(2)
(3)--(4)
(6)--(1)--(7);
\draw [dotted] (2)--(3);
\draw[double line] (4)--(5);
\end{tikzpicture}  \\[2em] 	

$C_{r}$ $(r \geq 2)$ & 
\begin{tikzpicture}[scale=.5, baseline=-.5ex]
\tikzset{every node/.style={scale=0.7}}

\node[Dnode, label = below:$1$] (1) {};
\node[Dnode, label = below:$2$] (2) [right = of 1] {};
\node[Dnode, label = below:$r-2$] (3) [right = of 2] {};
\node[Dnode, label = below:$r-1$] (4) [right =of 3] {};
\node[Dnode, label = below:$r$] (5) [right =of 4] {};			
\node[Dnode] (6) [left =of 1] {};
\draw (1)--(2)
(3)--(4);
\draw [dotted] (2)--(3);
\draw[double line] (5)--(4); 
\draw[double line] (6)--(1);
\end{tikzpicture}  \\[2em] 	

$D_{r}$ $(r \geq 4)$ & 
\begin{tikzpicture}[scale=.5, baseline=-.5ex]
\tikzset{every node/.style={scale=0.7}}

\node[Dnode, label=below:$2$] (1) {};
\node[Dnode, label=below:$3$] (2) [right = of 1] {};
\node[Dnode, label=below:$r-3$] (3) [right = of 2] {};
\node[Dnode, label=right:$r-2$] (4) [right =of 3] {};			

\node[Dnode, label = right:$r$] (5) [below right = 0.6cm and 0.6cm of 4] {};
\node[Dnode, label = right:$r-1$] (6) [above right=  0.6cm and 0.6cm of 4] {};

\node[Dnode, label=right:$1$] (7) [above left = 0.6cm and 0.6cm of 1] {};
\node[Dnode] (8) [below left=0.6cm and 0.6cm of 1] {};

\draw(1)--(2)
(3)--(4)--(5)
(4)--(6)
(7)--(1)--(8);
\draw[dotted] (2)--(3);
\end{tikzpicture}
\\ 
$E_6$ & 

\begin{tikzpicture}[scale=.5, baseline=-.5ex]
\tikzset{every node/.style={scale=0.7}}

\node[Dnode, label=below:$1$] (1) {};
\node[Dnode, label=below:$3$] (3) [right=of 1] {};
\node[Dnode, label=below:$4$] (4) [right=of 3] {};
\node[Dnode, label=right:$2$] (2) [above=of 4] {};
\node[Dnode, label=below:$5$] (5) [right=of 4] {};
\node[Dnode, label=below:$6$] (6) [right=of 5]{};
\node[Dnode] (7) [above=of 2]{};

\draw(1)--(3)--(4)--(5)--(6)
(7)--(2)--(4);
\end{tikzpicture}			
\\[1em]
$E_7$ & 
\begin{tikzpicture}[scale=.5, baseline=-.5ex]
\tikzset{every node/.style={scale=0.7}}

\node[Dnode, label=below:$1$] (1) {};
\node[Dnode] (8) [left=of 1] {};
\node[Dnode, label=below:$3$] (3) [right=of 1] {};
\node[Dnode, label=below:$4$] (4) [right=of 3] {};
\node[Dnode, label=right:$2$] (2) [above=of 4] {};
\node[Dnode, label=below:$5$] (5) [right=of 4] {};
\node[Dnode, label=below:$6$] (6) [right=of 5]{};
\node[Dnode, label=below:$7$] (7) [right=of 6]{};

\draw (8)--(1)--(3)--(4)--(5)--(6)--(7)
(2)--(4);
\end{tikzpicture}	
\\[1em]
$E_8$ & 
\begin{tikzpicture}[scale=.5, baseline=-.5ex]
\tikzset{every node/.style={scale=0.7}}

\node[Dnode, label=below:$1$] (1) {};
\node[Dnode, label=below:$3$] (3) [right=of 1] {};
\node[Dnode, label=below:$4$] (4) [right=of 3] {};
\node[Dnode, label=right:$2$] (2) [above=of 4] {};
\node[Dnode, label=below:$5$] (5) [right=of 4] {};
\node[Dnode, label=below:$6$] (6) [right=of 5]{};
\node[Dnode, label=below:$7$] (7) [right=of 6]{};
\node[Dnode, label=below:$8$] (8) [right=of 7]{};
\node[Dnode] (9) [right=of 8]{};

\draw(1)--(3)--(4)--(5)--(6)--(7)--(8)--(9)
(2)--(4);
\end{tikzpicture}
\\
$F_4$ 
&
\begin{tikzpicture}[scale = .5, baseline=-.5ex]
\tikzset{every node/.style={scale=0.7}}

\node[Dnode] (1) {};
\node[Dnode, label=below:$1$] (2) [right = of 1] {};
\node[Dnode, label=below:$2$] (3) [right = of 2] {};
\node[Dnode, label=below:$3$] (4) [right =of 3] {};
\node[Dnode, label=below:$4$] (5) [right =of 4] {};

\draw (1)--(2)
(2)--(3)
(4)--(5);
\draw[double line] (3)--(4);
\end{tikzpicture}  \\
$G_2$
&
\begin{tikzpicture}[scale =.5, baseline=-.5ex]
\tikzset{every node/.style={scale=0.7}}

\node[Dnode, label=below:$1$] (1) {};
\node[Dnode, label=below:$2$] (2) [right = of 1] {};
\node[Dnode] (3) [right=of 2] {};

\draw[triple line] (2)--(1); 
\draw (1)--(2);
\draw (2)--(3);

\end{tikzpicture}\\
\bottomrule
\end{tabular}
\end{center}
\caption{Extended (completed) Dynkin diagrams of root systems}\label{table_Dynkin_diagram}
\end{table}

\begin{proposition} [Theorem in Section 4.7 of \cite{Ak}, Lemma 2 of \cite{demazure}]  \label{prop: regular
  weight 1} Let $\alpha \in \Phi$ be a root such that $\alpha+ \rho$ is regular.
  Then $\alpha$ is one of the following:
  \[
    \theta , \quad   \theta^+ , \quad
    \theta^{++} , \quad   -\alpha_i    , \text{ where } 1 \leq i \leq r .
  \]
  The unique element $w \in W$ such that  $w(\alpha+ \rho)$ is regular dominant is  
$ 
   id ,   id , 
    s_k ,    s_i     $, where $1 \leq i \leq r,   
$ 
  respectively, and the corresponding  $w(\alpha+\rho) - \rho$ is 
$ 
\theta ,    \theta^+,  
    \theta^+ ,     0  $, 
     respectively. 
 \end{proposition}

We remark that  Proposition~\ref{prop: regular weight 1} is a main step in the proof of \eqref{e.Dcalc}. 
 For details, see \cite{Ak}, \cite{demazure}.
  Similarly, to compute $\coh^i(\mathcal{B}, T\mathcal{B} \otimes \mathcal I_X)$, we need to know the set of roots~$\alpha$ such that $\alpha -\theta+ \rho$ are  regular, which are given in Proposition~\ref{prop: regular weight 2}. Before stating it, we introduce some notations.

We define a subset $\Delta_0$ of $\Delta$  by
\[
\Delta_0:=\{\alpha  \in \Delta \mid \langle \theta,\alpha^{\vee} \rangle =0\}. 
\]
More explicitly:
\begin{itemize} 
\item If $\Delta$ is of type $A_r$ with $r \geq 2$, then   
$\Delta_{0}=\Delta \backslash \{\alpha_1, \alpha_r\}$.
 \item  If $\Delta$ is not of type $A_r$, then 
$\Delta_{0}=\Delta \backslash \{\alpha_{i_0}\}$, where $i_0$ is the index of the (unique) vertex adjacent to the non-indexed vertex in the extended Dynkin diagram. See Table~\ref{table_Dynkin_diagram}.
\end{itemize} 
For the   root system $ \Phi_0 $ generated by $\Delta_{0}$, define $\theta_{0}, \theta_{0}^+, \theta_{0}^{++}$ similarly. 
For the reader's convenience, we list up $\theta_0, \theta_0^+, \theta_0^{++}$. See Table~\ref{table theta0}. 

\begin{table}[h]
\begin{center}
\begin{tabular}{c||c|c}
\toprule
$\Delta$ & $i_0 $ & $\theta_{0}$   \\
\midrule 
$A_r$ &   &  $\alpha_2 + \dots + \alpha_{r-1}$  \\
$B_r$ $(r\geq 3) $ &  $2$ &  $(\alpha_1) + (\alpha_3 + 2 \alpha_4 + \dots + 2 \alpha_r)$ \\
$C_r$ ($r \geq 3$) & $1$ &   $2 \alpha_2 + \dots + 2 \alpha_{r-1} + \alpha_r$ \\
$C_2$  & $1$ &   $ \alpha_2$ \\
$D_r$ $(r \geq 4)$ &  $2$ &  $(\alpha_1) + (\alpha_3 + 2 \alpha_4 + \dots + 2 \alpha_{r-2} +\alpha_{r-1} + \alpha_r)$   \\
$E_6$ & $2$ &   $\alpha_1 + \alpha_3 + \alpha_4 + \alpha_5 + \alpha_6$  \\
$E_7$ & $1$&    $\alpha_2 + \alpha_3 + 2 \alpha_4 + 2 \alpha_5 + 2 \alpha_6 + \alpha_7$  \\
$E_8$ & $8$ &     $2 \alpha_1 + 2 \alpha_2 + 3 \alpha_3 + 4 \alpha_4 + 3 \alpha_5 + 2 \alpha_6 + \alpha_7$   \\
$F_4$ & $1$&  $\alpha_ 2 + 2 \alpha_3 + 2 \alpha_4$  \\
$G_2$ & $2$&   $\alpha_1$ \\
\bottomrule 
\end{tabular} \\
\medskip 
\begin{tabular}{c||c|c}
\toprule 
$\Delta$ &      $\theta_{0}^+$ & $\theta_{0}^{++}$ \\
\midrule 
$B_r$ $(r \geq 3)$ &    $\alpha_3 + \dots + \alpha_r$ & $\alpha_3 + \dots + \alpha_{r-1}$\\
$C_r$ $(r \geq 3)$ &     $\alpha_2 + 2 \alpha_3+\dots + 2 \alpha_{r-1} + \alpha_r$ 
    & $2 \alpha_3 + \dots + 2 \alpha_{r-1} + \alpha_r$\\
$F_4$ &   $\alpha_2 + 2 \alpha_3 + \alpha_4$ 
    & $\alpha_2 + 2 \alpha_3$\\
\bottomrule 
\end{tabular}
\end{center}
\caption{$\theta_0$, $\theta_0^+$, $\theta_0^{++}$ for $\Delta_0$}
\label{table theta0}
\end{table}

\begin{proposition} \label{prop: regular weight 2}
Let $\alpha \in \Phi$ be such that $\alpha - \theta + \rho$ is regular  
and let $w \in W$ be the unique element in $W$ such that $w(\alpha-\theta +\rho)$ is regular dominant. 
 We divide the cases into three:  
\begin{enumerate} 
\item $\alpha \in \Phi^+$ and $\langle\theta, \alpha^{\vee}\rangle >0$;
\item $\alpha \in \Phi^+$ and $\langle\theta, \alpha^{\vee}\rangle=0$;
\item $\alpha \in \Phi^-$. 
  \end{enumerate} 
We list $\alpha$  and   $w$ in Table~\ref{table regular}.

\begin{table}
\begin{center}
\begin{tabular}{c|c||c|c|c|c}
\toprule 
    & $\Delta$ & $\alpha$ & $w$ & $w(\alpha-\theta + \rho) - \rho$ & $\ell(w)$\\
\midrule 
 (1) & any type & $\theta$ & id & $0$ & $0$\\
   & $A_r$ $(r \geq 2)$  & $\theta-\alpha_1, \theta - \alpha_r$ & $s_1, s_r$  & $0$ &$1$\\
  & $\not=A_r$ & $\theta - \alpha_{i_0}$ & $s_{i_0}$ & $0$ & $1$\\
  \midrule
  (2) & $A_r$ $( r \geq 3 )$ & $\theta-\alpha_1-\alpha_r$ & $s_{1}s_{r}$ & $0$ & $2$\\
    & $B_r, D_r$  $( r \geq 4 )$ & $\theta_{0} $  & $s_3 s_2$ & $0$ & $2$\\
    & $C_r$ $( r \geq 3 )$  & $\theta_{0}^+$ & $s_2 s_1$ & $0$ & $2$\\
    & $C_r$ $( r \geq 3 )$ & $\theta_{0}^{++}$ & $s_1s_2s_1$ & $0$ &$3$\\
   \midrule
   (3)& $A_2$ & $-\alpha_1, -\alpha_2$   & $s_2s_1, s_1s_2$  & $0$ &$2$\\
    & $A_2$  & $-\theta$  & $s_1s_2s_1$  & $0$ & $3$ \\
     & $C_2$ &  $-\alpha_{1}$  & $s_2s_1$   & $0$ & $2$ \\
      & $C_2$  &  $-\theta$  &  $s_1s_2s_1$ & $0$ & $3$\\
\bottomrule 
\end{tabular}
\end{center}
\caption{$\alpha$ and $w$ such that $w(\alpha-\theta+\rho)$ is regular dominant}
\label{table regular}
\end{table}

\end{proposition}

For an element $w$ in the Weyl group~$W$, let $\Phi(w)=\{\alpha \in \Phi^+ \mid w^{-1}\alpha \in  \Phi^-\}$.
If $w=s_{ i_1 }\cdots s_{ i_{\ell} }$ is a reduced expression of $w$, then $\Phi(w)$ is given by
\begin{equation}\label{eq_Phi_w}
\{\alpha_{i_1}, s_{ i_1 }(\alpha_{i_2}), \dots, s_{ i_1 }s_{ i_2 }\cdots s_{ i_{\ell-1} }(\alpha_{i_{\ell}})\},
\end{equation}
where the roots in $\Phi(w)$ are pairwise distinct so that $|\Phi(w)| = \ell(w)$. (Corollary 2 to Proposition VI.1.17 of \cite{bourbaki-lie-4-6}). 
 
\begin{lemma} [Lemma 2.13 of \cite{Mac}] \label{lem: negative sum of positive roots}  Let $E$ be a subset of $\Phi^+$.  If $\rho- \sum_{\alpha \in E}\alpha$ is regular and $w$ is the unique element of $W$ such that  $w^{-1} \left(\rho- \sum_{\alpha \in E}\alpha \right)$ is regular dominant, then $ w^{-1} (\rho- \sum_{\alpha \in E}\alpha) =\rho $  and  $E=\Phi(w )$. 
\end{lemma}

\begin{proof}[Proof of Proposition~\ref{prop: regular weight 2}]  \label{proof_of_prop:regular weight 2}
Let $\alpha \in \Phi$. Assume that $\alpha - \theta + \rho$ is
  regular. We divide the cases as in the statement. 
  \\

\noindent {\sf Case (1)} If $\alpha \in \Phi^+$ and $\langle \theta, \alpha^{\vee} \rangle >0$, then $\theta-\alpha$ is either zero or a root. 
If $\theta - \alpha = 0$, then $\alpha = \theta$. Moreover, $\alpha - \theta  + \rho = \rho$ is a regular dominant weight. 
If $\theta - \alpha$ is a root, so is $\alpha-\theta$. Since $\alpha- \theta+\rho$ is regular by the assumption, applying Proposition~\ref{prop: regular weight 1}, we obtain that $\alpha -\theta$ is one of the following: $\theta ,   \theta^+ ,    \theta^{++} ,    -\alpha_i    , \text{ where } 1 \leq i \leq r .$ Since $\theta$ is maximal, $\alpha - \theta$ is $-\alpha_i$ for some $i$, and thus $ \theta - \alpha_i$ is a root. 
This can happen only when $i=1$ or $r$ if $\Delta$ is of type $A_r$, and when $i=i_0$ if $\Delta$ is of other type. 
Indeed, since $\theta+ \alpha_i$ is not a root while both $\theta$ and $\theta - \alpha_i$ are, 
it follows that $\langle \theta, \alpha_i^{\vee} \rangle  \geq 1$   
(Proposition 9 of Section 1 of Chapter VI of \cite{bourbaki-lie-4-6}) but $\langle \theta, \alpha_i^{\vee} \rangle =0$ for any $\alpha_i  \in \Delta_0$. 
\\

\noindent {\sf Case (2)} If $\alpha \in \Phi^+$ and $\langle \theta, \alpha^{\vee} \rangle =0$, then   $\alpha$ is a root in the root system $\Phi_{0}$.  Denote by $G_{0}$ the subgroup of $G$ with the root system $\Phi_{0}$ and by $W_{0}$ and $\rho_{0}$ the corresponding Weyl group and the half sum of positive roots.
We claim that  $\alpha   + \rho_{0}$ is a regular weight of~$G_{0}$.

Suppose, to the contrary, that $\alpha+\rho_{0}$ is a singular weight of $G_{0}$. Then
$\langle w(\alpha+\rho_{0}), \alpha_i^{\vee} \rangle =0$ for some $w \in W_{0}$ and $\alpha_i
\in \Delta_0$. 
Note that $\langle \theta, \beta^{\vee} \rangle = 0$ and $\langle \rho-\rho_0 , \beta^{\vee} \rangle =0 $ for any
$\beta \in \Phi_{0}$. The second follow from that fact that $\langle \rho, \alpha^{\vee} \rangle =1$ for any   $\alpha \in \Delta$ and $\langle \rho_0, \alpha^{\vee} \rangle =1$ for any $\alpha \in \Delta_0$. Thus we have  $\langle w(\alpha - \theta + \rho ), \alpha_i^{\vee}\rangle 
=\langle w(\alpha+ \rho_{0}), \alpha_i^{\vee} \rangle =0 $, contradicting to the assumption that
$\alpha - \theta + \rho$ is regular. Therefore, $ \alpha   + \rho_{0}$ is a
regular weight of $G_{0}$.

By Proposition~\ref{prop: regular weight 1},  $\alpha $ is either
$\theta_{0}$ or $\theta_{0}^+$  or $\theta_{0}^{++}$ or $-\alpha_i $
for $\alpha_i \in \Phi_0$. Among them, $\theta_0$, $\theta_0^+$, and $\theta_0^{++}$ are positive roots. The differences $\theta-\theta_0, \theta_0 - \theta_0^+, \theta_0^+ - \theta_0^{++}$ are given as follows:

\begin{center}
\begin{tabular}{c|  c|c|c}
\toprule 
$\Delta$ &     $\theta-\theta_{0}$ & $\theta_0-\theta_0^+$ & $\theta_0^+ -\theta_0^{++}$  \\
\midrule 
$A_r$ &     $\alpha_1 + \alpha_r$ & & \\
$B_r$ &     $2\alpha_2 + \alpha_3 $ & $\alpha_1 + (\alpha_4 + \dots + \alpha_r)$ & $\alpha_r$\\
$C_r$&      $2 \alpha_1$ & $\alpha_2$ & $\alpha_2$\\
$D_r$ &       $2\alpha_2 + \alpha_3$ &&\\
$E_6$ &     $\alpha_2 + \alpha_3 + 2 \alpha_4 + \alpha_5 $&&\\
$E_7$ &     $2 \alpha_1 + \alpha_2 + 2 \alpha_3 + 2 \alpha_4 + \alpha_5$&&\\
$E_8$ &      $\alpha_2 + \alpha_3 + 2 \alpha_4 + 2 \alpha_5 +2 \alpha_6 + 2 \alpha_7 + 2 \alpha_8$&& \\
$F_4$ &      $2 \alpha_1+2\alpha_2+2\alpha_3$ & $\alpha_4$ & $\alpha_4$\\
$G_2$ &       $2 \alpha_1 + 2 \alpha_2$&&\\
\bottomrule
\end{tabular}
\end{center}

We first consider the case when $\alpha= \theta_0$.
When $\Delta$ is of type $C_r$, we can see that $\rho - 2 \alpha_1$ is singular by considering 
\[
\langle s_2(\rho - 2 \alpha_1), \alpha_1^{\vee} \rangle = 
\langle \rho - 2 \alpha_1 - 3 \alpha_2, \alpha_1^{\vee} \rangle = 0.
\]
When $\Delta$ is not of type $C_r$, we can write $\theta- \theta_{0}$ as a sum $\beta+ \gamma$ of two distinct positive roots, $\beta, \gamma$. Then $\theta_{0} - \theta + \rho= \rho -(\beta + \gamma)$.   If $\theta_{0} - \theta + \rho$ is regular, then by Lemma~\ref{lem: negative sum of positive roots}, there is $w \in W$ such that $w(\rho) = \rho - (\beta+\gamma)$ and $\{\beta, \gamma\}=\Phi(w)$. Here, one of the positive roots, say $\beta$, should be a simple root, and $\gamma$ should be of the form $s_{\beta}(\delta)$ for some simple root $\delta \neq \beta$ as in~\eqref{eq_Phi_w}.
This can happen only when $\Delta$ is of type $A_r$ ($r \geq 3$) or of type $B_r,D_r$ ($r \geq 4$).

For the cases when $\alpha = \theta_0^+$ or $\theta_0^{++}$, one can see that each $\theta - \alpha$ can be written as a sum of distinct positive roots: 
\begin{itemize}
\item $B_r$:
\[
\begin{split}
\theta - \theta_0^+ & = 
(\alpha_1 + \alpha_2)+(\alpha_2 + \alpha_3  +  \alpha_4 + \cdots + \alpha_r),\\
\theta - \theta_0^{++} &= 
(\alpha_1 + \alpha_2)+(\alpha_2 + \alpha_3  +  \alpha_4 + \cdots + \alpha_r) + \alpha_r.
\end{split}
\]
\item $C_r$: 
\[
\begin{split}
\theta - \theta_0^+ &= 
\alpha_1+(\alpha_1+\alpha_2), \\
\theta - \theta_0^{++} &= 
\alpha_1 + (\alpha_1+ \alpha_2) + \alpha_2.
\end{split}
\]
\item $F_4$:
\[
\begin{split}
\theta - \theta_0^+ &= 
\alpha_1+(\alpha_1+2\alpha_2 + 2\alpha_3 + \alpha_4), \\
\theta - \theta_0^{++} &= 
\alpha_1+(\alpha_1+2\alpha_2 + 2\alpha_3 + \alpha_4) + \alpha_4.
\end{split}
\]
\end{itemize}
By the same arguments as in the previous case, $\alpha-\theta + \rho$ is regular only if $\Delta$ is of type $C_r$.
\medskip

\noindent {\sf Case (3)}  
If $\alpha \in \Phi^-$, then we have either $\alpha =-\theta$ or $\alpha\not=-\theta$. 
We consider two cases separately. 

Suppose that  
$\alpha = - \theta$. 
Since $\alpha - \theta + \rho = -2 \theta+ \rho$ is regular,   $\Delta$ is either of type $A_2$ or of type $C_2$. Indeed,
\begin{equation*}
\begin{array} {ll}
\langle s_3s_2s_1(-2\theta + \rho), \alpha_2^{\vee} \rangle =0 & \text{ if } \Delta \text{ is of type }C_3; \\
\langle s_3s_2s_1(-2\theta + \rho), \alpha_4^{\vee}\rangle =0 & \text{ if } \Delta \text{ is of type } C_r, r \geq 4; \\
\langle s_1(-2\theta+ \rho), \alpha_2^{\vee} \rangle =0  & \text{ if } \Delta \text{ is of type } A_r, r \geq 3; \\
\langle s_{i_1}(-2\theta+ \rho), \alpha_{i_0}^{\vee}\rangle =0  & \text{ otherwise }
\end{array} 
\end{equation*}
where $s_{i_1}$ is any simple reflection $s_{\alpha_{i_1}}$   with $\langle \alpha_{i_0}, \alpha_{i_1}^{\vee} \rangle \not=0$.

Suppose that  $\alpha \not=-\theta$. Write $\alpha = - \gamma$ for some positive root $\gamma$. 
Then $\gamma$ is a positive root not equal to $ \theta$. Since $\alpha - \theta + \rho = -\gamma-\theta+\rho$ is regular,   then $\gamma$ is one of the following:
\begin{enumerate}
\item [(a)] $\Delta$ is of type $A_2$ and $\gamma $ is $ \alpha_1$ or $\alpha_2$;
\item [(b)] $\Delta$ is of type $C_2$ and $\gamma$ is $\alpha_1$.
\end{enumerate}
To see this, applying  Lemma~\ref{lem: negative sum of positive roots} to $E=\{\gamma, \theta\}$, we get an element $w \in W$ such that  $\Phi(w)=\{\gamma, \theta\}$. Write $w=s_{i_1}s_{i_2}$. Then $\Phi(w)=\{\alpha_{i_1}, s_{i_1}(\alpha_{i_2})\}$. Thus $\theta = s_{i_1}(\alpha_{i_2})$. Therefore, $\Delta$ is of type $A_2$ or $C_2$. \\

From (1)--(3) we get Table~\ref{table regular}. 
\end{proof}

\subsection{Cohomology vanishing}
\label{ss_van}
Let $\underline{\mathfrak b}$, $\underline{\mathfrak g}$ and
$\underline {\mathfrak g/\mathfrak b}$ denote the  vector bundles $G \times^B
\mathfrak b$, $G \times^B \mathfrak g$ and $G \times^B (\mathfrak g/\mathfrak b)$ associated with 
$\mathfrak b$, $\mathfrak g$, and $\mathfrak g/\mathfrak b$, respectively.  Note that
$\underline{\mathfrak g/\mathfrak b}$ is the tangent bundle $T\mathcal{B}$ of $\mathcal{B}$. In the following computation of the cohomology groups of various sheaves $\mathcal E$ on $\mathcal B$, we will delete $\mathcal B$ in the notation $\coh^i(\mathcal B, \mathcal E)$. 

Before giving the proof of Theorem~\ref{prop: cohomology vanishing}, we explain how to use Proposition~\ref{prop: regular weight 2} in the computation of the cohomology groups of  $\underline{\mathfrak b}$ and $\underline{\mathfrak g/\mathfrak b}$. 
The vector bundle $\underline{\mathfrak b}$ has a filtration 
$0 \subset \mathcal D_0  \subset \mathcal D_1 \subset \mathcal D_2 \subset \dots \subset \mathcal D_m=\underline{\mathfrak b}$ where    $\mathcal D_0$ is the trivial vector bundle $\underline{\mathfrak t}$ and for each $1 \leq i\leq m= |\Phi^+|$, $\mathcal D_{i}/\mathcal D_{i-1}$ is   $\mathcal L(\alpha)$ for some positive root $\alpha$. 

For each $i$, the short exact sequence 
$$0 \rightarrow \mathcal D_{i-1} \rightarrow \mathcal D_i \rightarrow \mathcal D_i/\mathcal D_{i-1} \rightarrow 0$$
induces a long exact sequence  connecting the cohomology groups of $ \mathcal D_{i-1}, \mathcal D_{i}/\mathcal D_{i-1}$ and $\mathcal D_i$. In particular, if all the cohomology groups of  $ \mathcal D_{i-1}$ and $ \mathcal D_{i}/\mathcal D_{i-1}$ vanish, then all the cohomology groups of  $ \mathcal D_{i }$ vanish. 
To compute the cohomology group of $\mathcal D_i/\mathcal D_{i-1}$, we use Theorem~\ref{thm BWB} and  Proposition~\ref{prop: regular weight 2}. Then we get the cohomology groups of $\mathcal D_i$ recursively.  
The same method applies to the computation of the cohomology groups of $\underline{\mathfrak g/\mathfrak b}$.

When the cohomology group $\coh^k(\mathcal D_i/\mathcal D_{i-1})$ does not vanish for some $i$, to make the computation easy, we define a coarser filtration first,  for example, the filtration $(T^{(\bullet)}\otimes \mathcal I_X)$ of $\underline{\mathfrak g/\mathfrak b} \otimes \mathcal I_X$ in the proof of Theorem~\ref{prop: cohomology vanishing}, and then apply the method in the above to their quotients.

\begin{lemma} $\,$  \label{l.23} 

\begin{enumerate} 
\item When $\Delta$ is neither of type $A_2$ nor of type $C_2$, 
$\coh^i(\tang \mathcal B \otimes \mathcal I_X)=0$ for all $i$. 
\item  When $\Delta$ is of type $A_2$ or of type $C_2$, $\coh^i(\tang \mathcal B\otimes \mathcal I_X)=0$ for all $i \not=1$. 
\end{enumerate} 
\end{lemma} 

\begin{proof} 
Recall that $\tang \mathcal B =\underline{\mathfrak g/\mathfrak b}$.
From the short exact sequence $$ 0 \rightarrow \mathfrak b \rightarrow \mathfrak g
\rightarrow \mathfrak g/ \mathfrak b \rightarrow 0,$$ we get $$0 \rightarrow
\underline{\mathfrak b} \otimes \mathcal I_X \rightarrow \underline{\mathfrak g}
\otimes \mathcal I_X \rightarrow \underline{\mathfrak g/\mathfrak b} \otimes \mathcal
I_X \rightarrow 0.$$ 

Since $\mathfrak g$ is a representation space of $G$, the vector bundle
$\underline{\mathfrak g}$ is a trivial vector bundle and thus $\underline{\mathfrak g}
\otimes \mathcal I_X$ is $\mathcal L(-\theta)^{\oplus N}$, where $N$ is the
dimension of $\mathfrak g$. Since $-\theta + \rho$ is singular (Proposition~\ref{prop: regular
  weight 1}),  
 we get $\coh^i(\underline{\mathfrak g} \otimes \mathcal I_X)=0$ for all $i$ (Theorem~\ref{thm BWB}).
Therefore, we have  
\begin{equation} \label{e negpos}
\coh^i(\underline{\mathfrak g/\mathfrak b} \otimes \mathcal I_X ) =
\coh^{i+1}(\underline{\mathfrak b} \otimes \mathcal I_X) \quad \text{ for all  } i. 
\end{equation}

By Theorem~\ref{thm BWB} and Proposition~\ref{prop: regular weight 2}, the following holds: 
\begin{itemize}
\item If $\Delta$ is neither of type $A_2$ nor of type $C_2$,   $ \coh^{i+1}(\underline{\mathfrak b} \otimes \mathcal I_X) $ vanishes  for all  $i$. Thus, by~\eqref{e negpos}, $\coh^i(\underline{\mathfrak g/\mathfrak b} \otimes \mathcal I_X ) =0$ for all $i$.

 \item

 If $\Delta$ is of type $A_2$,  then  $\coh^{i+1}(\underline{\mathfrak b} \otimes \mathcal I_X)$ vanishes for all $i  \not\in\{1,2\}$ and $\coh^i(\underline{\mathfrak g/\mathfrak b} \otimes \mathcal I_X )$ vanishes for all $i \not \in \{0,1\}$.
 Thus, by \eqref{e negpos},  $\coh^i(\underline{\mathfrak g/\mathfrak b} \otimes \mathcal I_X )$ vanishes for all $i \not=1$.

 \item

  If $\Delta$ is   of type $C_2$,  then  $\coh^{i+1}(\underline{\mathfrak b} \otimes \mathcal I_X)$ vanishes for all $i  \not\in\{1,2\}$ and $\coh^i(\underline{\mathfrak g/\mathfrak b} \otimes \mathcal I_X )$ vanishes for all $i  \not\in\{0,1 \}$. Thus, by \eqref{e negpos},  $\coh^i(\underline{\mathfrak g/\mathfrak b} \otimes \mathcal I_X )$ vanishes for all $i  \not=1 $.
 \end{itemize}
This completes the proof. 
\end{proof}

\begin{proof} [Proof of Theorem~\ref{prop: cohomology vanishing}] 
For $\alpha =\sum_{\alpha_i \in \Delta} n_i \alpha_i \in \Phi^+$, define $\height(\alpha)$ by the sum
$\sum_{\alpha_i \in \Delta} n_i$. Let  $m:= \height(\theta)$.  Consider the filtration
\[ 
\mathfrak g^{(0)}=\mathfrak b \;\subset\; \mathfrak g^{(1)} \;\subset\; \cdots \;\subset\; \mathfrak g^{(d)} \;\subset\; \mathfrak g^{(d+1)} \;\subset\; \cdots \;\subset\; \mathfrak g^{(m)} =\mathfrak g
\]
by $B$-invariant subspaces, where
\begin{eqnarray*}
\mathfrak g^{(d)}:= 
\mathfrak b  \oplus \left( \bigoplus_{\height(\alpha)\leq d}\mathfrak g_{\alpha}\right) &\text{ for } 0\leq d \leq m. 
\end{eqnarray*} 
This induces a filtration by homogeneous subbundles 
\begin{eqnarray*}
 0=T^{(0)} \;\subset\; T^{(1)} \;\subset\; \cdots \;\subset\; T^{(d)} \;\subset\; \cdots \;\subset\; T^{(m)} = \underline{\mathfrak g/\mathfrak b}.
\end{eqnarray*}
After tensoring with $\mathcal{I}_X$, we get
\begin{eqnarray*}
  T^{(1)} \otimes \mathcal{I}_X \;\subset\; \cdots \;\subset\; T^{(d)}\otimes \mathcal{I}_X \;\subset\; \cdots \;\subset\; T^{(m)}\otimes \mathcal{I}_X =\underline{\mathfrak g/\mathfrak b}\otimes \mathcal{I}_X.
\end{eqnarray*}

By Lemma~\ref{l.23}, it suffices to prove that  
\begin{enumerate}
\item[(i)] $\coh^1(\underline{\mathfrak g/\mathfrak b}\otimes\mathcal{I}_X) \cong\bk$  if $\Delta$ is of type $A_2$,

 \item[(ii)]  $ \coh^1(\underline{\mathfrak g/\mathfrak b}\otimes\mathcal{I}_X)=0$   if $\Delta$ is of type $C_2$.
 \end{enumerate}

(i) Assume that $\Delta$ is of type $A_2$. Then $\theta=\alpha_1 + \alpha_2$ and $m=2$, so we get $T^{(2)}  = \underline{\mathfrak g/\mathfrak b}$. Furthermore, $\{\theta-\alpha_1, \theta -\alpha_2\}=\{\alpha_2,\alpha_1\}$. 
By   Theorem~\ref{thm BWB} and Proposition~\ref{prop: regular weight 2},  
$$
\coh^0 \left(\left(T^{(2)}/T^{(1)}\right)\otimes \mathcal I_X \right) = \bk ,  \quad \coh^1(T^{(1)}\otimes \mathcal I_X)=\bk \oplus \bk
$$  and  $\coh^i(\left(T^{(2)}/T^{(1)}\right)\otimes \mathcal I_X) =0 $ for any $i \not=0$
 and $\coh^j(T^{(1)}\otimes \mathcal I_X)=0$ for any $j \not=1$. 
Applying this to the long exact sequence associated with the  short exact sequence  
$$0 \rightarrow T^{(1)}\otimes \mathcal I_X \rightarrow T^{(2)}\otimes \mathcal I_X \rightarrow \left(T^{(2)}/T^{(1)}\right)\otimes \mathcal I_X \rightarrow 0,$$
we get 
\[
\begin{tikzcd}[column sep = 1em, row sep = 0.2cm] 
\coh^0(T^{(2)}\otimes \mathcal I_X) \rar \rar \arrow[d, equal]
    & \coh^0\left(\left(T^{(2)}/T^{(1)}\right)\otimes \mathcal I_X \right) \rar \arrow[d, equal]
    & \coh^1(T^{(1)}\otimes \mathcal I_X) \rar \arrow[d, equal]
    & \coh^1(T^{(2)} \otimes \mathcal I_X) \rar
    &0. \\
0 & \bk & \bk \oplus \bk 
\end{tikzcd}
\]
Here, by Lemma~\ref{l.23}, $\coh^0(T^{(2)}\otimes \mathcal I_X) =0$, and thus we get  $ \coh^1(T^{(2)} \otimes \mathcal I_X) =\bk$. 

(ii) Assume that $\Delta$ is of type $C_2$. Then $\theta = 2\alpha_1 +\alpha_2$ and $m=3$ and $T^{(3)}= \underline{\mathfrak g/\mathfrak b}$. Furthermore, $\theta-\alpha_{i_0}=\alpha_1 + \alpha_2$, which is the only root of height $2$. 
By   Theorem~\ref{thm BWB} and Proposition~\ref{prop: regular weight 2},  
$$
\coh^0 \left(\left(T^{(3)}/T^{(2)}\right)\otimes \mathcal I_X \right) = \bk ,  \quad 
\coh^1\left(\left(T^{(2)}/T^{(1)}\right)\otimes \mathcal I_X \right)=\bk$$  
and  $\coh^i\left(\left(T^{(3)}/T^{(2)}\right)\otimes \mathcal I_X\right) =0 $ for any $i \not=0$
 and $\coh^j(\left(T^{(2)}/T^{(1)}\right)\otimes \mathcal I_X)=0$ for any $j \not=1$ and $\coh^k(T^{(1)} \otimes \mathcal I_X)=0$ for all $k$.  
Applying this to the long exact sequences associated with the  short exact sequences 
$$0 \rightarrow T^{(1)}\otimes \mathcal I_X \rightarrow T^{(2)}\otimes \mathcal I_X \rightarrow \left(T^{(2)}/T^{(1)}\right)\otimes \mathcal I_X \rightarrow 0 $$
and 
$$0 \rightarrow T^{(2)}\otimes \mathcal I_X \rightarrow T^{(3)}\otimes \mathcal I_X \rightarrow \left(T^{(3)}/T^{(2)}\right)\otimes \mathcal I_X \rightarrow 0,$$ 
we get $\coh^1(T^{(2)} \otimes \mathcal I_X)=\bk$ and  $\coh^k(T^{(2)} \otimes \mathcal I_X)=0$ for all $k \not=1$ and  
\[
\begin{tikzcd}[column sep = 1em, row sep = 0.2cm] 
\coh^0(T^{(3)}\otimes \mathcal I_X) \rar  \arrow[d, equal]
    & \coh^0 \left(\left(T^{(3)}/T^{(2)}\right)\otimes \mathcal I_X \right) \rar \arrow[d, equal]
    & \coh^1(T^{(2)}\otimes \mathcal I_X) \rar  \arrow[d, equal]
    & \coh^1(T^{(3)} \otimes \mathcal I_X) \rar 
    & 0. \\
0 & \bk & \bk
\end{tikzcd}
\]
Here, by Lemma~\ref{l.23}, $\coh^0(T^{(3)}\otimes \mathcal I_X) =0$, and thus we get  $ \coh^1(T^{(3)} \otimes \mathcal I_X) =0$. 
\end{proof} 
 
\subsection{Parabolic setting}
\label{s.partial}
In this subsection  we prove Theorem~\ref{t_comp p}, an analogue of Theorem~\ref{t_comp}, repeating the processes in Subsections~\ref{s.adj}--\ref{ss_van}. As in the proof  of Theorem~\ref{t_comp}, the main vanishing theorem is Theorem~\ref{prop: cohomology vanishing p} below. 

We use the same notations as in Subsection \ref{s.notations}, otherwise stated. 
Let $P$ be a   parabolic subgroup of $G$ containing $B$. Then it corresponds to a subset $\Sigma_P$ of $\Delta$. The correspondence satisfies that the Lie algebra $\mathfrak p$ of $P$ is generated by $\mathfrak b$ and the root spaces  $\mathfrak g_{\alpha}$ where $\alpha$ is a linear combination of simple roots in $\Sigma_P$.  Let $\mathcal P$ be the variety of conjugacy classes of $P$. 

A subspace $H \subset \mathfrak g$ is said to be a \emph{Hessenberg space for }$P$ if $H$ is $\mathrm{Ad}(P)$-invariant and contains $\mathfrak p$. Then it corresponds a \emph{Hessenberg subset}, a subset $M \subset \Phi^+$ satisfying that for any $\beta \in M$ and $\alpha \in (-\Delta) \cup \Sigma_P$ with $\beta+ \alpha \in \Phi^+$, we have $\beta+\alpha \in M$. The correspondence satisfies $H=\mathfrak b \oplus \bigoplus_{\alpha \in M}\mathfrak g_{\alpha}$. Similarly to $\mathcal B(M,s)$, we write 
\[\mathcal P(M,s):=\pi_2^{-1}(s)\]
for the scheme-theoretic fiber of $\pi_2:G\times^PH\to \fg$ over $s\in \fg$.

We say that an integral weight $\lambda$ is dominant with respect to $P$ if $\langle \lambda, \alpha^{\vee}\rangle \geq 0$ for any simple root $\alpha$ in $\Sigma_P$. For such $\lambda$, we write $E_{\lambda}$ for the corresponding representation of $P$ and 
$\mathcal{L}_P(\lambda)$ for the homogeneous  vector  bundle $G \times ^PE_{\lambda}$.

\begin{theorem}[{Borel--Weil--Bott theorem for $\mathcal P$ \cite[Theorem in Section 4.3]{Ak}}]  \label{thm BWB P} 
Let $P$ be a standard parabolic subgroup of $G$ and $\lambda$ be an integral weight of $G$, dominant with respect to $P$. If $\lambda + \rho$ is   regular, then there is a unique $w \in W$ such that $w(\lambda + \rho)$ is regular dominant and  
\begin{eqnarray*}
H^{i}(\mathcal P, \mathcal L_P(\lambda)) =
\begin{cases}
V_{w(\lambda+\rho)-\rho} & \text{ if }i=\ell(w); \\
0 & \text{ otherwise.}
\end{cases}
\end{eqnarray*}
\end{theorem}

From now on, let   $M$ be as in Theorem~\ref{t_comp}. 
Let  $P$ be the standard parabolic subgroup of $G$ with  $\Sigma_P=\Delta_0$, which is defined in Subsection~\ref{s.notations}, and let 
$Y= {\mathcal{P}}(M,s)$. We assume that $G$ is not of type $A_1$.  
Then $Y$ is a divisor in $\mathcal P$ with ideal sheaf $\mathcal{I}_Y=\mathcal{L}_P(-\theta)$
and with normal bundle $\mathscr{N}_{P}=\mathcal{L}_P(\theta)|_{Y}$. 
Write $\tang \mathcal P$ and $\tang Y$ for the tangent bundles of $\mathcal P$ and $Y$
respectively.
We get three related exact sequences:
\begin{align}
   &0\to \mathcal{O}_\mathcal P\to \mathcal{L}_P(\theta)\to\mathscr{N}_P\to 0;\label{e.nrp}\\
   &0\to \tang Y \to \tang \mathcal P|_{Y} \to \mathscr{N}_P\to 0;\label{e.tnp}\\
   &0\to \tang \mathcal P\otimes\mathcal{I}_Y\to \tang \mathcal P\to\tang \mathcal P|_{Y}\to 0
   \label{e.trp}
\end{align}
We view each of the sheaves above as sheaves on $\mathcal P$ and, for each such
sheaf $\mathcal{F}$, we write
$h^i(\mathcal{F}) = \dim\coh^i(\mathcal P,\mathcal{F})$.

Applying  the same argument as in the proof of Theorem~\ref{prop: cohomology vanishing} to the short exact sequence
\beq\label{e.ses.gp}0 \rightarrow \mathfrak p \rightarrow \mathfrak g \rightarrow \mathfrak g/\mathfrak p \rightarrow 0,\eeq
 we get the following result.   

\begin{theorem}\label{prop: cohomology vanishing p}
  Suppose $G$ is a simple group over $\bk$.  
  Then,
  \[
    \coh^i(\mathcal P, \tang \mathcal P\otimes\mathcal{I}_Y) \cong
    \begin{cases} 
      \bk & \text{if  $G$ is of type $A_r$ with $r \geq 2$ and $i=1$};\\
      0 & \text{otherwise.}
    \end{cases}
  \]
\end{theorem}

\begin{proof}  
As in the previous subsection, denote by $\underline{\mathfrak g/\mathfrak p}$ and $ \underline{\mathfrak p}$  the associated vector bundles $G \times^P \mathfrak g/\mathfrak p$ and $G \times^P \mathfrak p$. Then $\underline{\mathfrak g/\mathfrak p}$ is the tangent bundle $T\mathcal P$.

From the short exact sequence~\eqref{e.ses.gp},
 we get
\begin{equation} \label{e.np}
\coh^i(\underline{\mathfrak g/\mathfrak p} \otimes \mathcal I_Y) = \coh^{i+1}(\underline{\mathfrak p} \otimes \mathcal I_Y) \text{ for all } i. 
\end{equation}

For $\alpha =\sum_{\alpha_i \in \Delta} n_i \alpha_i \in \Phi^+$, define $\height_P(\alpha)$ by the sum 
$\sum_{\alpha_i \in \Delta \backslash \Delta_0} n_i$. Then $\height_P(\theta)=2$ and we have  a filtration
$$\mathfrak g^{(-2)} \;\subset\;  \mathfrak g^{(-1)} \;\subset\; \mathfrak g^{(0)}= \mathfrak p \;\subset\; \mathfrak g^{(1)} \;\subset\; \mathfrak g^{(2)} =\mathfrak g,$$
where
\begin{eqnarray*}
\mathfrak g^{(d)}=\left\{
\begin{array}{ll}
\mathfrak b  \oplus \left( \bigoplus_{\height_P (\alpha)\leq d}\mathfrak g_{\alpha}\right) &\text{ for } 0\leq d \leq 2;  \\
\bigoplus_{\height_P(\alpha)\leq d}\mathfrak g_{\alpha}   &\text{ for } -2 \leq d \leq -1,
\end{array} \right.
\end{eqnarray*}
and the induced filtrations
\begin{eqnarray*}
 0  \;\subset\; T^{(1)} \;\subset\;   T^{(2)} =\underline{\mathfrak g/\mathfrak p}; \\
 0 \;\subset\;   V^{(-2)} \;\subset\; V^{(-1)} \subset V^{(0)} =\underline{\mathfrak p}.
\end{eqnarray*}
After tensoring with $\mathcal{I}_Y$ we get
\begin{eqnarray*}
  T^{(1)} \otimes \mathcal{I}_Y \;\subset\;   T^{(2)}\otimes \mathcal{I}_Y =\underline{\mathfrak g/\mathfrak p} \otimes \mathcal{I}_Y; \\
  V^{(-2)} \otimes \mathcal I_Y \;\subset\; V^{(-1)} \otimes \mathcal I_Y \;\subset\; V^{(0)}  \otimes \mathcal I_Y=\underline{\mathfrak p} \otimes \mathcal I_Y.
\end{eqnarray*}
 
As in the proof of Theorem~\ref{prop: cohomology vanishing}, we need to know for which root $\alpha$, the weight $\theta-\alpha +\rho$ is regular. 
For those $\alpha $ in Proposition~\ref{prop: regular weight 2}, $\ell(w)$ and  $\height_P(\alpha )$ are given as follows.
\begin{table}
\begin{center}
\begin{tabular}{c|c||c|c|c}
\toprule 
& $\Delta$ & $\alpha$ &  $\ell(w)$ &  $\height_P(\alpha)$\\
\midrule 
(1) &  any type &  $\theta$ & $0$&  $2$ \\
   & $A_r$ ($r \geq 2$) &  $\theta-\alpha_1, \theta-\alpha_r$ & $1$& $1$ \\
  & $\not=A_r$ &   $\theta- \alpha_{i_0}$ &  $1$&  $1$ \\
  \midrule
  (2) 
   & $\ast$  &  $\ast$ &  $\geq 2$& $0$ \\
   \midrule
   (3)& $A_2$ & $-\alpha_1, -\alpha_2$ & $2$  & $-1$  \\
   &  $A_2$ & $-\theta$  &  $3$ & $-2$  \\
     & $C_2$ & $-\alpha_1$   & $2$ &  $-1$  \\
     & $C_2$  & $-\theta$   &  $3$&   $-2$ \\
\bottomrule 
\end{tabular}
\end{center}
\caption{$\mathrm{ht}_P(\alpha)$ for $\alpha$ in Table~\ref{table regular}}
\label{table height}
\end{table}

By  Theorem~\ref{thm BWB P} and \eqref{e.np}, we get \begin{equation}\label{eq 0}
\coh^i(T^{(2)} \otimes \mathcal I_Y)=0  \text{ for any } i\not=1 .
\end{equation}
Indeed, by Theorem~\ref{thm BWB P}, those $\alpha$ in the item (1) (resp. the items (2) and (3)) may contribute nontrivially to the cohomology groups of $\underline{\mathfrak g/\mathfrak p}\otimes \mathcal I_Y$ (resp. $\underline{\mathfrak p}\otimes \mathcal I_Y$). Since $\ell(w) $ is greater than or equal to $2$ for every $\alpha$ in the items (2) and (3), $H^{i+1}(\underline{\mathfrak p} \otimes \mathcal I_Y)$ vanishes for any $i \leq 0$. By~\eqref{e.np}, $H^i(\underline{\mathfrak g/\mathfrak p}\otimes \mathcal I_Y)$ vanishes for $i \leq 0$. Since $\ell(w)$ is less than or equal to $1$ for every $\alpha$ in the item (1), $H^{i}(\underline{\mathfrak{g}/\mathfrak p} \otimes \mathcal I_Y)$ vanishes for any $i \geq 2$ by Theorem~\ref{thm BWB P}. Therefore, $\coh^i(T^{(2)} \otimes \mathcal I_Y) $ vanishes for any $i\not=1$.  

By Theorem~\ref{thm BWB P}, we also get  
  \begin{equation} \label{eq 1}
    \coh^i(T^{(2)}/T^{(1)} \otimes \mathcal I_Y) = 
    \begin{cases}
      \bk& \text{  if  }i=0 ;\\
      0&       \text{  if  }i=1, 
    \end{cases}
  \end{equation}
   and 
   \begin{equation} \label{eq 2}
    \coh^1(T^{(1)} \otimes \mathcal I_Y) = 
    \begin{cases}
      \bk \oplus \bk& \text{ if } G \text{ is of type } A_r \text{ with } r \geq 2;\\
      \bk&       \text{ otherwise}. 
    \end{cases}
  \end{equation}
 For \eqref{eq 1}, use that for $\alpha$ in the item (1) with $\height_P(\alpha)=2$, $\ell(w)$ is 0. For \eqref{eq 2}, note that $\mathfrak g^{(1)}/\mathfrak g^{(0)}$ is the direct sum of two irreducible representation of $P$ with highest weights $\theta-\alpha_1$ and $\theta-\alpha_r$ when $G$ is of type $A_r$, and is an irreducible representation of $P$ with highest weight $\theta-\alpha_{i_0}$ otherwise. From Table~\ref{table height}, we see that for those roots, we have $\ell(w)=1$. Then \eqref{eq 2} follows from Theorem \ref{thm BWB P}. 
 
 To compute $\coh^1(T^{(2)} \otimes \mathcal I_Y) $, consider the long exact sequence associated with 
 $$0 \rightarrow T^{(1)} \otimes \mathcal I_Y \rightarrow T^{(2)} \otimes \mathcal I_Y \rightarrow \left( T^{(2)}/T^{(1)}\right) \otimes \mathcal I_Y \rightarrow 0.$$
 Since $\coh^0(T^{(2)} \otimes \mathcal I_Y)\stackrel{\eqref{eq 0}}{=}0$ and $\coh^1\left(\left(T^{(2)}/T^{(1)}\right) \otimes \mathcal I_Y\right)\stackrel{\eqref{eq 1}}{=}0$, we get
 $$0 \rightarrow \coh^0\left(\left(T^{(2)}/T^{(1)}\right) \otimes \mathcal I_Y\right)   \rightarrow \coh^1(T^{(1)} \otimes \mathcal I_Y) \rightarrow \coh^1(T^{(2)} \otimes \mathcal I_Y) \rightarrow 0.$$ 
 From $\coh^0\left(\left(T^{(2)}/T^{(1)}\right) \otimes \mathcal I_Y\right) \stackrel{\eqref{eq 1}}{=}\bk$ and 
 \eqref{eq 2}
  it follows that 
   \begin{equation} 
    \coh^1(T^{(2)} \otimes \mathcal I_Y) = 
    \begin{cases}
      \bk &  \text{ if } G \text{ is of type } A_r \text{ with } r \geq 2;\\
      0&       \text{ otherwise}.
    \end{cases}
  \end{equation}
 This completes the proof of Theorem~\ref{prop: cohomology vanishing p}. 
 \end{proof}

  By~\cite{demazure}, we know that
  \begin{equation}\label{e.Dcalcp}
    h^i(\mathcal{L}_P(\theta)) = 
    \begin{cases}
      \dim G& \text{ if } i =0;\\
      0&       \text{ otherwise}.
    \end{cases}
  \end{equation}
  and 
  \begin{equation}\label{e.Dcalcpp}
      h^i(\tang \mathcal P) =
    \begin{cases}
      \dim G & \text{ if } G \text{ is not of type } C_r \text { and } i =0;\\
       \dim G' & \text{ if } G \text{ is of type } C_r \text{ and } i =0;\\
      0&       \text{ otherwise}.
    \end{cases}
  \end{equation} 
  where $G'$ is of type $A_{2r-1}$. 
  
\begin{proposition} \label{p P} The following statements hold.
\begin{enumerate}  
\item $h^0(\mathscr N_P) = \dim G -1$ and the map $\coh^0(Y, \mathscr N_P) \rightarrow \coh^1(Y, TY)$ is surjective. 
\item Moreover, 
\begin{equation*} 
h^0(Y,TY) - h^1(Y,TY) = 
\begin{cases} 
2  & \text{ if } G \text{ is of type } A_r \text{ with } r \geq 2;\\
\dim G' - \dim G +1 & \text{ if } G \text{ is of type } C_r ;\\
1 & \text{ otherwise.}
\end{cases} 
\end{equation*}
\end{enumerate} 
\end{proposition}
\begin{proof}
 (1)   
  From the long exact sequence for~\eqref{e.nrp}, we   get that
  \begin{equation} \label{e.nrpp} 
  0 \rightarrow \coh^0(\mathcal O_{\mathcal P})  \rightarrow \coh^0(\mathcal  L_P(\theta)) \rightarrow \coh^0(\mathscr N_P) \rightarrow \coh^1(\mathcal O_{\mathcal P})=0 \end{equation}
  and 
  $
  0= \coh^1(\mathcal L_P(\theta)) = \coh^1(\mathscr N_P) $.  Thus $h^0(\mathscr N_P) = \dim G -1$. 
    From the long exact sequence for~\eqref{e.tnp}, we then get that  
  \begin{equation} \label{e.tnpp}
  0 \rightarrow \coh^0(TY) \rightarrow \coh^0(T{\mathcal P}|_{Y}) \rightarrow \coh^0(\mathscr N_P) \rightarrow \coh^1(TY) \rightarrow \coh^1(T{\mathcal P}|_{Y}) \rightarrow 0
  \end{equation} 
  
   From the long exact sequence for~\eqref{e.trp}, we   get that 
   \begin{equation} \label{e.trpp} 
   0 \rightarrow \coh^0(T\mathcal P \otimes \mathcal I_Y) \rightarrow \coh^0(T\mathcal P) \rightarrow \coh^0(T\mathcal P|_{Y}) \rightarrow \coh^1(T\mathcal P \otimes \mathcal I_Y) \rightarrow \coh^1(T\mathcal P)=0
   \end{equation}
   and $\coh^1(T\mathcal P|_{Y}) \cong \coh^2(T\mathcal P \otimes \mathcal I_Y)$.

By Theorem~\ref{prop: cohomology vanishing p},~\eqref{e.trpp} becomes 
\begin{equation} \label{e.trppp}
 0   \rightarrow \coh^0(T\mathcal P) \rightarrow \coh^0(T\mathcal P|_{Y}) \rightarrow \coh^1(T\mathcal P \otimes \mathcal I_Y) \rightarrow 0
   \end{equation}
   and $\coh^1(T\mathcal P|_{Y}) =0$.  By \eqref{e.tnpp} the map $\coh^0(Y, \mathscr N_P) \rightarrow \coh^1(Y, TY)$ is surjective.
 
 (2)  Putting $\coh^1(T\mathcal P|_Y)   =0$ into  ~\eqref{e.tnpp} we get 
  \begin{equation} \label{e.tnppp}
  0 \rightarrow \coh^0(TY) \rightarrow \coh^0(T{\mathcal P}|_{Y}) \rightarrow \coh^0(\mathscr N_P) \rightarrow \coh^1(TY) \rightarrow 0
  \end{equation} 
  
 Combining  ~\eqref{e.nrpp},~\eqref{e.trppp}, and~\eqref{e.tnppp}, we get $h^0(TY) - h^1(TY) = \dim G' - \dim G + h^1(T\mathcal P \otimes \mathcal I_Y) +1$. 
\end{proof}

\label{proof_t_comp p}
\begin{proof}[Proof of Theorem~\ref{t_comp p}] 
The surjectivity of the Kodaira--Spencer map \eqref{eq:KSmap.p.intro}
follows from Proposition~\ref{p P}(1) by the same argument as in the proof of Theorem~\ref{t_comp}.

If $G$ is neither of   type   $A_r$ with $r \geq 2$ nor of type $C_r$, then 
by Theorem~\ref{prop: cohomology vanishing p} and \eqref{e.trp}, every vector field in $\coh^0(Y,TY)$ is the restriction of a vector field in $\coh^0(  T\mathcal P)\cong \mathfrak g$. By  the same arguments as in the proof of Theorem~\ref{t_comp}, we get $h^0(Y,TY)=r$, the rank of  $G$. 
Then the desired result for $h^i(Y,TY)$ follows by Proposition~\ref{p P}(2). 
On the other hand, if $G$ is of type $A_r$ or of type $C_r$, then we cannot apply the same arguments. In the rest of the proof, first, we explain why we cannot,  and then we prove the statement in a different way. 

  If $G$ is of type $A_r$ with $r \geq 2$, then  by Theorem \ref{prop: cohomology vanishing p}, $\coh^1(\mathcal P, T\mathcal P\otimes \mathcal I_Y)$ does not vanish,  and from \eqref{e.trp} we get $ \coh^0(T\mathcal P|_Y) =\coh^0(T \mathcal P) \oplus  \bk$. For a vector field  in $\coh^0(Y,TY)$ which can be extended to a vector field in $\coh^0(  T\mathcal P)$, applying the same arguments in the above, we see that it is induced from elements in $\mathfrak t$, the Lie algebra of the centralizer of $s$ in $G$.   However, as we remark in Remark \ref{rem A}, a vector field in $\coh^0(Y,TY)$ does not have to be extended to a vector field in $\coh^0(T\mathcal P)$. 
  Still, 
by Theorem~\ref{t.autsA}(2), we have $h^0(Y,TY)=r$. 
  Then the equality  $h^1(Y,TY)=r-2$ follows from  Proposition~\ref{p P}(2).

   If $G$ is of type $C_r$, then we cannot apply the same arguments because $\coh^0( T\mathcal P)$ is not isomorphic to $\mathfrak g$ (see \eqref{e.Dcalcpp}). In this case, $\mathcal P$ is  $\mathbb P^{2r-1}$ and $Y$ is a smooth quadric $\mathbb Q^{2r-2}$ in $\mathbb P^{2r-1}$.  Therefore we get $h^0(Y, TY)=\frac{2r(2r-1)}{2}=r(2r-1)$ and $h^i(Y, TY)=0$ for any $i \geq 1$.  
\end{proof}

\section{Additional vanishing}  
\label{alt_coho}
In this section, $G$ is a simple group. 
We maintain the notation of Section~\ref{s.notations}.
In particular, $\Delta$ is a set of simple roots in the root system $\Phi$ 
and $\Phi^{+}$ is the set of positive roots.  All of these are with respect to 
a maximal torus $T$ in $G$.

Let $\height(\beta)$ denote the height of $\beta \in \Phi$
with respect to $\Delta$.  Let $\height(\beta^\vee)$ denote the height of $\beta^\vee$ in the 
coroot system $\Phi^\vee$ with respect to the coroots of $\Delta$.

We generally follow the notation of \cite[\S 1.7]{kummar-lauritzen-thomsen}. 
Recall that our Borel subgroup $B$, containing~$T$, corresponds to $-\Delta$. 
If $P$ is a parabolic subgroup of $G$ and $\mathsf{M}$ is a $P$-module, then 
$\coh^i(G/P,\mathsf{M})$ denotes the sheaf cohomology $\coh^i(G/P, \mathcal L(\mathsf{M}))$ 
of the homogeneous vector bundle $\mathcal L(\mathsf{M})$ on $G/P$.
If ${\bk}_{\lambda}$ is the one-dimensional representation of $P$ corresponding a character $\lambda \in X^*(P)$, 
we write $\coh^i(G/P, \mathsf{M} \otimes \lambda)$ instead of $\coh^i(G/P, \mathsf{M} \otimes {\bk}_{\lambda})$.

Let $\mathfrak{n}_P$ denote the nilradical of the Lie algebra of $P$.
Set $\mathfrak{n} = \mathfrak{n}_B$ and $\mathfrak{n}_\alpha = \mathfrak{n}_{P_\alpha}$ 
where $P_\alpha$ is the parabolic subgroup containing $B$ with Levi factor corresponding to $\alpha \in \Delta$.   
Let $S^n(\mathsf{M})$ denote the $n$-th symmetric power of $\mathsf{M}$ for $n \geq 0$ and equal to zero if $n<0$.

The following lemma is used throughout this section.  
\begin{lemma} \label{lemma: past results cohom}
Let $\alpha \in \Delta$.
Let $\mathsf{M}$ be a $P_{\alpha}$-module and let $\lambda \in X^*(B)$.  
Let $k = \langle \lambda, \alpha^{\vee} \rangle$.  
\begin{enumerate} 
\item If $k=-1$, then $\coh^i(\mathcal B, \mathsf{M}\otimes \lambda)=0$ for all $i \geq 0$. 
\item If $k \leq -2$, then $\coh^i(\mathcal B, \mathsf{M}\otimes \lambda)=\coh^{i-1}(\mathcal B, \mathsf{M} \otimes (s_{\alpha}\lambda - \alpha))$ for all $i \geq 0$.
\item If $k = -1$, then $\coh^i(\mathcal B, S^n \mathfrak{n}^* \otimes \lambda)= \coh^{i}(\mathcal B,  S^{n-1} \mathfrak{n}^* \otimes (\lambda + \alpha))$ for all $i \geq 0$ and all $n \geq 0$.
\item If $k \leq -2$, then
$$\coh^i(\mathcal B, S^n \mathfrak{n}_\alpha^* \otimes \lambda)= \coh^{i-1}(\mathcal B,  S^{n} \mathfrak{n}_\alpha^* \otimes (\lambda + (-k-1)\alpha))$$ for all $i \geq 0$ and all $n \geq 0$.  
\item  For any parabolic subgroup $P$, we have
$\coh^i(\mathcal B, S^n \mathfrak{n}_P^*)= 0$ for all $i \geq 1$ and $n \geq 0.$
\end{enumerate} 
\end{lemma} 

\begin{proof}
The first two parts are due to Demazure \cite{demazure:simple} (see also Lemma 3.1 of \cite{broer2}).  
The third and fourth parts are contained in Lemma 3.3 of \cite{broer2}
(the condition in that lemma should read $s\lambda -\lambda = n\alpha$).
See also Lemma 6 in \cite{kummar-lauritzen-thomsen} for part (3), which
follows from part (1) using a Koszul resolution.  
Part (4) follows immediately from part (2) since $\mathfrak{n}_\alpha$ is
a $P_\alpha$-module.
The last part is due to Elkik (see \cite[Theorem~2.2]{broer2}).
\end{proof}

If $G$ is simply laced, then we adopt the convention that all roots of $G$ are both long and short.
\begin{theorem}\label{prop: alt cohom vanishing short}
Let $\beta \in \Phi^+$ be a short root.
Let $\alpha \in \Delta$ be any short simple root.  
  Then
  \[
    \coh^i(\mathcal{B}, S^n \mathfrak{n}^* \otimes (-\beta)) \cong
    \begin{cases}
      \coh^0 (\mathcal{B}, S^{n-\height(\beta)} \mathfrak{n}^*) & \text{if } $i=0$;\\
      \coh^0 (\mathcal{B}, \mathcal S^{n-\height(\beta)+1} \mathfrak{n}_\alpha^*) & \text{if } $i=1$;\\
      0 & \text{otherwise.}
    \end{cases}
  \]
\end{theorem}

\begin{proof}
If $\beta$ is not a
  simple root, then there is some $\gamma \in \Delta$
  such that $\langle \beta, \gamma^\vee \rangle = 1$ since $\beta$ is short.

  Then $\langle -\beta, \gamma^\vee \rangle = -1$,
  so $s_\gamma(-\beta) = -\beta + \gamma$ is a short negative root.
  Moreover,  
  \begin{equation}\label{equation: change_of_height_cohomology}   
  \coh^i(\mathcal{B}, S^j\mathfrak{n}^* \otimes (-\beta)) \cong \coh^i(\mathcal{B}, S^{j-1} \mathfrak{n}^* \otimes
  (-\beta+\gamma))
  \end{equation}
  for all $i \geq 0$, $j \geq 0$ by part (3) of Lemma~\ref{lemma: past results cohom} with $\lambda = -\beta$. 

    We can repeat the above argument with 
    $\beta' = \beta - \gamma$, if it is not simple, and keep repeating until
    $\beta'$ is simple, i.e., $\height(\beta')=1$.
    This will take exactly $\height(\beta)-1$ iterations since $\height(\beta-\gamma) = \height(\beta)-1$.
  We arrive at the identity
\[
\coh^i(\mathcal{B}, S^j\mathfrak{n}^* \otimes (-\beta) ) \cong
  \coh^i(\mathcal{B}, S^{j-{\height}(\beta)+1} \mathfrak{n}^* \otimes
  (-\alpha))
  \]
  for all $i \geq 0$, $j \geq 0$, 
  for some short simple root $\alpha$.

Next, tensor the Koszul sequence of $B$-modules
\begin{equation}\label{equation: Koszul}
0 \to S^{j-1}\mathfrak{n}^* \otimes \bk_{\alpha} \to S^j\mathfrak{n}^* \to
  S^j\mathfrak{n}_\alpha^*  \to 0
  \end{equation}
   by ${\bk}_{-\alpha}$
  to get the exact sequence of $B$-modules
\begin{equation}\label{equation: Koszul_2}
    0 \to S^{j-1}\mathfrak{n}^* \to S^j\mathfrak{n}^* \otimes {\bk}_{-\alpha} \to
    S^j\mathfrak{n}_\alpha^* \otimes {\bk}_{-\alpha} \to 0.
\end{equation}
  Regarding the first nonzero term, 
  $\coh^i(\mathcal{B}, S^{j-1}\mathfrak{n}^*)$ vanishes when $i \neq 0$ 
  by part (5) of Lemma~\ref{lemma: past results cohom}.
  Regarding the third nonzero term, part (4) of the lemma with $\lambda = -\alpha$ and $k= -2$
  implies that
  $$\coh^i(\mathcal{B}, S^j\mathfrak{n}_\alpha^* \otimes (-\alpha) ) \cong \coh^{i-1}(\mathcal{B},
  S^j\mathfrak{n}_\alpha^* ),$$ 
  and the right side vanishes except 
  when $i \neq 1$, again by part (5) of Lemma~\ref{lemma: past results cohom}.  
  
Hence, the long exact
  sequence in  cohomology over $\mathcal{B}$ with respect to \eqref{equation: Koszul_2} degenerates with 
  $$\coh^{0}(\mathcal{B}, S^j\mathfrak{n}^* \otimes (-\alpha)) \cong \coh^{0}(\mathcal{B}, S^{j-1}\mathfrak{n}^*)
  \text{ and } \coh^{1}(\mathcal{B}, S^j\mathfrak{n}^* \otimes (-\alpha)) \cong 
  \coh^{0}(\mathcal{B},  S^j\mathfrak{n}_\alpha^*)$$ 
  for all $j \geq 1$; $\coh^{i}(\mathcal{B}, S^j\mathfrak{n}^* \otimes (-\alpha)) =0$
  for $i \geq 2$ and all $j \geq 0$. 

Hence, the theorem holds for any short positive root $\beta$ and some short root $\alpha \in \Delta$.
Moreover, if $\alpha$ is any short simple root, then \eqref{equation: change_of_height_cohomology} can be reversed to climb to $\beta = \theta^+$, the maximal short root,
so the right-hand side is independent of the choice of short simple root.
\end{proof}

\begin{theorem}\label{prop: alt cohom vanishing long}
Suppose $\Phi$ has two root  lengths.  Let $\beta \in \Phi^+$ be a long root.
Let $\alpha, \delta \in \Delta$ be the unique simple roots 
that determine a Levi subalgebra of type $C_2$ or $G_2$, with $\alpha$ short.  
Denote by $\varpi$ the short dominant root for this subalgebra \textup{(}that is, 
$\varpi = (q-1)\alpha+\delta$ where $q=2$ in the $C_2$ case and $q=3$ in the $G_2$ case\textup{)}.
Then 
  \[
    \coh^i(\mathcal{B}, S^n \mathfrak{n}^* \otimes (-\beta) ) \cong
    \begin{cases}
      \coh^0 (\mathcal{B}, S^{n-\height(\beta) } \mathfrak{n}^*) & \text{if } $i=0$;\\
      \coh^0 (\mathcal{B}, S^{n-\height(\beta)+1} \mathfrak{n}_{\alpha}^*) \oplus 
       \coh^0 (\mathcal{B}, S^{n-\height(\beta^\vee)} \mathfrak{n}_{\alpha}^* \otimes \varpi) & \text{if } $i=1$;\\
      0 & \text{otherwise.}
    \end{cases}
  \]
\end{theorem}

\begin{proof}
Let $\Delta_l$ be the subset of  $\Delta$ consisting of long simple roots and let $\Phi_l$ be the roots supported on $\Delta_l$.  Then $\Phi_l$ is a type $A$, irreducible root system and hence 
$\height(\beta) = \height(\beta^{\vee})$ for $\beta \in \Phi_l$.
Moreover, the proof of Theorem~\ref{prop: alt cohom vanishing short} goes through since any such $\beta$ is conjugate to a simple root using only simple reflections $s_\gamma$ where $\gamma$ is long.
Hence, for these $\beta$ 
the statement of Theorem~\ref{prop: alt cohom vanishing short} holds with the short simple root $\alpha$ there 
replaced by any {\em long} simple root.
This agrees with the statement in this theorem 
since, by Lemma~\ref{lemma: functions on subregular cover}, 
we have for a long simple root $\beta$  that
$$\coh^0 (\mathcal{B}, S^{n-\height(\beta)+1} \mathfrak{n}_{\beta}^*) 
\cong \coh^0 (\mathcal{B}, S^{n-\height(\beta)+1} \mathfrak{n}_{\alpha}^*) \oplus 
       \coh^0 (\mathcal{B}, S^{n-\height(\beta)} \mathfrak{n}_{\alpha}^* \otimes \varpi).$$

Next, consider the remaining $\beta \in \Phi^+$.  These are the long roots $\beta$ 
that are supported on some short simple root.
As noted in Section~\ref{s.notations}, there is a unique short simple root $\alpha_k \in \Delta$ such that 
$\tilde\beta := \theta^{+} {+} \alpha_k$ is a root (and this root is long).  Also, 
$\langle \tilde\beta, \alpha_k^\vee \rangle = q$.

Using the tables in Subsection~\ref{s.notations}, we see that 
$\tilde\beta^\vee = \alpha^{\vee}_1+\alpha^{\vee}_2 + \dots + \alpha^{\vee}_r$ in both $B_r$ and $C_r$, 
equals $\alpha^{\vee}_1 + 2\alpha^{\vee}_2+2\alpha^{\vee}_3 + \alpha^{\vee}_4$ in $F_4$, and
equals $\alpha^{\vee}_1 + \alpha^{\vee}_2$ in $G_2$.
We observe that $\height(\tilde\beta^\vee) = \frac{h}{q}$, where $h$ is the Coxeter number of $G$.

By another case-by-case observation, 
$\beta$ is conjugate via simple reflections $s_\gamma$, where $\gamma$ is long, to $\tilde\beta$.
Hence, \eqref{equation: change_of_height_cohomology} (since only long simple roots $\gamma$ are used)
implies that
$$\coh^i(\mathcal{B}, S^j\mathfrak{n}^* \otimes (-\beta ))
 \cong
  \coh^i(\mathcal{B}, S^{j-{\height}(\beta) + \height(\tilde\beta)} \mathfrak{n}^* \otimes (-\tilde\beta))$$
for all $i \geq 0$, $j \geq 0$.  
Note that $\height(\beta) - \height(\tilde\beta) = \height(\beta^\vee) -  \height(\tilde\beta^\vee)$,
so it is enough to prove the result for $\beta = \tilde\beta$.

Tensoring \eqref{equation: Koszul} (with $\alpha$ replaced by $\alpha_k$)
by $\bk_{-\tilde\beta}$ gives 
\begin{equation}
0 \to S^{j-1}\mathfrak{n}^* \otimes (-\theta^+) \to S^j\mathfrak{n}^* \otimes (-\tilde\beta) \to
  S^j\mathfrak{n}_{\alpha_k}^* \otimes  (-\tilde\beta) \to 0.
\end{equation}
Since $\langle -\tilde\beta, \alpha^\vee \rangle = -q$ as noted above,
part (4) of Lemma~\ref{lemma: past results cohom} yields
$$\coh^i(\mathcal{B}, S^j\mathfrak{n}_{\alpha_k}^* \otimes (-\tilde\beta)) 
\cong \coh^{i-1}(\mathcal{B}, S^j\mathfrak{n}_{\alpha_k}^*\otimes (-\tilde\beta + (q-1)\alpha_k)).$$

If $q=2$, then 
$-\tilde\beta + (q-1)\alpha_k = -\theta^+$ and the latter cohomology equals
$\coh^{i-1}(\mathcal{B}, S^{j-\frac{h}{2}}\mathfrak{n}_{\alpha}^*\otimes \varpi)$
by Lemmas~\ref{lemma: up to short dominant} and \ref{lemma: neg to pos}, where $h$ is the Coxeter number of~$G$.
Then by Lemma~\ref{lemma: functions on subregular cover}, this cohomology vanishes except when $i=1$.

At the same time, the cohomology of $S^{j-1}\mathfrak{n}^* \otimes (-\theta^+)$ is known 
by Theorem~\ref{prop: alt cohom vanishing short} and it vanishes except when $i=0$ or $i=1$.
Hence the long exact sequence in cohomology gives for $i=0$ that
$$\coh^{0}(\mathcal{B}, S^j\mathfrak{n}^* \otimes (-\tilde\beta)) \cong 
\coh^{0}(\mathcal{B}, S^{j-1}\mathfrak{n}^* \otimes (-\theta^+)) \cong 
\coh^{0}(\mathcal{B}, S^{j- \height(\tilde\beta)}\mathfrak{n}^*)$$
since $\height(\tilde\beta) = \height(\theta^+
)+1$.
For $i=1$, we get
$$\coh^{1}(\mathcal{B}, S^j\mathfrak{n}^* \otimes (-\tilde\beta)) \cong 
\coh^{1}(\mathcal{B}, S^{j-1}\mathfrak{n}^* \otimes (-\theta^+)) \oplus 
\coh^{0}(\mathcal{B}, S^{j-\frac{h}{2}}\mathfrak{n}_{\alpha}^*\otimes \varpi).$$
This is the promised result for $i=1$ by Theorem~\ref{prop: alt cohom vanishing short} and 
since  $\height(\tilde\beta^\vee) = \frac{h}{2}$. 
For $i \geq 2$, all cohomology vanishes.  

If $q=3$, then $-\tilde\beta + (q-1)\alpha = -\alpha - \delta$ and the result follows using 
Lemmas ~\ref{lemma: neg to pos} and \ref{lemma: functions on subregular cover}.

If $\beta$ is not covered by the previous argument, then $q=2$ and $\beta$ belongs to the irreducible root system of
a proper Levi subalgebra, and we can appeal to induction.  The two long roots in the $C_2$ case were already covered by the previous results so the base case has been taken care of.
\end{proof}

We can recover another proof 
of Theorem~\ref{prop: cohomology vanishing}, which amounts to 
considering $\coh^i(\mathcal{B}, S^n \mathfrak{n}^* \otimes (-\theta))$ for $n=1$.
In the non-simply-laced cases, $\height(\theta) \geq 3$ and 
$\height(\theta^\vee) \geq 2$, so all three symmetric powers in Theorem~\ref{prop: alt cohom vanishing long}
involve negative powers for $n=1$ and hence vanish.  
In the simply-laced cases, where $\theta$ can also be considered short, $\height(\theta) \geq 3$, except for 
types $A_2$ and $A_1$.  Hence all cohomology groups vanish by Theorem~\ref{prop: alt cohom vanishing short}
except in these two cases.
When $G$ is of type
$A_2$, the $i=0$ case vanishes and the $i=1$ case yields the trivial representation of $G$. 
When $G$ is of type $A_1$, the $i=1$ case vanishes and the $i=0$ case yields the trivial representation.

The following three lemmas were used in the proof of Theorem~\ref{prop: alt cohom vanishing long}.
\begin{lemma} \label{lemma: functions on subregular cover}
Keep the notation of  
Theorem~\ref{prop: alt cohom vanishing long}.  Then
$$\coh^0 (\mathcal{B}, S^{n} \mathfrak{n}_{\delta}^*) \cong
\coh^0 (\mathcal{B}, S^{n} \mathfrak{n}_{\alpha}^*) \oplus \coh^0 (\mathcal{B}, S^{n-1} \mathfrak{n}_{\alpha}^* \otimes \varpi)$$
and 
$$\coh^i (\mathcal{B}, S^{n-1} \mathfrak{n}_{\alpha}^* \otimes \varpi) = 0$$
for all $i \geq 1$ and for all $n \geq 0$.
\end{lemma}

\begin{proof}
The equation
   \begin{equation}\label{equation: Koszul2}
0 \to S^{j-1} \mathfrak{n}_{\delta}^* \otimes \bk_{\alpha} \to S^j\mathfrak{n}_{\delta}^* \to
  S^j V^*  \to 0
  \end{equation}
  holds where $V \subset \mathfrak{n_{\delta}}$ is the sum of all root spaces except for the root spaces 
  for $\alpha$ and for $\delta$.
Since $\mathfrak{n}_{\delta}$ is $P_{\delta}$-stable and $\langle \alpha, \delta^\vee \rangle = -1$ , 
we have 
$\coh^i (\mathcal{B}, S^{j-1} \mathfrak{n}_{\delta}^* \otimes \alpha) =0$ for all $i$
by Lemma~\ref{lemma: past results cohom}.
Hence,
$\coh^i (\mathcal{B}, S^{j} \mathfrak{n}_{\delta}^*) \cong \coh^i (\mathcal{B}, S^{j} V^*)$ 
for all $i \geq 0$ and $j \geq 0$, using the long exact sequence in cohomology.

Next, consider the equation 
  \begin{equation}\label{equation: Koszul3}
0 \to S^{j-1} \mathfrak{n}_{\alpha}^* \otimes \bk_{\delta} \to S^j\mathfrak{n}_{\alpha}^* \to
  S^j V^*  \to 0.
  \end{equation}
Note that $\mathfrak{n}_{\alpha}$ is $P_{\alpha}$-stable and $\langle \delta, \alpha^\vee \rangle = -q$.
Hence, part (4) of Lemma~\ref{lemma: past results cohom} implies
$$\coh^i (\mathcal{B}, S^{j-1} \mathfrak{n}_{\alpha}^* \otimes \delta) \cong 
\coh^{i-1} (\mathcal{B}, S^{j-1} \mathfrak{n}_{\alpha}^* \otimes (q{-}1)\alpha{+}\delta),$$
The weight $(q{-}1)\alpha{+}\delta$ is $\varpi$.
Since 
$\coh^i (\mathcal{B}, S^{j} \mathfrak{n}_{\alpha}^*)$ and $\coh^i (\mathcal{B}, S^{j} \mathfrak{n}_{\delta}^*)$ 
always vanish by part (5) of Lemma~\ref{lemma: past results cohom},
the long exact sequence for \eqref{equation: Koszul3} implies the result.
\end{proof}

\begin{lemma}\label{lemma: neg to pos}
Keep the notation of 
Theorem~\ref{prop: alt cohom vanishing long}.  Then
$$\coh^i (\mathcal{B}, S^{n} \mathfrak{n}_{\alpha}^* \otimes (-\alpha-\delta))
\cong \coh^i (\mathcal{B}, S^{n-2} \mathfrak{n}_{\alpha}^* \otimes \varpi)$$
for all $i \geq 0$, $n \geq 0$.
\end{lemma}

\begin{proof}
Let $V$ be the subspace of $\mathfrak{n}_{\alpha}$ omitting the root spaces for $\delta$ and $\alpha+\delta$.
Then $V$ is $P_{\delta}$-stable.  
Set $\mathsf{M} = \mathfrak{n}_{\alpha}/V$, which is a $B$-module. 
Consider the Koszul resolution
  \begin{equation}\label{equation: Koszul4}
0 \to S^{j-2} \mathfrak{n}_{\alpha}^* \otimes \bk_{\alpha+2\delta} 
\to S^{j-1}\mathfrak{n}_{\alpha}^* \otimes \mathsf{M}^*
\to S^j\mathfrak{n}_{\alpha}^* \to
  S^j V^*  \to 0.
  \end{equation} 
Now $\mathsf{M}^*$ is a $B$-module that can be written as the tensor product of $B$-modules $\mathsf{M}'$ and $\bk_\lambda$ 
for 
$\lambda \in X^*(T)$ with $\langle \lambda, \alpha \rangle = -q+1$. 
In fact, $\mathsf{M}'$ is just the two-dimensional irreducible representation when restricted to the $SL_2$ Levi subgroup
for $\alpha$.  

Let $\mu = -\alpha{-}\delta$ and tensor \eqref{equation: Koszul4} with $\bk_\mu$, giving
  \begin{equation}\label{equation: Koszul4.1}
0 \to S^{j-2} \mathfrak{n}_{\alpha}^* \otimes \bk_{\delta} 
\to S^{j-1}\mathfrak{n}_{\alpha}^* \otimes \mathsf{M}^* \otimes \bk_\mu
\to S^j\mathfrak{n}_{\alpha}^* \otimes \bk_\mu \to
  S^j V^* \otimes \bk_\mu \to 0.
  \end{equation} 
Since $\langle \mu, \delta^{\vee} \rangle = -1$, 
the final non-zero term  in \eqref{equation: Koszul4.1} has vanishing cohomology on~$\mathcal B$ in all degrees
by part (1) of Lemma~\ref{lemma: past results cohom}.
Moreover, the second non-zero term $S^{j-1}\mathfrak{n}_{\alpha}^* \otimes \mathsf{M}' \otimes \bk_{\lambda+ \mu}$ 
also has vanishing cohomology in all degrees by part (1) of Lemma~\ref{lemma: past results cohom},
using the fact that 
$S^{j-1}\mathfrak{n}_{\alpha}^* \otimes \mathsf{M}'$ is a $P_{\alpha}$-module
and $\langle \lambda+\mu, \alpha^{\vee} \rangle = -1$.  The latter is a consequence of the fact
that $\langle \mu, \alpha^\vee \rangle = q-2$.

Breaking \eqref{equation: Koszul4.1} into two short exact sequences
and using their long exact sequences in cohomology now implies that 
$$\coh^{i+1} (\mathcal{B}, S^{n-2} \mathfrak{n}_{\alpha}^* \otimes \delta)
\cong \coh^i (\mathcal{B}, S^{n} \mathfrak{n}_{\alpha}^* \otimes \mu).$$
The result follows since
$\coh^{i+1} (\mathcal{B}, S^{n-2} \mathfrak{n}_{\alpha}^* \otimes \delta) \cong 
\coh^{i} (\mathcal{B}, S^{n-2} \mathfrak{n}_{\alpha}^* \otimes (q-1)\alpha+\delta)$
as in the proof of Lemma~\ref{lemma: functions on subregular cover}.
\end{proof}

\begin{lemma}\label{lemma: up to short dominant}
Keep the notation of Lemma~\ref{lemma: neg to pos}
and assume that $q=2$,
so that $\mu = -\varpi = -\alpha{-}\delta$.
As before, $\alpha_k$ is the unique \textup{(}short\textup{)} simple root such that $\theta^+ + \alpha_{k}$ is a root
and $h$ is the Coxeter number of $G$.
Then 
$$\coh^i (\mathcal{B}, S^{n} \mathfrak{n}_{\alpha}^* \otimes \varpi)
\cong
\coh^i (\mathcal{B}, S^{n-(\frac{h}{2}-2)} \mathfrak{n}_{\alpha_k}^* \otimes \theta^+)$$
and 
$$\coh^i (\mathcal{B}, S^{n-(\frac{h}{2}-2)} \mathfrak{n}_{\alpha}^* \otimes (-\varpi))
\cong
\coh^i (\mathcal{B}, S^{n} \mathfrak{n}_{\alpha_k}^* \otimes (-\theta^+))$$
for all $i \geq 0$, $n \geq 0$.
\end{lemma}

\begin{proof}
This is interpreting Theorem 2.2 in \cite{Johnson-Sommers} for $\Omega= \{ \alpha \}$
and observing that $m = \frac{h}{2}-2$ and $\Omega' = \{ \alpha_k \}$.   
The same proof
works if $\lambda$ and $\phi= \theta^+$ in that theorem are replaced by $-\phi$ and $-\lambda$, respectively.
\end{proof}

\subsection{Parabolic setting}

Let $P$ be the parabolic subgroup of $G$ stabilizing $\theta$. 

\begin{theorem}\label{prop: type A parabolic cohomology}
  Let $G$ be of type $A_r$ with $r \geq 2$.   Then
  \[
    \coh^i(\mathcal{B}, S^n \mathfrak{n}_P^* \otimes (-\theta)) \cong
    \begin{cases}
      \coh^0 (\mathcal{B}, S^{n-2} \mathfrak{n}_P^*) & \text{if } $i=0$;\\
      \coh^0 (\mathcal{B}, S^{n-1} \mathfrak{n}_Q^*) & \text{if } $i=1$;\\
      0 & \text{otherwise,}
    \end{cases}
  \]
  where $Q$ is the parabolic containing $B$ whose Levi factor corresponds to all simple roots except $\alpha_2$.
\end{theorem}

\begin{proof}

Let $\alpha_1, \dots, \alpha_r$ be the usual numbering of the simple roots of $G$.
Then $\langle -\theta, \alpha_r^\vee \rangle = -1$ and $\langle -\theta, \alpha_j^\vee \rangle = 0$  for 
$2 \leq j \leq r-1$.  
Since $\mathfrak{n}_P$ is stable for $P$, we 
can use Theorem 1.1 in \cite{Sommers:cohomology} relative to the $A_{r-1}$ subsystem for $\alpha_2, \dots, \alpha_{r}$ to deduce
\[
    \coh^i(\mathcal{B}, S^n \mathfrak{n}_P^* \otimes (-\theta )) \cong
    \coh^i (\mathcal{B}, S^{n-1} \mathfrak{n}_R^* \otimes (-\alpha_1))
\]
for all $i \geq 0$ and $n\geq 0$,
where $R$ is the parabolic subgroup with Levi factor whose simple roots are $\alpha_3, \dots, \alpha_{r}$.  

Next, as in the proof of Theorem~\ref{prop: alt cohom vanishing short}, 
we have a short exact sequence
$$0 \to S^{j-1}\mathfrak{n}_R^* \otimes \bk_{\alpha_1} \to S^j\mathfrak{n}_R^* \to S^j\mathfrak{n}_Q^*  \to 0,$$
which we can tensor with $\bk_{-\alpha_1}$ to get the short exact sequence
$$0 \to S^{j-1}\mathfrak{n}_R^*  \to S^j\mathfrak{n}_R^* \otimes \bk_{-\alpha_1}\to 
S^j\mathfrak{n}_Q^* \otimes \bk_{-\alpha_1} \to 0.$$
As in the proof of Theorem~\ref{prop: alt cohom vanishing short}, we have
$\coh^i (\mathcal{B}, S^{j} \mathfrak{n}_Q^* \otimes (-\alpha_1)) \cong
\coh^{i-1} (\mathcal{B}, S^{j} \mathfrak{n}_Q^*)$
and the latter cohomology vanishes except when $i=0$ by part (5) of Lemma~\ref{lemma: past results cohom}.
Similarly $\coh^{i} (\mathcal{B}, S^{j-1} \mathfrak{n}_R^*)$ vanishes except when $i=0$.
So the long exact sequence in cohomology yields 
\[
\coh^i (\mathcal{B}, S^{n-1} \mathfrak{n}_R^* \otimes (-\alpha_1)) \cong
 \begin{cases}
      \coh^0 (\mathcal{B}, S^{n-2} \mathfrak{n}_R^*) & \text{if } i=0;\\
      \coh^0 (\mathcal{B}, S^{n-1} \mathfrak{n}_Q^*) & \text{if } i=1;\\
      0 & \text{otherwise.}
    \end{cases}
\]
The result follow since 
 $\coh^i (\mathcal{B}, S^{n-2} \mathfrak{n}_R^*) 
 \cong 
 \coh^i (\mathcal{B},S^{n-2} \mathfrak{n}_P^*)$
 for all $i \geq 0$ by loc. cit.
\end{proof}

As discussed in, for example, \cite{kummar-lauritzen-thomsen}, there is no difference between computing
cohomologies of $P$-modules on $\mathcal{P}$ or $\mathcal{B}$. 
The theorem yields another proof of Theorem~\ref{prop: cohomology vanishing p}, which is the case when $n=1$ above.  In that case, the cohomology vanishes for $i=0$ and equals the trivial representation of $G$ when $i=1$.

There is a conjectural version of Theorem \ref{prop: type A parabolic cohomology} for $G$ outside of type $A$ and type~$C$ (so it holds for type $B_r$ when $r \geq 3$).  
In these types, there is a Levi subalgebra $\mathfrak{m}$ of  $\mathfrak{g}$ (or in type $G_2$, a pseudo-Levi subalgebra) of type $A_2$ consisting of long roots
of $\mathfrak{g}$.  
Although this also holds in type $A_r$ for $r \geq 2$, in these other types, the unique orbit $\mathcal O$ in $\mathfrak{g}$ whose intersection with $\mathfrak{m}$ is the regular nilpotent orbit in $\mathfrak{m}$ has a $2$-fold cover.  It is expected that the global sections in the conjecture 
compute the $G$-module of functions on this cover corresponding to the non-trivial local system on $\mathcal O$.
The following conjecture is a special case of one proposed in \cite{sommers:functions}.
\begin{conjecture}\label{prop: non A/C parabolic coh}
  Let $G$ be simple of type different from type $A$ and type $C$. 
  Then
  \[
    \coh^i(\mathcal{B}, S^n \mathfrak{n}_P^* \otimes (-\theta)) \cong
      \coh^i( \mathcal{B}, S^{n-4} \mathfrak{n}_P^* \otimes \theta),
  \]
  which vanishes for $i>0$.
\end{conjecture}

It seems possible that in type $C$, a similar statement holds, but with a different degree shift: 
  \[
    \coh^i(\mathcal{B}, S^n \mathfrak{n}_P^* \otimes (-\theta)) \cong
      \coh^i( \mathcal{B}, S^{n-2} \mathfrak{n}_P^* \otimes \theta).
  \]\label{maybe type C parabolic coh}
 
\section{Moduli stack in type \texorpdfstring{$A$}{A}}\label{s.moduli}

Fix $n\geq 4$. Let $G=\GL_n$ and $\grs=\mathfrak{gl}_n^\rs$. 
Let $T\subset G$ be a maximal torus.
Let $\Sch$ and $\Grpd$ denote the categories of schemes over $\bk$ and groupoids, respectively. 
\begin{definition}\label{def:moduli.stack} 
Define the moduli pseudofunctor $\fM_n:\Sch^{\mathrm{op}}\to \Grpd$ by
\[\fM_n(U)=\left\{
 \begin{tabular}{c}
\mbox{families of regular semisimple Hessenberg varieties $\cX\to U$} \\[2pt]
\mbox{whose geometric fibers are isomorphic to $X(s)$ for some $s\in \grs$}
\end{tabular}
\right\}\]
where isomorphisms of $\cX_1,\cX_2\in \fM_n(U)$ are the isomorphisms of families over $U$.
\end{definition}

In this section, we prove the following. 
We use the group actions defined in Subsection~\ref{ss:group.actions}. 

\begin{theorem}\label{thm:moduli.stack}
	The moduli pseudofunctor $\fM_n$ is an algebraic stack represented by the quotient stack
 \beq\label{eq:isom.quotient.stacks}[\grs/\cG]\cong [\PP(\pgl_n)^{\rs}/\Aut\Fl_n]\cong [M_{0,n+1}/((T/\mathbb{G}_m)\rtimes (S_n\times \mu_2))]\eeq
		where $(T/\mathbb{G}_m)\rtimes\mu_2$ acts trivially on $M_{0,n+1}$.
		
		In particular, the rigidification of $\fM_n$ along $(T/\mathbb{G}_m)\rtimes \mu_2$ is the Deligne--Mumford stack $[M_{0,n+1}/S_n]$, and the good moduli space of $\fM_n$ is $M_{0,n+1}/S_n$.
\end{theorem}

We refer the reader to \cite{olsson}, \cite{acv,romagny} and \cite{alper} for references for stacks, their rigidification and good moduli spaces respectively.

\medskip

\subsection{Moduli stack of smooth Fano varieties}
\begin{definition}
	Let $\fMFano_d$ denote the moduli pseudofunctor which assigns to a $\bk$-scheme $U$ the groupoid of smooth proper families $\cX\to U$ of relative dimension~$d$ whose relatively anticanonical line bundle $\omega_{\cX/U}^{-1}$ is ample.
\end{definition}

The following fact is well-known to experts.
\begin{theorem}\label{thm:MFano}
	The moduli pseudofunctor $\fMFano_d$ is represented by a smooth algebraic stack of finite type.
\end{theorem}

\begin{proof}
	We sketch the proof. For smooth proper families $\cX\to U$ of Fano varieties of dimension $d$ and a positive integer $m>0$, the pushforwards of $\omega_{\cX/U}^{-m}$ to $U$ are locally free of locally constant rank and commute with arbitrary base changes, and $\omega_{\cX/U}^{-m}$ have no higher direct images, by the Kodaira vanishing theorem. Moreover, the condition of $\omega_{\cX/U}^{-m}$ being very ample is open, as it is equivalent to being strongly ample in this case (\cite[\S8.38]{kollar-modulibook}). In particular, due to the boundedness of smooth Fano varieties of a given dimension (\cite{kollar-miyaoka-mori}), one can show that there exists $m>0$ such that $\omega_X^{-m}$ is very ample for every smooth Fano variety $X$ of dimension $d$. 
	
	Using a similar argument as in \cite[\S8.4.3]{olsson}, $\fMFano_d$ is the finite union of quotient stacks of locally closed subschemes in the Hilbert schemes of the projective spaces by the actions of the projective linear groups. This shows that $\fMFano_d$ is an algebraic stack of finite type. The smoothness follows from the Nakano vanishing theorem, which implies that deformations of smooth Fano varieties are unobstructed.
\end{proof}

\medskip

Consider the family of regular semisimple Hessenberg varieties 
\beq \label{eq:fam.over.grs}\{(s,F_\bullet)\in \grs \times \Fl_n:sF_1\subset F_{n-1}\}\lra \grs\eeq
over $\grs$, whose fiber over $s\in \grs$ is $X(s)$ in a fixed $\Fl_n$. 
The map \eqref{eq:fam.over.grs} is equivariant with respect to the actions of $\cG=\Aut \Fl_n \times \Aff$ on $\grs \times \Fl_n$ and $\grs$ (see Subsection~\ref{ss:group.actions}), where $\Aff$ acts trivially on $\Fl_n$. Moreover, there is a natural strong polarization given by the relative anticanonical line bundle of \eqref{eq:fam.over.grs}.

\begin{lemma}\label{lem:XFano}
	For $s\in \grs$, $X(s)$ is a smooth Fano variety of dimension $d(n)=\frac{n(n-1)}{2}-1$. In particular, \eqref{eq:fam.over.grs} is an element of $\fMFano_{d(n)}(\grs)$.
\end{lemma}
\begin{proof} 
	By the adjunction formula,  the anticanonical line bundle $\omega_{X(s)}^{-1}$ is the restriction of $L(2\rho)\otimes (\cL_1\otimes \cL_n)^\vee\cong L(2\rho-e_1+e_n)$ to $X(s)$. Since $2\rho-e_1+e_n$ is regular dominant, $\omega_{X(s)}^{-1}$ is ample.
\end{proof}
\begin{remark}
	For a complete classification of Fano and weak Fano regular semisimple Hessenberg varieties in type $A$, see \cite{abe-fujita-zeng}.
\end{remark}

\subsection{Moduli stacks of \texorpdfstring{$X(s)$}{X(s)} and \texorpdfstring{$Y(s)$}{Y(s)}} 
We prove Theorem~\ref{thm:moduli.stack}. 
\medskip

Using our main theorems in type $A$ and Theorem~\ref{t_comp}, we obtain the following. 
\begin{theorem}\label{thm:moduli.stack.open}
	Let $n\geq 4$. The morphism 
	\beq\label{eq:immersion.Fano}[\grs/\cG]\lra \fMFano_{d(n)}\eeq
	induced by \eqref{eq:fam.over.grs} is an open immersion. Moreover, for every $\bk$-scheme $U$ and a family in $\fM_n(U)$, the induced morphism $U\to \fMFano_{d(n)}$ factors through \eqref{eq:immersion.Fano}. In particular, $\fM_n$ is represented by the quotient stack $[\grs/\cG]$.
\end{theorem}
\begin{proof}
	For the first assertion, we prove that \eqref{eq:immersion.Fano} is representable, injective on closed points, and smooth. The first two properties are direct consequences of Corollary~\ref{cor:interpret.AX1} (\cite[Lemma~2.3.9]{abramovich-hassett}). 
	
	For the smoothness of \eqref{eq:immersion.Fano}, it suffices to show that the Kodaira--Spencer map $\mathfrak{g}\cong \tang_{s}\grs \to \coh^1(X(s),\tang{X(s)})$ induced by \eqref{eq:fam.over.grs} is surjective. This map can be identified with the composition
	\beq\label{eq:differential}\coh^0(\Fl_n,\cL_1\otimes \cL_n)\lra \coh^0(X(s),\cL_1\otimes \cL_n)\lra \coh^1(X(s),\tang X(s))\eeq
	of two maps which are parts of the long exact sequences induced by
	\[\begin{split}
		0\lra \cO_{\Fl_n}\lra \cL_1\otimes \cL_n\lra \cL_1\otimes \cL_n|_{X(s)}\lra 0,&\\
		0\lra \tang X(s)\lra \tang \Fl_n|_{X(s)}\lra \cL_1\otimes\cL_n|_{X(s)}\lra 0,&
	\end{split}\]
	respectively, where $\cL(\theta)=\cL_1\otimes \cL_n$ on $X(s)$ is the normal bundle $\mathscr{N}$ of $X(s)\subset \Fl_n$. The first map in \eqref{eq:differential} is surjective since $\coh^1(\Fl_n,\cO_{\Fl_n})=0$, and the second map in \eqref{eq:differential} is surjective for $n\geq 4$ by Theorem~\ref{t_comp}.

	The second assertion on factorization follows from the first since it is true on closed points. 
	The last assertion follows from the second assertion, which implies that $\fM_n$ and $[\grs/\cG]$
	have the same universal property.
\end{proof}

\begin{proof}[Proof of Theorem~\ref{thm:moduli.stack}]
	It remains to prove the isomorphisms in \eqref{eq:isom.quotient.stacks} and the assertions on the rigidification and the good moduli of $\fM_n$. The first two stacks in~\eqref{eq:isom.quotient.stacks} are isomorphic since $\Aff$ freely acts on $\grs$ and $\PP(\pgl_n)^{\rs}=\grs/\Aff$. Moreover, the $\Aff$-equivariant isomorphisms \eqref{eq:isom.quotients2} yield isomorphisms of stacks
	\[\begin{split}
		[\grs/\cG]&\cong [\mathfrak{t}^\rs/(\Aff\times ((N(T)/\mathbb{G}_m)\rtimes \mu_2))]\\
		&\cong [(\mathbb{A}^n-\Delta)/(\Aff\times ((N(T)/\mathbb{G}_m)\rtimes \mu_2))]\\
		&\cong [M_{0,n+1}/((N(T)/\mathbb{G}_m)\rtimes \mu_2)].
	\end{split}
	\]
	The first isomorphism is induced by the inclusion $\mathfrak{t}^{\rs}\hookrightarrow \grs$ which is equivariant with respect to the adjoint actions of $N(T)/\mathbb{G}_m$ and $G/\mathbb{G}_m\cong \PGL_n$ together with the homomorphism $N(T)/\mathbb{G}_m\hookrightarrow G/\mathbb{G}_m$. The natural isomorphism $\mathfrak{t}^{\rs}\cong \mathbb{A}^n-\Delta$ for the diagonal $\Delta$ defined in~\eqref{eq:M0n+1.Aff} induces the second isomorphism. The last isomorphism follows from \eqref{eq:M0n+1.Aff}, since the $\Aff$-action in \eqref{eq:M0n+1.Aff} is free. 
 
	Moreover, the action of $(T/\mathbb{G}_m)\rtimes(S_n\times\mu_2)$ on $M_{0,n+1}$ factors through the quotient map $(T/\mathbb{G}_m)\rtimes(S_n\times\mu_2)\to S_n$, since $(T/\mathbb{G}_m)\rtimes\mu_2$ acts trivially on $M_{0,n+1}$. This induces the desired rigidification map $\fM_n\to [M_{0,n+1}/S_n]$. 
	
	The rigidification map is a good moduli space morphism, since $(T/\mathbb{G}_m)\rtimes \mu_2$ is reductive.
	Since the map $\fM_n\to M_{0,n+1}/S_n$ is the composition of the rigidification and the coarse moduli map of the (tame) Deligne--Mumford
    stack $[M_{0,n+1}/S_n]$, it is a good moduli space.
\end{proof}

The same arguments apply to $Y(s)$. For $s\in \grs$, $Y(s)$ is a smooth Fano variety of dimension $2n-4$, since it is the transverse intersection of two $(1,1)$-divisors in $\PP\times \PP^\vee$ cut out by the conditions $F_1\subset F_{n-1}$ and $sF_1\subset F_{n-1}$ respectively.

We define $\fM^Y_n$ analogously to $\fM_n$ as the moduli pseudofunctor parametrizing families whose geometric fibers are isomorphic to $Y(s)$ for some $s\in\grs$.
\begin{theorem}\label{thm:moduli.stack.Y}
	The moduli pseudofunctor $\fM^Y_n$ is an algebraic stack represented by the quotient stack
	\beq\label{eq:isom.quotient.stacks.Y}
	\begin{split}
		[\tgrs/\tcG]&\cong [\Gr_2(\coh^0(\PP\times\PP^\vee,\cO(1,1)))^{\rs}/\Aut\PP\times\PP^\vee]\\
		&\cong  [M_{0,n}/((T/\mathbb{G}_m)\rtimes (S_n\times \mu_2))].
	\end{split}
	\eeq
		In particular, its rigidification along $(T/\mathbb{G}_m)\rtimes\mu_2$ is the Deligne--Mumford stack $[M_{0,n}/S_n]$, and its good moduli space is $M_{0,n}/S_n$.
\end{theorem}

\begin{proof}
The proof of that the moduli pseudofunctor  is represented by $[\tgrs/\tcG]$ is essentially the same as that of Theorem~\ref{thm:moduli.stack.open}, by using Theorem~\ref{t_comp p} instead of Theorem~\ref{t_comp}, and by replacing $X(s)$, $d(n)$ and the family \eqref{eq:fam.over.grs} over $\grs$ by $Y(s)$, $2n-4$ and the family
\beq \label{eq:fam.over.Utilde}\{([s:u],F_1,F_{n-1})\in \tgrs\times \PP\times\PP^\vee :sF_1, uF_1\subset F_{n-1}\}\lra \tgrs\eeq
over $\tgrs$ respectively. Every fiber of \eqref{eq:fam.over.Utilde} is isomorphic to $Y(s')$ for some $s'\in \grs$, since by definition $[s:u]\in \tgrs$ lies in the $\widetilde{\cG}$-orbit of $\grs\times\{\id\}$ and \eqref{eq:fam.over.Utilde} is $\tcG$-equivariant where the last factor $\PGL_2$ of $\tcG$ acts trivially on $\PP\times\PP^\vee$.

The first isomorphism in \eqref{eq:isom.quotient.stacks.Y} 
holds because $\Gr_2(\coh^0(\PP\times\PP^\vee,\cO(1,1)))^{\rs}$ is the quotient of $\PP(\fg\oplus\fg)^{\rs}$ by the free $\PGL_2$-action.
Moreover, we have isomorphisms
\beq\label{eq:isom.quotient.stacks.Y2}
\begin{split}
	[\tgrs/\tcG]&\cong [\PP(\mathfrak{t}\oplus\mathfrak{t})^\rs/(\PGL_2\times ((N(T)/\mathbb{G}_m)\rtimes \mu_2))]\\
	&\cong [((\PP^1)^n-\Delta)/(\PGL_2\times ((N(T)/\mathbb{G}_m)\rtimes \mu_2))]\\
	&\cong [M_{0,n}/((N(T)/\mathbb{G}_m)\rtimes \mu_2)]
\end{split}\eeq
of stacks, where
the first isomorphism in \eqref{eq:isom.quotient.stacks.Y2} corresponds to (simultaneous) diagonalization. There is a natural isomorphism $\PP(\mathfrak{t}\oplus\mathfrak{t})^\rs\cong (\PP^1)^n-\Delta$ sending the pair of $i$-th diagonal entries to an element of the $i$-th $\PP^1$ for $1\leq i\leq n$.
The second isomorphism in \eqref{eq:isom.quotient.stacks.Y2} follows from this. 
The last isomorphism follows from 
the definition of $M_{0,n}$ in \eqref{eq:M0n.def} since the $\PGL_2$-action is free.
These clarify the isomorphisms in \eqref{eq:isom.quotient.stacks.Y}. 
The proof for the rest assertions on the rigidification and the good moduli space is the same as in the proof of Theorem~\ref{thm:moduli.stack}.
\end{proof}

\begin{corollary}
	There exists a representable morphism from the moduli stack $\fM_n$ of $X(s)$ to the moduli stack of $Y(s)$ induced by either \eqref{eq:inclusion.pgln} or the \textup{(}$S_n$-equivariant\textup{)} forgetful morphism $M_{0,n+1}\to M_{0,n}$. This morphism sends the isomorphism class of $X(s)$ to the isomorphism class of $Y(s)$ on the level of closed points.
\end{corollary}
\begin{proof}
	This follows from Theorems~\ref{thm:moduli.stack} and~\ref{thm:moduli.stack.Y} and the fact that 
	the morphism $\PP\coh^0(\Fl_n,\cL_1\otimes\cL_n)^{\rs}\hookrightarrow \Gr_2(\coh^0(\PP\times\PP^\vee,\cO(1,1)))^{\rs}$ in~\eqref{eq:inclusion.pgln} is equivariant under the actions of $\Aut \Fl_n$ and $\Aut\PP\times\PP^\vee$ compatibly with the group homomorphism $\Aut \Fl_n\hookrightarrow \Aut\PP\times\PP^\vee$ in \eqref{eq:aut.ext.Fl}, which is injective. 
	Similarly, the forgetful morphism $M_{0,n+1}\to M_{0,n}$ which forgets the last marking is $S_n$-equivariant.
\end{proof}

\subsection{GIT compactifications}\label{ss:compact}

    One of the simplest $S_n$-equivariant compactifications of $M_{0,n+1}$ is given by the projective space $\PP^{n-2}$ (cf.~\cite{kapranov-jag,hassett}). The $S_n$-action on $\PP^{n-2}$ is 
    characterized by the fact that any permutation of $n$ points in general position extends uniquely to an automorphism of $\PP^{n-2}$.
	
	By the invariant theory of the adjoint action of $G$ on $\fg$, the GIT quotient
	\[\PP\coh^0(\Fl_n,\cL_1\otimes\cL_n)/\!\!/\Aut\Fl_n\;\cong\;\PP(\pgl_n)/\!\!/\SL_n\]
	is canonically isomorphic to $\PP^{n-2}/S_n \cong \PP(2,3,\dots,n)$
	(cf.~\cite[Example~6.4]{dolgachev-invariant}).
	Here $\PP(2,3,\dots,n)$ denotes the weighted projective space.
	Under this identification, a nonzero (traceless) matrix is sent to its characteristic polynomial.
	This provides a natural compactification of the moduli space $M_{0,n+1}/S_n$ of $\fM_n$.
	For other $S_n$-equivariant compactifications of $M_{0,n+1}$, see \cite{Knu,Kap,hassett,choi-kiem-lee}.

	Similarly, the GIT quotient
	\[\Gr_2(\coh^0(\PP\times \PP^\vee, \cO(1,1)))/\!\!/\Aut(\PP\times \PP^\vee)\]
	gives a natural projective GIT model for $M_{0,n}/S_n$, as the moduli space of pencils of $(1,1)$-divisors in $\PP\times \PP^\vee$.
	
	On the other hand, it is well known that the GIT quotient $\PP\Sym^n(\bk^2)^\vee/\!\!/\SL_2$, parametrizing binary forms (cf.~\cite{kirwan}), gives a classical compactification of $M_{0,n}/S_n$, with $M_{0,n}/S_n$ identified as the open locus of forms with distinct roots.
	
	The following proposition shows that these two GIT compactifications coincide.
	\begin{proposition}
		The following GIT quotients are naturally isomorphic:
		\beq\label{eq:GIT.Y2}
		\begin{split}
			&\Gr_2(\coh^0(\PP\times \PP^\vee, \cO(1,1)))/\!\!/\Aut\PP\times \PP^\vee\\
			&\cong\PP(\End(\bk^n)\oplus \End(\bk^n))/\!\!/(H\times\SL_2) 
			\cong \PP\Sym^n(\bk^2)^\vee/\!\!/ \SL_2
		\end{split}\eeq
		where $H=\SL_n\times \SL_n$. Here the action of $H\times\SL_2$  
		is defined as in Subsection~\ref{ss:group.actions.Y}.
		Moreover, the composition of the isomorphisms restricts to the identity on $M_{0,n}/S_n$.
	\end{proposition}
\begin{proof}
	Let $V=H^0(\PP\times\PP^\vee,\cO(1,1))$.
	First note that the involution exchanging the two factors of $\PP\times\PP^\vee$ acts trivially on the GIT quotient $\Gr_2(V)/\!\!/H$.
	
	Let $U\subset \PP(V\oplus V)$ be the locus of linearly independent pairs.
	Then the natural projection $U\to \Gr_2(V)$ is a $\PGL_2$-torsor.
	Moreover, $U$ contains the GIT semistable locus.
	Indeed, if $(A,B)$ is proportional then, after the $\SL_2$-action, it is of the form $(A,0)$,
	which is unstable by the Hilbert--Mumford criterion \cite[Theorem~2.1]{git} for the one-parameter subgroup $t\mapsto \mathrm{diag}(t,t^{-1})\in \SL_2$.
	Hence, we have
	\[\PP(V\oplus V)/\!\!/(H\times \SL_2)=U/\!\!/(H\times \SL_2).\]

	Since the $\PGL_2$-action on $U$ commutes with $H$ and $U\to \Gr_2(V)$ is a principal $\PGL_2$-bundle,
	we may take the GIT quotient first with respect to $\SL_2$
	and then with respect to $H$, which yields
	\[U/\!\!/(H\times \SL_2)\cong \Gr_2(V)/\!\!/H.\]
	This proves the first isomorphism in \eqref{eq:GIT.Y2}, since $\Gr_2(V)/\!\!/\Aut\PP\times\PP^\vee=\Gr_2(V)/\!\!/H$. 
	
	For the second isomorphism, consider the natural action of $H\times\SL_2$ on $V\oplus V=\End(\bk^n)\oplus\End(\bk^n)$.
	For $(A,B)\in V\oplus V$, we set
	\[D(A,B)(u,v)=\det(uA+vB)\in \Sym^n(\bk^2)^\vee.\]
	It is well known that the invariant ring for the $H$-action on $V\oplus V$ is a polynomial algebra generated by the coefficients of $\det(uA+vB)$;
	see \cite[Example~10.5.2]{derksen-weyman-book}.
	
	Taking $\SL_2$-invariant parts then induces an isomorphism 
	\[\bk[V\oplus V]^{H\times\SL_2}\cong \bk[\Sym^n(\bk^2)^\vee]^{\SL_2}\]
	of graded invariant rings, and taking $\Proj$ gives the desired isomorphism.
	
	In each GIT quotient, $Y(s)$ corresponds to the point $\langle \id, s\rangle$, $(\id,s)$ and $D(\id,s)$, respectively, where $D(\id,s)$ is the homogenized characteristic polynomial of $s$ whose roots are precisely the eigenvalues of $s$. 
	Hence the composition of the isomorphisms in \eqref{eq:GIT.Y2} restricts to the identity on $M_{0,n}/S_n$.
\end{proof}


\def\cprime{$'$}


\end{document}